\RequirePackage{ifpdf}
\ifpdf
\documentclass[pdftex]{sigma}
\else
\documentclass{sigma}
\fi

\usepackage{amsmath}
\usepackage{amsfonts}
\usepackage{amssymb}
\usepackage[mathscr]{eucal}
\usepackage{graphicx}
\usepackage{hyperref}
\usepackage[new]{old-arrows}
\usepackage{enumitem}
\usepackage{comment}
\usepackage{todonotes}
\usepackage[all]{xy}

\makeatletter
\newcommand{\Xbar}{{\mathchoice
     {\smash@bar\textfont\displaystyle{0.55}{2.5}\mathscr{X}}
     {\smash@bar\textfont\textstyle{0.55}{2.5}\mathscr{X}}
     {\smash@bar\scriptfont\scriptstyle{0.55}{2.5}\mathscr{X}}
     {\smash@bar\scriptscriptfont\scriptscriptstyle{0.55}{2.5}\mathscr{X}}
          }}
\newcommand{\smash@bar}[4]{
     \smash{\rlap{\raisebox{-#3\fontdimen5#10}{$\m@th#2\mkern#4mu\mathchar'26$}}}          }
\makeatother
\newcommand{\X}[1]{\mathscr{\Xbar}_{#1}}

\newcommand{\A}{\mathscr{A}}
\newcommand{\D}{\mathscr{D}}
\newcommand{\K}{\mathscr{K}}
\newcommand{\LL}{\mathscr{L}}
\newcommand{\R}[1]{\mathbb{R}^{#1}}
\newcommand{\T}{\mathsf{T}}
\newcommand{\PP}{\mathscr{P}}

\newcommand{\dd}{\mathrm{d}}

\newcommand{\pr}{\mathrm{pr}}
\newcommand{\dlie}[1]{\mathrm{L}_{#1}}
\newcommand{\cSch}[1]{[\hspace{-0.065cm}[ #1 ]\hspace{-0.065cm}]}
\newcommand{\Cinf}[1]{\mathbf{\mathit{C}}^{\infty}_{#1}}
\newcommand{\ec}[1]{\mbox{$#1$}}

\newcommand{\fiteq}[1]{\resizebox{\textwidth}{!}{\text{$\displaystyle #1$}}}

\newcommand{\orcid}[1]{\href{https://orcid.org/#1}{\includegraphics[width=0.02\textwidth]{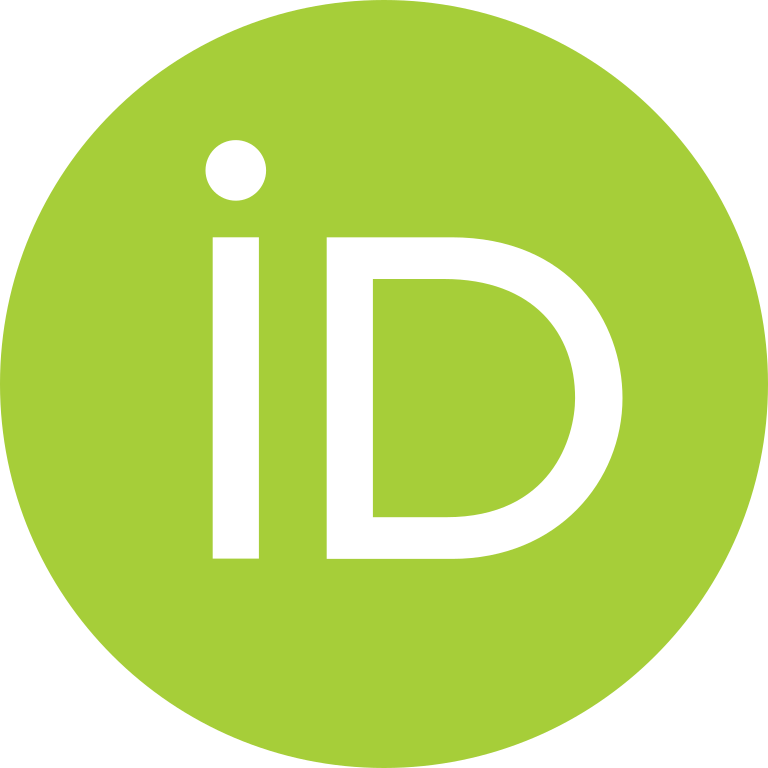}}}

\usepackage[symbol]{footmisc}

\usepackage{amsthm}
\usepackage{thmtools}
\declaretheorem[name=Example,style=definition, sibling=theorem, qed=$\triangleleft$]{examplex}

\numberwithin{equation}{section}

\begin{document}

\renewcommand{\PaperNumber}{***}

\thispagestyle{empty}

\ArticleName{Poisson Structures on Trivial Extension Algebras}
\ShortArticleName{Poisson Structures on Trivial Extension Algebras}

\Author{D. Garc\'ia-Beltr\'an~${}^{1^\dag}$ \orcid{0000-0001-9149-8200},\footnote{\textbf{Acknowledgements:} we are very grateful to an anonymous referee for this helpful comments and observations, which improve the presentation of the manuscript. DGB thanks to the National Council of Science and Technology (CONACyT) for a research fellowship held during the work on the manuscript.}
J. C. Ru\'iz-Pantale\'on~${}^{2^\ddag}$ \orcid{0000-0001-7642-8298}
and Yu. Vorobiev~${}^{3^\ddag}$ \orcid{0000-0002-1076-4544}}
\AuthorNameForHeading{Garc\'ia-Beltr\'an, Ru\'iz-Pantale\'on, Vorobiev}

\Address{$^{\dag}$~CONACyT Research--Fellow, Departamento de Matem\'aticas, Universidad de Sonora, M\'exico}
\Address{$^{\ddag}$~Departamento de Matem\'aticas, Universidad de Sonora, M\'exico}
\EmailD{$^{1}$\href{mailto:email@address}{dennise.garcia\,@unison.mx},
        $^{2}$\href{mailto:email@address}{jose.ruiz\,@unison.mx},
        $^{3}$\href{mailto:email@address}{yurimv\,@guaymas.uson.mx}}

\Abstract{We present a class of Poisson	structures on trivial extension algebras which generalizes some known structures induced by Poisson modules. We show that there exists a one--to--one correspondence between such a class of Poisson structures and some data involving (not necessarily flat) contravariant derivatives, and then we give a formulation of this result in terms of Lie algebroids. Some properties of the first Poisson cohomology are presented. Examples coming from Poisson modules and Poisson submanifolds are given.}


\Classification{17B63, 17B60, 53D17, 16W25}

    \section{Introduction}\label{sec:introdution}

In this paper, we introduce and study a class of Poisson algebra structures on a given trivial extension algebra which generalizes the well--known Poisson structures induced by Poisson modules \cite{ReVoWe96, Bur01, Car03, Zhu19, ZhuGa20}. Another important origin of such structures is given by the so--called infinitesimal Poisson algebras on the space of fiberwise affine functions on the normal bundle of a Poisson submanifold \cite{Mar12,RuGaVo20}. Our goal is to develop further the results of \cite{RuGaVo20} and give a uniform approach for the study of this class of Poisson algebras of geometric nature in a general algebraic framework which is more convenient for the treatment of some cohomological aspects. In particular, we obtain some generalizations of known results on first Poisson cohomology associated to Poisson modules \cite{Zhu19}.

A \emph{trivial extension algebra} \ec{P_{0}\ltimes P_{1}} of a commutative \ec{R}--algebra \ec{P_{0}} by a \ec{P_{0}}--(bi)module \ec{P_{1}} is the (commutative) $R$--algebra consisting of the \ec{R}--module \ec{P_{0} \oplus P_{1}} with multiplication defined by
    \begin{equation}\label{EcPproduct}
        (f \oplus \eta) \cdot (g \oplus \xi) := fg \oplus (f\xi + g\eta), \quad f,g \in P_{0},\ \eta,\xi \in P_{1}.
    \end{equation}
Here, $R$ is a commutative ring with unit. As is well--known, if \ec{P_{0}} is endowed with a Poisson algebra structure, then every Poisson module \ec{P_1} over \ec{P_{0}} induces a Poisson structure on \ec{P_{0}\ltimes P_{1}}. In this case, the natural projection \ec{P_{0}\ltimes P_{1} \to P_{0}} is a Poisson morphism. This property is our starting point for constructing a more general class of Poisson structures on trivial extension algebras.

Given a Poisson algebra \ec{P_{0}}, for each \ec{P_{0}}--module \ec{P_{1}}, we consider Poisson structures on \ec{P_{0}\ltimes P_{1}} such that the natural projection \ec{P_{0}\ltimes P_{1} \to P_{0}} is a Poisson morphism, which we call \emph{admissible Poisson structures}. The corresponding Poisson algebras are called \emph{admissible Poisson algebras}. In particular, every admissible Poisson structure is compatible with multiplication \eqref{EcPproduct} by the Leibniz rule.

Our approach is based on an algebraic version of the notion of contravariant derivative \cite{Vais91,Vais94}. We show that there exists a one--to--one correspondence between admissible Poisson algebras  and some data \ec{([\,,\,]_{1},\D,\K)} consisting of a Lie bracket \ec{[\,,\,]_{1}} on \ec{P_{1}}, a (not necessarily flat) contravariant derivative $\D$ and a skew--symmetric mapping \ec{\K:\Omega^{1}_{P_0} \times \Omega^{1}_{P_0}\to P_{1}} satisfying some compatibility conditions representing a factorization of the Jacobi identity (Theorem \ref{teo:correspondenceGPA-PT-LA}). The data \ec{([\,,\,]_{1},\D,\K)} is called a \emph{Poisson triple} of \ec{P_{0}\ltimes P_{1}}. Therefore, Poisson triples parameterize admissible Poisson structures and allow us to formulate the main results. Another important ingredient is the notion of Lie algebroid in an algebraic framework \cite{Kos90,Mac95}. We show that there exists a one--to--one correspondence between admissible Poisson algebras and a class of Lie algebroids structures on \ec{\Omega^{1}_{P_0}\oplus P_1} satisfying that the natural projection \ec{\Omega^{1}_{P_0} \oplus P_1 \to \Omega^{1}_{P_{0}}} is a Lie algebroid morphism (Theorem \ref{teo:correspondenceGPA-PT-LA}). Here, \ec{(\Omega^{1}_{P_0},\cSch{\,,\,}_{P_{0}},\varrho_{P_{0}})} is the Lie algebroid of the Poisson algebra \ec{P_{0}} \cite{Kos90, GaVaVo12}.

    In the framework of our approach, the family of admissible Poisson algebras associated to Poisson modules are parameterized by Poisson triples of the form \ec{(0,\D,0)}, where $\D$ is a flat contravariant derivative (Corollary \ref{cor:APAbyPM}). A natural deformation of this family is related to the abelian case, that is, to admissible Poisson algebras defined by Poisson triples of the form \ec{(0,\D,\K)}, where $\K$ is a 2-cocycle of a coboundary operator associated to a flat contravariant derivative $\D$. Then, the cohomology class of $\K$ controls the non-triviality of the corresponding deformation (Theorem \ref{teo:PisoPlambda}).

    To formulate a Poisson equivalence criterion, we describe a family of Poisson algebra transformations that preserve the admissibility property (Definition \ref{def:gauget}). In particular, a way of constructing new one-parametric admissible Poisson structures from given ones is provided (Theorem \ref{teo:NewAdmissible}).

    We are also interested in the study of the first cohomology of admissible Poisson algebras, which is a quite difficult task, in general (see, for example \cite{Zhu19,ZhuGa20}). First, we describe the Casimir elements and the Poisson and Hamiltonian derivations of a given admissible Poisson algebra $\PP=(P_{0}\ltimes P_{1}, \{\,,\,\})$ in terms of the corresponding Poisson triple (Propositions \ref{prop:Casim}, \ref{prop:PoissDer} and \ref{prop:hamalgebra}). Then,
taking into account a cochain complex \ec{(\Gamma^{\ast}_{\PP},\partial_{\D})} associated to \ec{\PP}, we deduce some information on the first Poisson cohomology \ec{\mathscr{H}^{1}(\PP)} of $\PP$. One of our main results is the following (Theorem \ref{teo:RestrictedCoho}):

\paragraph{Claim.}
\emph{There exists a short exact sequence}
    \begin{equation*}
        0 \,\longrightarrow\, \frac{\mathrm{H}_{\partial_{\D}}^{1}\big( \Gamma^{\ast}_{\PP} \big)}{\ker{\mathrm{J}}}
        \,\longrightarrow\, \mathscr{H}^{1}_{\mathrm{rest}}(\PP)
        \,\longrightarrow\, \frac{\mathfrak{M}(\PP)}{\mathscr{C}(\PP) + \mathrm{Inn}(P_{1},[\,,\,]_{1})}
        \,\longrightarrow\, 0.
    \end{equation*}

Here, \ec{\mathscr{H}^{1}_{\mathrm{rest}}(\PP) \subseteq \mathscr{H}^{1}(\PP)} is the so--called \emph{restricted first Poisson cohomology of \ec{\PP}} and \ec{\mathrm{J}} is an $R$--linear mapping from the first cohomology group \ec{\mathrm{H}_{\partial_{\D}}^{1}(\Gamma^{\ast}_{\PP})} of \ec{(\Gamma^{\ast}_{\PP},\partial_{\D})} to \ec{\mathscr{H}^{1}(\PP)}. The sets \ec{\mathfrak{M}(\PP)}, \ec{\mathscr{C}(\PP)} and \ec{\mathrm{Inn}(P_{1},[\,,\,]_{1})} are special (sub)modules of generalized derivations of \ec{P_{1}}.

Using the claim above, we deduce some consequences in particular cases. For example, if the Lie algebra \ec{(P_{1},[\,,\,]_{1})} is perfect and centerless, then we have (Theorem \ref{teo:H1MCI})
\begin{equation}\label{eq:H1-Intro}
     \mathscr{H}^{1}(\PP) \simeq
        \frac{\mathfrak{M}(\PP)}{\mathscr{C}(\PP) + \mathrm{Inn}(P_{1},[\,,\,]_{1})}.
\end{equation}
Furthermore, if additionally the first Poisson cohomology of the given Poisson algebra \ec{P_{0}} is trivial, then \ec{\mathscr{H}^{1}(\PP)} can be characterized in terms of an ideal \ec{\mathfrak{M}_{0}(\PP)} of \ec{\mathfrak{M}(\PP)} and a submodule \ec{\mathscr{C}_{0}(\PP)} of \ec{\mathfrak{M}_{0}(\PP)}  (Theorem \ref{teo:H1M0}),
 \begin{equation}\label{eq:H1Particular-Intro}
        \mathscr{H}^{1}(\PP) \simeq \frac{\mathfrak{M}_{0}(\PP)}{\mathscr{C}_{0}(\PP)+\mathrm{Inn}(P_{1},[\,,\,]_{1})}.
    \end{equation}

On the other hand, in the case when the Lie algebra \ec{(P_{1},[\,,\,]_{1})} is abelian, there exists an exact sequence
 \begin{equation*}
        0 \,\longrightarrow\, \mathscr{H}^{1}_{\mathrm{rest}}(\PP)
        \,\longrightarrow\, \mathscr{H}^{1}(\PP)
        \,\longrightarrow\, \mathscr{E}(\PP),
    \end{equation*}
where \ec{\mathscr{E}(\PP)} is a family of \ec{P_{0}}--linear mappings from $P_{1}$ to \ec{P_{0}} (Theorem \ref{teo:secAbelian}). We show that if the cohomology class of the $2$--cocycle \ec{\K} is trivial, then this exact sequence is short and splits (Theorem \ref{teo:SplitDeltaK}). Consequently,
\begin{equation}\label{eq:Rest1Split-Intro}
        \mathscr{H}^{1}(\PP) \simeq
        \mathrm{H}_{\partial_{\D}}^{1}\big( \Gamma^{\ast}_{\PP} \big) \oplus \frac{\widetilde{\mathfrak{M}}(\PP)}{\mathscr{C}(\PP)} \oplus
        \mathscr{E}(\PP),
    \end{equation}
where \ec{\widetilde{\mathfrak{M}}(\PP)} is a submodule of generalized derivations of \ec{P_{1}}. We remark that splitting \eqref{eq:Rest1Split-Intro} recovers some of the results in \cite{Zhu19} for admissible Poisson algebras associated to Poisson modules since, in this case, \ec{[\,,\,]_{1}=0} and \ec{\K=0}.

Our algebraic approach is also applied to some examples of admissible Poisson algebras of geometric nature \cite{RuGaVo20}. More precisely, we applied the general results to the description of the so--called infinitesimal (first) Poisson cohomology around Poisson submanifolds \cite{KaVo98, RuGaVo20}. In particular, in the case of a symplectic leaf, we adapt formulas \eqref{eq:H1-Intro}--\eqref{eq:H1Particular-Intro} to the computation of the infinitesimal first Poisson cohomology under some assumptions on the isotropy of the leaf \cite{Mac95, Duf05, VeVo18}.

The paper is organized as follows.
Section \ref{sec:preliminaries} recalls basic definitions.
In Section \ref{sec:GPA}, we define admissible Poisson structures on trivial extension algebras and present some properties.
In Section \ref{sec:poissontriples}, a one-to-one correspondence is shown to exist between admissible Poisson algebras, Poisson triples and a class of Lie algebroids.
In Section \ref{sec:ComplexAPA}, two natural cochain complexes associated to admissible Poisson algebras are presented.
In Section \ref{sec:equivalenceGPA}, the so-called gauge transformations are defined and equivalence criteria for admissible Poisson algebras are given.
Section \ref{sec:hamder} is devoted to the study of Poisson and Hamiltonian derivations of admissible Poisson algebras.
Then, in Section \ref{sec:FirstCohomology}, we apply the previous results to describe some properties of the first Poisson cohomology of admissible Poisson algebras.
Finally, in Section \ref{sec:PoissonModules}, admissible Poisson algebras induced by Poisson modules are studied and, in Section \ref{sec:PoissonSub}, admissible Poisson algebras arising in the infinitesimal geometry of Poisson submanifolds are described.

    \section{Preliminaries}\label{sec:preliminaries}

Here, we recall some basic facts about Poisson algebras, Lie algebroids structures, and contravariant derivatives (for more details, see also \cite{Kos90, Vais91, Mac95, GaVaVo12}).

\paragraph{Poisson Algebras.} Throughout this section $P$ denotes an associative and commutative unital algebra, with product $\cdot$, over a commutative ring $R$ with unit.

\begin{definition}
A \emph{Poisson structure} on $P$ is a Lie bracket \ec{\{\,,\,\}:P \times P \to P} compatible with the product by the so--called Leibniz rule,
    \begin{equation*}
        \{a, b \cdot c\} = \{a, b\} \cdot c + b \cdot \{a, c\},
    \end{equation*}
for all \ec{a,b,c \in P}.
\end{definition}

The triple \ec{\PP = (P,\cdot,\{\,,\,\})} is called a \emph{Poisson ($R$--)algebra}. A \emph{Poisson subalgebra} of a Poisson algebra $\PP$ is an commutative and associative subalgebra $S$ of $P$ which is closed under the Poisson structure, \ec{\{S,S\} \subseteq S}. An associative ideal $I$ of $P$ is called \emph{Poisson ideal} if it is also a Lie ideal with respect to the Poisson structure, \ec{\{P,I\} \subseteq I}.

A \emph{Casimir} element of $\PP$ is an element \ec{k \in P} satisfying \ec{\{k,a\} = 0}, for all \ec{a \in P}. The Poisson subalgebra of such elements is denoted by \ec{\mathrm{Casim}(\PP)}.

Recall that a \emph{derivation} of the commutative algebra $P$ is an $R$--linear mapping \ec{D:P \to P} that satisfies the Leibniz rule: \ec{D(ab) = D(a)b + aD(b)}, for all \ec{a,b \in P}. We denote the $P$--module of all derivations of $P$ by \ec{\mathrm{Der}(P)}. A \emph{derivation of a Poisson algebra $\PP$}, or Poisson derivation, is a derivation \ec{D \in \mathrm{Der}(P)} which is also a derivation of the Poisson structure, \ec{D\{a,b\} = \{D a,b\} + \{a,D b\}}, for all \ec{a,b \in P}. We denote by \ec{\mathrm{Poiss}(\PP)} the Lie subalgebra of all derivations of $\PP$. Given an element \ec{h \in P}, a derivation of the form \ec{D = \{h, \cdot\}} is called a \emph{Hamiltonian derivation} of $\PP$, associated to the Hamiltonian element $h$ and with respect to the Poisson structure \ec{\{\,,\,\}}. By \ec{\mathrm{Ham}(\PP)} we denote the Lie ideal of all Hamiltonian derivations of $\PP$, $\pmb{[} \mathrm{Poiss}(\PP), \mathrm{Ham}(\PP)\pmb{]} \subseteq \mathrm{Ham}(\PP)$. Here and in the remainder of the paper, the bracket
    \begin{equation*}\label{eq:Commutator}
        \pmb{[}\,,\,\pmb{]}: \mathrm{End}_{R}(P) \times \mathrm{End}_{R}(P) \longrightarrow \mathrm{End}_{R}(P)
    \end{equation*}
will denote the commutator of $R$--linear mappings from $P$ into itself: \ec{\pmb{[}D_{1},D_{2}\pmb{]} := D_{1} \circ D_{2} - D_{2} \circ D_{1}}, for all \ec{D_{1},D_{2} \in \mathrm{End}_{R}(P)}.

An $R$--linear mapping \ec{\phi: (P', \cdot', \{\,,\,\}') \mapsto (P, \cdot, \{\,,\,\})} between two Poisson algebras is called a \emph{Poisson morphism} if it is a morphism for both the commutative and Poisson structures,
    \begin{equation}\label{EcPhiAlg}
        \phi\{a,b\}' = \{\phi(a),\phi(b)\} \quad \text{and} \quad \phi(a \cdot' b) = \phi(a) \cdot \phi(b),
    \end{equation}
for all \ec{a,b \in P}. A \emph{Poisson isomorphism} is an invertible Poisson morphism. Two Poisson algebras are said to be \emph{Poisson isomorphic} or \emph{equivalent} if there exists a Poisson isomorphism $\phi$ between them. In this case, we say that they are Poisson \emph{$\phi$--equivalent} to specify a particular isomorphism $\phi$.

\begin{examplex}
The standard example of a Poisson algebra is given by a Poisson bivector field \ec{\Pi \in \Gamma \wedge^2 \T{M}} on a smooth manifold $M$. In this case, we have that \ec{P = \Cinf{M}} is the commutative algebra of smooth functions on $M$ with pointwise multiplication and \ec{\{f,g\}_{\Pi} := \Pi(\dd f, \dd g)} is the Poisson structure on \ec{M}.
\end{examplex}

\paragraph{Cohomology of a Poisson Algebra.} Let \ec{\PP = (P,\cdot,\{\,,\,\})} be a Poisson algebra. If we define \break \ec{\mathfrak{X}^{k}_{P} := \{\text{$R$--multilinear skew--symmetric derivations}\ X:P \times \cdots \times P \to P\}}, then the Poisson structure \ec{\{\,,\,\}} induces a coboundary operator \ec{\delta_{\PP}:\mathfrak{X}^{k}_{P} \to \mathfrak{X}^{k+1}_{P}} defined by
    \begin{equation}\label{eq:DeltaP}
        \fiteq{\big( \delta_{\PP}X \big)(\pi_0,\dots,\pi_{k}):=\sum_{i=0}^{k}(-1)^{i} \big\{ \pi_{i},X(\pi_{0},\dots,\widehat{\pi}_{i},\dots,\pi_{k}) \big\} + \sum_{0\leq i < j\leq k}(-1)^{i+j} X\big( \{\pi_{i},\pi_{j}\},\pi_{0},\dots,\widehat{\pi}_{i},\dots,\widehat{\pi}_{j},\dots,\pi_{k} \big)},
    \end{equation}
with \ec{X \in \mathfrak{X}^{k}_{P}} and \ec{\pi_{0},\dots,\pi_{k} \in P}. Here, and in the remainder of the paper, the symbol \ $\widehat{}$ \ denotes omission. The cohomology induced by \ec{\delta_{\PP}} is called the \emph{cohomology of the Poisson algebra} $\PP$.

\begin{lemma}\label{lema:IsoCohomo}
Equivalent Poisson algebras have isomorphic Poisson cohomologies.
\end{lemma}

In particular, the zero Poisson cohomology of $\PP$ is \ec{\mathrm{Casim}(\PP)} and the first Poisson cohomology is the quotient \ec{\mathrm{Poiss}(\PP) / \mathrm{Ham}(\PP)}.

\paragraph{Lie Algebroids.} Let $P$ be the associative and commutative $R$--algebra with unit and $A$ a faithful $P$--(bi)module.

\begin{definition}\label{def:LieAlgebroid}
A \emph{Lie algebroid structure} on $A$ over $P$ is a tuple \ec{(\cSch{\,,\,}, \varrho)}, where \ec{\cSch{\,,\,}} is a Lie bracket on $A$ and \ec{\varrho:A \to \mathrm{Der}(P)} is a morphism of $P$--modules, called the anchor map, satisfying the following ``Leibniz rule'' condition
	\begin{equation}\label{EcLeibnizAnchor}
	   \cSch{X, aY} = a\cSch{X, Y} + \varrho(X)(a)\,Y, \quad a \in P;
	\end{equation}
for all \ec{X, Y \in A}.
\end{definition}

It is well--known that, since $A$ is faithful, condition \eqref{EcLeibnizAnchor} implies that the anchor map is a morphism of Lie algebras: \ec{\varrho\cSch{X,Y} = \pmb{[}\varrho(X), \varrho(Y) \pmb{]}}.
The triple \ec{\A = (A,\cSch{\,,\,}, \varrho)} is called a \emph{Lie algebroid on $A$ over $P$}.

\begin{remark}\label{rmrk:AnclaCero}
If \ec{\varrho = 0}, then a Lie algebroid structure on $A$ over $P$ is just a $P$--linear Lie bracket on $A$, by \eqref{EcLeibnizAnchor}. If \ec{\cSch{\,,\,} = 0} and \ec{\varrho = 0}, then $A$ is said to be endowed with the \emph{trivial} Lie algebroid.
\end{remark}

\begin{remark}
The definition of Lie algebroid in an algebraic framework (Definition \ref{def:LieAlgebroid}) coincides with the equivalent notions of Lie pseudoalgebra \cite{Mac95} and Lie-Rinehart algebra \cite{Hue98}, which are frequently used by different authors. In this paper, we continue to use the term Lie algebroid following \cite{GaVaVo12}.
\end{remark}

A \emph{morphism of Lie algebroids}, between two Lie algebroids \ec{(A',\cSch{\,,\,}',\varrho')} and \ec{(A,\cSch{\,,\,},\varrho)} over the \emph{same} commutative $R$--algebra $P$, is a morphism of $P$--modules \ec{\varphi: A' \rightarrow A} over the identity such that
    \begin{equation}\label{ecMorphismAlgbd}
         \varphi\cSch{X',Y'}' = \cSch{\varphi(X'),\varphi(Y')} \quad \text{and} \quad \varrho' = \varrho \circ \varphi,
    \end{equation}
for all \ec{X',Y' \in A'}.

\paragraph{The Lie Algebroid of a Poisson Algebra.} For the associative and commutative $R$--algebra $P$ with unit \ec{1_{P}}, we denote by \ec{\Omega^{1}_{P}} the (free) $P$--submodule of \ec{{\mathrm{Der}(P)}^{\ast}} consisting of the kernel of the multiplication \ec{P\otimes_{R} P \to P}, and equipped with the $R$--derivation \ec{\dd:P \to \Omega^{1}_{P}} defined by \ec{\dd{f} := 1_{P} \otimes_{R} f - f \otimes_{R} 1_{P}}, for all \ec{f\in P}. By construction, we have
    \begin{equation*}
        \Omega^{1}_{P} = \mathrm{span}_{P}\{\dd{f} \mid f \in P\},
    \end{equation*}
(for more details, see \cite{DuMi94, GaVaVo12}). Moreover, given \ec{\alpha \in \Omega^1_{P}}, we define \ec{\dd\alpha} by the usual formula: $(\dd\alpha)(D_{1},D_{2}) := D_{1}(\alpha(D_{2})) - D_{2}(\alpha(D_{1})) - \alpha(\pmb{[}D_{1},D_{2}\pmb{]})$, for all \ec{D_{1},D_{2} \in \mathrm{Der}(P)}.

Now, suppose that \ec{(P,\cdot,\{\,,\,\})} is a Poisson algebra. Extending the mapping \ec{\dd h \to \{h, \cdot\}} by $P$--linearity, we get a morphism \ec{\varrho_{P}:\Omega^{1}_{P} \to \mathrm{Der}(P)} uniquely defined by
    \begin{equation*}
        \varrho_{P}(\dd h) := \{h, \cdot\} , \quad h \in P.
    \end{equation*}

\begin{definition}\label{def:algebroidpoisson}
The triple \ec{(\Omega^{1}_{P}, \cSch{\,,\,}_{P}, \varrho_{P})} is called the \emph{Lie algebroid of the Poisson algebra} \ec{(P,\cdot,\{\,,\,\})}, where the Lie bracket is defined by \ec{\cSch{\alpha,\beta}_{P} := (\dd\beta)(\varrho_{P}(\alpha), \cdot) - (\dd\alpha)(\varrho_{P}(\beta), \cdot) + \dd(\beta( \varrho_{P}(\alpha) ))}, for all \break \ec{\alpha,\beta \in \Omega^{1}_{P}} \cite{Hue90, GaVaVo12}.
\end{definition}

Note that if \ec{\alpha = \dd f} and \ec{\beta = \dd g}, for some \ec{f,g \in P}, then \ec{\cSch{\dd f,\dd g}_{P} = \dd\{f,g\}}.

\begin{examplex}
Every Poisson manifold \ec{(M,\Pi)} induces the Lie algebroid on \ec{\Gamma\,\T^{\ast}M} over \ec{\Cinf{M}} given by \ec{(\Omega_{P}^{1} = \Gamma\,\T^{\ast}M, \cSch{\,,\,}_{P} = \{\,,\,\}_{\T^{\ast}M}, \varrho_{P} = \Pi^{\natural})}, called the cotangent Lie algebroid \cite{Vais94, Mac95, Duf05}.
\end{examplex}

\paragraph{Contravariant Differentials.} We recall the definition of contravariant derivative \cite{Vais91} and introduce some natural coboundaries operators (see also \cite{ReVoWe96, Car03,Zhu19}).

Let \ec{\PP_{0} = (P_{0},\cdot,\{\,,\,\})} be a Poisson algebra and $P_{1}$ a faithful $P_{0}$--(bi)module. A \emph{contravariant derivative (connection)} on \ec{P_{1}} is an $R$--bilinear mapping \ec{\D:\Omega^{1}_{P_{0}} \times P_{1} \to P_{1}}, \ec{(\alpha, \eta) \mapsto \D_{\alpha}\eta}, satisfying
        \begin{equation*}
            \D_{f\alpha}\eta = f\D_{\alpha}\eta \quad \text{and} \quad \D_{\alpha}(f\eta) = f\D_{\alpha}\eta + \varrho_{P}(\alpha)(f)\,\eta, \quad f \in P_{0}.
        \end{equation*}
Here, \ec{(\Omega^{1}_{P_{0}}, \cSch{\,,\,}_{P_{0}}, \varrho_{P_{0}})} is the Lie algebroid of the Poisson algebra $\PP_{0}$.

Let $\D$ be a contravariant derivative on $P_{1}$. The \emph{curvature} of $\D$ is the $P_{0}$--linear endomorphism of \ec{P_{1}} defined by
    \begin{equation}\label{EcCurvD}
         \mathrm{Curv}^{\D}(\alpha,\beta) := \pmb{[}\D_{\alpha},\D_{\beta}\pmb{]} - \D_{\cSch{\alpha, \beta}_{P_{0}}}, \quad \alpha,\beta \in \Omega^{1}_{P_{0}}.
    \end{equation}
In particular, $\D$ is said to be \emph{flat} if its curvature is zero, \ec{\mathrm{Curv}^{\D} = 0}.

\begin{remark}
For two contravariant derivatives $\D$ and $\D'$ on $P_{1}$, the difference \ec{\D_{\alpha}-\D_{\alpha}'} is a $P_{0}$--linear endomorphism of \ec{P_{1}}, for all \ec{\alpha \in \Omega^{1}_{P_{0}}}. Then, given a \ec{P_{0}}--bilinear mapping \ec{\Xi: \Omega_{P_{0}}^{1} \times P_{1} \to P_{1}}, the sum \ec{\D + \Xi} defines a (new) contravariant derivative on \ec{P_{1}}.
\end{remark}

We define \ec{\chi^{k}_{P_{0}}(\Omega^{1}_{P_{0}}; P_{1}) := \{\text{$P_{0}$--multilinear skew--symmetric mappings}\ \Omega_{P_{0}}^{1} \times \cdots \times \Omega_{P_{0}}^{1} \to P_{1}\}}, which is a $P_{0}$--module. Then, $\D$ induces a contravariant differential \ec{\overline{\dd}_{\D}:\chi^{k}_{P_{0}}(\Omega^{1}_{P_{0}}; P_{1}) \to \chi^{k+1}_{P_{0}}(\Omega^{1}_{P_{0}}; P_{1})} by
    \begin{equation}\label{eq:dD}
        \fiteq{(\overline{\dd}_{\D}\overline{Q})(\alpha_{0},\ldots,\alpha_{k}) := \sum_{i=0}^{k}(-1)^{i}\,\D_{\alpha_{i}} \big( \overline{Q}(\alpha_{0}, \ldots, \widehat{\alpha}_{i},\ldots,\alpha_{k}) \big) + \sum_{0 \leq i < j \leq k} (-1)^{i+j} \overline{Q}\big( \cSch{\alpha_{i},\alpha_{j}}_{P},\alpha_{0} \ldots, \widehat{\alpha}_{i},\ldots,\widehat{\alpha}_{j}, \ldots \alpha_{k} \big)},
    \end{equation}
for \ec{\overline{Q} \in \chi^{k}_{P_{0}}(\Omega^{1}_{P_{0}}; P_{1})}, with \ec{\alpha_{0},\ldots,\alpha_{k} \in \Omega^{1}_{P_{0}}}. Moreover, $\overline{\dd}_{\D}$ is a coboundary operator if and only if the contravariant derivative $\D$ is flat since
    \begin{equation}\label{eq:dD2}
        \big( \overline{\dd}_{\D}^{2}\overline{Q} \big)(\alpha_{0},\ldots,\alpha_{k+1}) = \sum_{0 \leq i < j \leq k+1} (-1)^{i+j+1}\, \mathrm{Curv}^{\D}(\alpha_{i},\alpha_{j})\big( \overline{Q}(\alpha_{0} \ldots, \widehat{\alpha}_{i},\ldots,\widehat{\alpha}_{j}, \ldots \alpha_{k+1}) \big).
    \end{equation}

Similarly, we define \ec{\X{R}^{k}(P_{0}; P_{1}) := \{\text{$R$--multilinear skew--symmetric derivations}\ P_{0} \times \cdots \times P_{0} \to P_{1}\}}, which is a $P_{0}$--module. Then, $\D$ induces a contravariant differential \ec{\dd_{\D}:\X{R}^{k}(P_{0}; P_{1}) \to \X{R}^{k+1}(P_{0}; P_{1})} by
    \begin{equation}\label{eq:dDbar}
        \fiteq{(\dd_{\D}Q)(f_{0},\ldots,f_{k}) := \sum_{i=0}^{k}(-1)^{i}\,\D_{\dd f_{i}} \big( Q(f_{0}, \ldots, \widehat{f}_{i},\ldots,f_{k}) \big) + \sum_{0 \leq i < j \leq k} (-1)^{i+j}\, Q\big( \{f_{i},f_{j}\},f_{0} \ldots, \widehat{f}_{i},\ldots,\widehat{f}_{j}, \ldots f_{k} \big)},
    \end{equation}
for \ec{Q \in \X{R}^{k}(P_{0}; P_{1})}, with \ec{f_{0},\ldots,f_{k} \in P_{0}}. We have that $\dd_{\D}$ is a coboundary operator if and only if the contravariant derivative $\D$ is flat since
    \begin{equation}\label{eq:dD2bar}
        \big( \dd_{\D}^{2}Q \big)(f_{0},\ldots,f_{k+1}) = \sum_{0 \leq i < j \leq k+1} (-1)^{i+j+1}\, \mathrm{Curv}^{\D}(\dd f_{i},\dd f_{j})\big( Q(f_{0} \ldots, \widehat{f}_{i},\ldots,\widehat{f}_{j}, \ldots f_{k+1}) \big).
    \end{equation}

Now, we present cochain complexes associated to a given Poisson algebra $\PP_{0}$ and some data \ec{([\,,\,]_{1}, \D)} on $P_{1}$, consisting of a $P$--linear Lie bracket \ec{[\,,\,]_{1}} preserved by the contravariant derivative $\D$,
    \begin{equation}\label{eq:Dpreserves}
        \D_{\alpha}[\eta, \xi]_{1} = [\D_{\alpha}\eta, \xi]_{1} + [\eta, \D_{\alpha}\xi]_{1}, \quad \alpha \in \Omega_{P_{0}}^{1},\ \eta, \xi \in P_{1}.
    \end{equation}
Moreover, the curvature of $\D$ is of adjoint type with respect to \ec{[\,,\,]_{1}},
    \begin{equation}\label{eq:CurvAd}
        \mathrm{Curv}^{\D}(\alpha,\beta) = [\eta, \cdot\,]_{1}, \quad \text{for some}\ \eta \in P_{1}; \quad \text{with}\ \alpha,\beta \in \Omega_{P_{0}}^{1}.
    \end{equation}

Note that $\D$ leaves invariant the center \ec{Z(P_{1})} of the Lie algebra \ec{(P_{1},[\,,\,]_{1})}, which is also a $P_{0}$--module.

Set \ec{\overline{\Gamma}_{P_{0}}^{k}(\Omega_{P_{0}}^{1}; Z(P_{1})) := \{\text{$P_{0}$--\emph{multilinear skew--symmetric mappings}}\ \Omega_{P_{0}}^{1} \times \cdots \times \Omega_{P_{0}}^{1} \to Z(P_{1})\}}. Then, by \eqref{eq:dD2}, \eqref{eq:Dpreserves} and \eqref{eq:CurvAd}, the contravariant differential $\overline{\dd}_{\D}$ in \eqref{eq:dD} restricts to the \emph{coboundary operator},
    \begin{equation}\label{eq:Partial}
        \overline{\partial}_{\D} \equiv \overline{\dd}_{\D}|_{\Gamma^{k}_{P_{0}}}:\overline{\Gamma}_{P_{0}}^{k}\big( \Omega_{P_{0}}^{1}; Z(P_{1}) \big) \longrightarrow \overline{\Gamma}_{P_{0}}^{k+1}\big( \Omega_{P_{0}}^{1}; Z(P_{1}) \big) \qquad (\overline{\partial}_{\D}^{2} = 0).
    \end{equation}

Similarly, set \ec{\Gamma_{R}^{k}(P_{0}; Z(P_{1})) := \{\text{$R$--\emph{multilinear skew--symmetric derivations}}\ P_{0} \times \cdots \times P_{0} \to Z(P_{1})\}}. Then, by \eqref{eq:dD2bar}, \eqref{eq:Dpreserves} and \eqref{eq:CurvAd}, the contravariant differential $\dd_{\D}$ in \eqref{eq:dDbar} restricts to the \emph{coboundary operator},
    \begin{equation}\label{eq:PartialBar}
        \partial_{\D} \equiv \dd_{\D}|_{\Gamma^{k}_{R}}:\Gamma_{R}^{k}\big( P_{0}; Z(P_{1}) \big) \longrightarrow \Gamma_{R}^{k+1}\big( P_{0}; Z(P_{1}) \big) \qquad (\partial_{\D}^{2} = 0).
    \end{equation}

\begin{proposition}\label{prop:TwoCochainCmplx}
Every data \ec{(\PP_{0}, (P_{1}, [\,,\,]_{1}, \D))} satisfying \eqref{eq:Dpreserves} and \eqref{eq:CurvAd} induce two cochain complexes associated to the coboundaries operators \eqref{eq:Partial} and \eqref{eq:PartialBar}.
\end{proposition}

These cochain complexes play an important role in Sections \ref{sec:ComplexAPA}, \ref{sec:equivalenceGPA} and \ref{sec:FirstCohomology}.

    \section{Admissible Poisson Algebras}\label{sec:GPA}

In this section, we present a natural class of Poisson structures on the trivial extension algebra of a given Poisson algebra which generelizes the known class of Poisson structures induced by Poisson modules.

Let $R$ be a commutative ring with unit. We assume here, and in the rest of the sections, that
    \begin{itemize}
      \item a Poisson $R$--algebra \ec{\PP_{0} = (P_{0}, \cdot, \{\,,\,\}_{0})};
      \item a faithful \ec{P_0}--(bi)module \ec{P_{1}}.
    \end{itemize}

We denote by
    \begin{equation*}
        P = P_{0} \ltimes P_{1}
    \end{equation*}
the trivial extension algebra of the commutative algebra \ec{P_{0}} by \ec{P_{1}} and by
    \begin{equation*}
        \pr_0: P \rightarrow P_{0}, \quad \pr_1: P \rightarrow P_{1} \quad \text{and} \quad \iota_{0}: P_{0} \hookrightarrow P_{0} \oplus \{0\}, \quad \iota_{1}: P_{1} \hookrightarrow \{0\} \oplus P_{1},
    \end{equation*}
the canonical projections and the inclusion maps, respectively.

\begin{definition}\label{def:GPA}
A Poisson structure \ec{\{\,,\,\}} on the trivial extension algebra $P$ is said to be \emph{admissible} if the natural projection \ec{P \to P_{0}} is a Poisson morphism,
    \begin{equation}\label{EcGPA}
        \pr_{0}\{f \oplus \eta, g \oplus \xi\} = \{f,g\}_{0},
    \end{equation}
for all \ec{f \oplus \eta, g \oplus \xi \in P}.
\end{definition}

In particular, an admissible Poisson structure on $P$ is compatible with the multiplication \eqref{EcPproduct} by the Leibniz rule. The pair \ec{\PP = (P, \{\,,\,\})} will be called an \emph{admissible Poisson algebra} (APA).

\begin{remark}
By definition of trivial extension algebra, the projection \ec{\pr_{0}} is a morphism of commutative algebras: \ec{\pr_{0}[(f \oplus \eta) \cdot (g \oplus \xi)] = fg = \pr_{0}(f \oplus \eta) \cdot \pr_{0}(g \oplus \xi)}.
\end{remark}

\begin{examplex}
Let \ec{(A, \cdot, \{\,,\,\}_{0})} be a Poisson algebra and \ec{P_{1} \subset A} a Poisson ideal. If \ec{P_{0}= A / P_{1}} is the quotient Poisson algebra, then the trivial extension algebra of \ec{P_{0}} by \ec{P_{1}} is an admissible Poisson algebra \cite{Car03} with (admissible) Poisson structure defined by $\big\{ (f + P_{1}) \oplus f', (g + P_{1}) \oplus g' \big\} := \big( \{f,g\}_{0} + P_{1} \big) \oplus \big( \{f,g'\}_{0} - \{g,f'\}_{0} \big)$, for all \ec{f,g \in A} and \ec{f',g' \in P_{1}}.
\end{examplex}

\begin{examplex}
Let \ec{(A, \cdot, \{\,,\,\})} be a Poisson algebra. If there exists a Poisson ideal \ec{P_{1} \subset A} such that \ec{P_{1}^{2}=0} and a complementary subalgebra \ec{P_{0} \subset A} to \ec{P_{1}}, \ec{A=P_{0} \oplus P_{1}}, then \ec{(A, \cdot, \{\,,\,\})} is an admissible Poisson algebra.
\end{examplex}

\begin{examplex}\label{exm:APAmatrix}
Let \ec{\PP_{0}=(P_{0}, \cdot, \{\,,\,\}_{0})} be a Poisson algebra and \ec{P_{1}} the \ec{P_{0}}--module of (\ec{k \times k})--matrices with entries in \ec{P_{0}}. Then, the tuple \ec{(P = P_{0}\ltimes P_{1}, \{\,,\,\})} is an admissible Poisson algebra with (admissible) Poisson structure given by
    \begin{equation}\label{eq:APmatrix}
        \{f\oplus M, g\oplus N\} = \{f,g\}_{0} \oplus \big( D_{\dd{f}}N- D_{\dd{g}}M + [|M,N|] \,\big),
    \end{equation}
for all \ec{f\oplus M, g\oplus N \in P}. Here, we define
    \begin{equation}\label{eq:Dmatrix}
        \big( D_{g\dd{f}}M \big)_{ij} := g\{f, M_{ij}\}_{0}, \quad f,g \in P_{0};
    \end{equation}
and \ec{[|\,,\,|]} denotes the commutator of matrices.
\end{examplex}

\begin{remark}\label{remark:admissible}
An admissible Poisson structure on $P$ does not always exist. A necessary condition is that the \ec{P_{0}}--module \ec{P_{1}} admits a contravariant derivative (Theorem \ref{teo:correspondenceGPA-PT-LA}), which is generally not the case \cite{Car03}.
\end{remark}

Let \ec{\PP = (P, \{\,,\,\})} be an admissible Poisson algebra. Definition \ref{def:GPA} implies that the admissible Poisson structure of $\PP$ can be represented as
    \begin{equation*}
        \{P,P\} = \{P_{0},P_{0}\}_{0} \oplus \pr_{1}\{P,P\}.
    \end{equation*}
So, it is completely determined by the second factor \ec{\pr_{1}\{P,P\}}. In particular, the ideal \ec{\{0\}\oplus P_1} of the commutative algebra $P$ is also a Poisson ideal of $\PP$.

\begin{proposition}
The $P_{0}$--module \ec{P_1} inherits from $\PP$ a \ec{P_0}--linear Lie bracket \ec{[\,,\,]_{1}} defined by
    \begin{equation}\label{EcLieBraP1}
        [\eta,\xi]_{1} := \pr_{1}\{\iota_{1} \eta,\iota_{1} \xi\}, \quad \eta,\xi \in P_{1}.
    \end{equation}
\end{proposition}

We say that the tuple \ec{(P_{1},[\,,\,]_{1})} is the Lie algebra associated to $\PP$, or induced by $\PP$. Note that, taking into account Remark \ref{rmrk:AnclaCero}, the above proposition can be reformulated as follows.

\begin{corollary}\label{cor:P1algebroid}
Every admissible Poisson algebra induces a Lie algebroid structure on \ec{P_{1}} over \ec{P_{0}} with Lie bracket defined in \eqref{EcLieBraP1} and the zero anchor map.
\end{corollary}

In addition to a \ec{P_{0}}--linear Lie bracket on \ec{P_{1}}, every admissible Poisson algebra induces two other data which are presented in the next section. Moreover, in Section \ref{sec:PoissonModules} we show that admissible Poisson structures extend the well--known class of Poisson structures induced by Poisson modules. In Section \ref{sec:PoissonSub}, we present a class of admissible Poisson algebras coming from the infinitesimal geometry of Poisson submanifolds.

    \section{Poisson Triples}\label{sec:poissontriples}

Here, we show that admissible Poisson algebras are parameterized by some data, that we call Poisson triples, involving not necessarily flat contravariant derivatives. Moreover, they are in correspondence with a special class of Lie algebroids.

Let \ec{(\Omega^{1}_{P_{0}}, \cSch{\,,\,}_{P_{0}}, \varrho_{P_{0}})} be the Lie algebroid of the given Poisson algebra $\PP_{0}$ (Definition \ref{def:algebroidpoisson}). Suppose that we have a triple \ec{([\,,\,]_{1},\D,\K)} consisting of
    \begin{itemize}
    \item a \ec{P_{0}}--linear Lie bracket \ec{[\,,\,]_{1}} on \ec{P_{1}};
    \item a contravariant derivative $\D$ on \ec{P_{1}};
    \item a \ec{P_{0}}--linear skew--symmetric mapping \ec{\mathscr{K}: \Omega^{1}_{P_{0}} \times \Omega^{1}_{P_{0}} \to P_{1}};
    \end{itemize}
satisfying the following conditions:
    \begin{align}
            & \hspace{0.88cm} \pmb{[}\D_{\alpha},\mathrm{ad}_{\eta}\pmb{]} = \mathrm{ad}_{\D_{\alpha}\eta}, \label{EcPT1} \\[0.15cm]
            & \hspace{0.89cm} \pmb{[}\D_{\alpha},\D_{\beta}\pmb{]} - \D_{\cSch{\alpha, \beta}_{P_{0}}} = \operatorname{ad}_{\K(\alpha,\beta)}, \label{EcPT2} \\[0.15cm]
            & \underset{(\alpha,\beta,\gamma)}{\mathfrak{S}} \D_{\alpha}\,\mathscr{K}(\beta,\gamma) + \mathscr{K}(\alpha,\cSch{\beta,\gamma}_{P_{0}}) = 0, \label{EcPT3}
    \end{align}
for all \ec{\alpha,\beta,\gamma \in P_{0}} and \ec{\eta \in P_{1}}. Here, \ec{\mathrm{ad}_{\eta} := [\eta,\cdot]_{1}} and ${\mathfrak{S}}$ denotes the cyclic sum.

\begin{definition}\label{def:PoissonTrip}
A triple \ec{([\,,\,]_{1},\D,\mathscr{K})} satisfying \eqref{EcPT1}--\eqref{EcPT3} is said to be a \emph{Poisson triple} of the trivial extension algebra $P$.
\end{definition}

Conditions for Poisson triples of $P$ admit the following interpretation: \eqref{EcPT1} says that $\D$ derives the Lie bracket \ec{[\,,\,]_{1}}, see formula \eqref{eq:Dpreserves}. Condition \eqref{EcPT2} states that \ec{\K \in \chi^{2}_{P_{0}}(\Omega^{1}_{P_{0}}; P_{1})} controls the curvature \eqref{EcCurvD} of $\D$ in the sense that \ec{\mathrm{Curv}^{\D} = \mathrm{ad} \circ {\K}}. In particular, this means that condition \eqref{eq:CurvAd} holds for \ec{\eta = \K(\alpha,\beta)}. Moreover, we have that $\D$ is \emph{flat} if and only if $\K$ takes values in the center of the Lie algebra \ec{(P_{1},[\,,\,]_{1})}. Condition \eqref{EcPT3} is a Bianchi identity type for $\K$, saying that
   \begin{equation}\label{eq:dK0}
        \dd_{\D}\K = 0,
    \end{equation}
where \ec{\dd_{\D}} is the contravariant differential \eqref{eq:dD} on \ec{\chi^{2}_{P_{0}}(\Omega^{1}_{P_{0}}; P_{1})} associated to $\D$ and $\PP_{0}$.

\begin{examplex}\label{exm:TripleMatrix}
Let \ec{\PP_{0}=(P_{0}, \cdot, \{\,,\,\}_{0})} be a Poisson algebra and \ec{P_{1}} the \ec{P_{0}}--module of (\ec{k \times k})--matrices with entries in \ec{P_{0}}. Then, a Poisson triple of \ec{P_{0} \ltimes P_{1}} is given by
    \begin{equation}\label{eq:PTripleMatrix}
        \big( [\,,\,]_{1} = [|\,,\,|], \D = D, \K = 0 \big),
    \end{equation}
where \ec{[|\,,\,|]} is the commutator of matrices and $D$ is defined in \eqref{eq:Dmatrix}.
\end{examplex}

Now, consider the $P_{0}$--module given by the direct sum of the \ec{P_{0}}--modules \ec{\Omega_{P_{0}}^{1}} and \ec{P_{1}},
    \begin{equation}\label{A}
        A = \Omega^1_{P_0} \oplus P_1.
    \end{equation}
Let \ec{(\cSch{\,,\,}, \varrho)} be a Lie algebroid structure on $A$ over \ec{P_{0}}. By formula \eqref{ecMorphismAlgbd}, the canonical projection \ec{\mathrm{proj}_{0}: A \to \Omega^{1}_{P_0}} is a morphism of Lie algebroids if
    \begin{equation}\label{EcA}
        \mathrm{proj}_{0}\cSch{\alpha \oplus \eta, \beta \oplus \xi} = \cSch{\alpha,\beta}_{P_{0}},
    \end{equation}
for all \ec{\alpha \oplus \eta, \beta \oplus \xi \in A}, and
    \begin{equation}\label{eq:Rho0Proj0}
        \varrho = \varrho_{P_{0}} \circ \mathrm{proj}_{0}.
    \end{equation}

We arrive at the main result of this section.

\begin{theorem}\label{teo:correspondenceGPA-PT-LA}
There is a one--to--one correspondence between the sets of
    \begin{enumerate}[label=(\roman*)]
      \item \label{Correspondence1} (admissible) Poisson structures on $P$ satisfying \eqref{EcGPA};
      \item \label{Correspondence2} Poisson triples of $P$;
      \item \label{Correspondence3} Lie algebroid structures \ec{(\cSch{\,,\,}, \varrho_{P_{0}} \circ \mathrm{proj}_{0})} on $A$ over \ec{P_{0}} satisfying \eqref{EcA}.
    \end{enumerate}
\end{theorem}

As a consequence of this theorem, the existence of a contravariant derivative on \ec{P_{1}} is a necessary condition for the trivial extension algebra $P$ to admit an admissible Poisson structure (Remark \ref{remark:admissible}).

\begin{examplex}
Let $\PP$ be the admissible Poisson algebra defined by the bracket \eqref{eq:APmatrix}. The corresponding Poisson triple of $P$ is given in \eqref{eq:PTripleMatrix}. The corresponding Lie algebroid structure \ec{(\cSch{\,,\,}, \varrho_{P_{0}} \circ \mathrm{proj}_{0})} on \ec{\Omega^{1}_{P_{0}} \oplus P_{1}} over \ec{P_{0}} is given by the Lie bracket \ec{\cSch{\dd f \oplus M, \dd g \oplus N} = \cSch{\dd f,\dd g}_{P_{0}} \oplus (D_{\dd{f}}N- D_{\dd{g}}M + [|M,N|])}, for all \ec{f,g \in P_{0}} and \ec{M,N \in P_{1}}.
\end{examplex}

\begin{remark}
Admissible Poisson algebras induce contravariant derivatives and Lie brackets satisfying conditions \eqref{eq:Dpreserves} and \eqref{eq:CurvAd}.
\end{remark}

In the remainder of this section we present a proof of Theorem \ref{teo:correspondenceGPA-PT-LA}.

\paragraph{APA's and Poisson Triples.} First, we show that there is a one--to--one correspondence between admissible Poisson algebras and Poisson triples.

\begin{lemma}\label{LemaTripleToAlg}
Each Poisson triple \ec{([\,,\,]_{1},\D,\mathscr{K})} of the commutative algebra $P$ induces an admissible Poisson structure on $P$ defined by
    \begin{equation}\label{EcBracketPT}
        \{f \oplus \eta, g \oplus \xi\} := \{f,g\}_{0} \oplus \big( \D_{\dd{f}}\xi - \D_{\dd{g}}\eta + [\eta,\xi]_{1} + \K(\dd{f},\dd{g}) \big),
    \end{equation}
for all \ec{f \oplus \eta, g \oplus \xi \in P}.
\end{lemma}

The proof of this lemma is a direct verification. In particular, the Leibniz rule follows from the definition of the data \ec{([\,,\,]_{1},\D,\mathscr{K})}. Jacobi identity is derived from the conditions \eqref{EcPT1}--\eqref{EcPT3}.

\begin{lemma}\label{lema:GAP-PT}
Each admissible Poisson algebra $\PP$
induces a Poisson triple \ec{([\,,\,]_{1},\D^{\PP},\K^{\PP})} of $P$, where \ec{[\,,\,]_{1}} is the Lie bracket \eqref{EcLieBraP1} and \ec{\D^{\PP}} and \ec{\K^{\PP}} are defined by
    \begin{align*}
        \D^{\PP}_{\dd{f}}\eta &:= \pr_1\{\iota_{0}f, \iota_{1}\eta\}, \nonumber \\
        \K^{\PP}(\dd{f},\dd{g}) &:= \pr_1\{\iota_{0}f, \iota_{0}g\},
	\end{align*}
for all \ec{f,g \in P_{0}} and \ec{\eta \in P_{1}}. Moreover, the Poisson structure \ec{\{\,,\,\}} is expressed by this Poisson triple in the same way as \eqref{EcBracketPT}.
\end{lemma}

The Leibniz rule for \ec{\{\,,\,\}} ensures that \ec{\D^{\PP}} is a contravariant derivative on $P_{1}$ and \ec{\K^{\PP} \in \chi^{2}_{P_{0}}(\Omega^{1}_{P_{0}}; P_{1})}. The Jacobi identity implies that the triple \ec{([\,,\,]_{1},\D^{\PP},\K^{\PP})} satisfies conditions \eqref{EcPT1}--\eqref{EcPT3}. This completes the proof of the correspondence between the items \ref{Correspondence1} and \ref{Correspondence2} in Theorem \ref{teo:correspondenceGPA-PT-LA}.

\begin{corollary}\label{cor:P0KD0}
The subalgebra \ec{P_{0} \oplus \{0\}} of $P$ is a Poisson subalgebra of $\PP$ if and only if \ec{\K^{\PP} = 0}.
\end{corollary}

\begin{examplex}
For the admissible Poisson algebra $\PP$ defined by the bracket \eqref{eq:APmatrix}, the subalgebra \ec{P_{0} \oplus \{0\}} is a Poisson subalgebra of $\PP$ since the corresponding Poisson triple \eqref{eq:PTripleMatrix} of $P$ is such that \ec{\K^{\PP} = 0}.
\end{examplex}

As a consequence of Corollary \ref{cor:P0KD0} and condition \eqref{EcPT2}, admissible Poisson algebras for which \ec{P_{0} \oplus \{0\}} is a Poisson subalgebra induce flat contravariant derivatives on \ec{P_{1}}. Such a class of admissible Poisson algebras is given, for example, by Poisson modules (Section \ref{sec:PoissonModules}).

\paragraph{Poisson Triples and Lie Algebroids.} Now, we show that there is a one--to--one correspondence between Poisson triples and Lie algebroid structures on $A$ over \ec{P_{0}} with Lie bracket satisfying condition \eqref{EcA} and anchor map of the form \eqref{eq:Rho0Proj0}.

\begin{lemma}\label{lemma:PT-LA}
Each Poisson triple \ec{([\,,\,]_{1},\D,\mathscr{K})} of the commutative algebra $P$ induces a Lie algebroid structure on $A$ over \ec{P_{0}} with Lie bracket (satisfying condition \eqref{EcA}) given by
    \begin{equation}\label{eq:BracketLATriple}
        \cSch{\alpha \oplus \eta, \beta \oplus \xi} = \cSch{\alpha,\beta}_{P_{0}} \oplus \big( \D_{\alpha}\xi - \D_{\beta}\eta + [\eta,\xi]_{1} + \K(\alpha,\beta) \big),
    \end{equation}
for all \ec{\alpha \oplus \eta, \beta \oplus \xi \in A}, and anchor map \ec{\varrho = \varrho_0 \circ \mathrm{proj}_{0}}.
\end{lemma}
From the definition of the data \ec{([\,,\,]_{1},\D,\mathscr{K})} it follows that the bracket in \eqref{eq:BracketLATriple} is compatible with the Lie algebra morphism \ec{\varrho_0 \circ \mathrm{proj}_{0}} by means of the Leibniz rule \eqref{EcLeibnizAnchor}. Conditions \eqref{EcPT1}--\eqref{EcPT3} for Poisson triples of $P$ imply that the bracket in \eqref{eq:BracketLATriple} is a Lie bracket on $A$.

For the reciprocal, suppose that \ec{\A = (\cSch{\,,\,},\varrho_{P_{0}} \circ \mathrm{proj}_{0})} is a Lie algebroid structure on \ec{A} over \ec{P_{0}} satisfying condition \eqref{EcA}. Then, the Lie bracket \ec{\cSch{\,,\,}} of $\A$ can be represented as
    \begin{equation*}
        \cSch{A,A} = \cSch{\Omega^{1}_{P_{0}}, \Omega^{1}_{P_{0}}}_{P_{0}} \oplus \mathrm{proj}_{1}\cSch{A,A},
    \end{equation*}
where \ec{\mathrm{proj}_{1}:A \to P_{1}} is the canonical projection. Consequently, we have the following data: a \ec{P_0}--linear Lie bracket \ec{[\,,\,]_{1}^{\A}} on $P_{1}$,
    \begin{equation}\label{EcAbracket1}
        [\eta,\xi]_{1}^{\A} := \mathrm{proj}_{1}\cSch{\jmath_{1}\eta, \jmath_{1}\xi}, \quad \eta,\xi \in P_{1};
    \end{equation}
a contravariant derivative $\D^{\A}$ on $P_{1}$,
    \begin{equation}\label{EcDA}
        \hspace{1.75cm}\D^{\A}_{\alpha}\eta := \mathrm{proj}_{1}\,\cSch{\jmath_{0}\alpha, \jmath_{1}\eta}, \quad \alpha \in \Omega^{1}_{P_{0}},\, \eta \in P_{1};
    \end{equation}
and a \ec{P_{0}}--linear skew--symmetric mapping \ec{\mathscr{K}^{\A}: \Omega^{1}_{P_{0}} \times \Omega^{1}_{P_{0}} \to P_{1}},
    \begin{equation}\label{EcKA}
        \K^{\A}(\alpha,\beta) := \mathrm{proj}_1\,\cSch{\jmath_{0}\alpha, \jmath_{0}\beta}, \quad \alpha,\beta \in \Omega^{1}_{P_{0}}.
    \end{equation}
Here, \ec{\jmath_{0}: \Omega^{1}_{P_{0}} \hookrightarrow \Omega^{1}_{P_{0}} \oplus \{0\}} and \ec{\jmath_{1}: P_{1} \hookrightarrow \{0\} \oplus P_{1}} are the inclusion maps.

\begin{lemma}\label{lema:TripleLA}
The data \ec{([\,,\,]_{1}^{\A},\D^{\A},\K^{\A})} defined in \eqref{EcAbracket1}--\eqref{EcKA} is a Poisson triple of $P$.
\end{lemma}
This completes the proof of the Theorem \ref{teo:correspondenceGPA-PT-LA}.

    \section{Cochain Complexes}
        \label{sec:ComplexAPA}

In this section, we show that admissible Poisson algebras have associated two natural cochain complexes and describe some of their properties. These complexes play an important role in the computation of the first Poisson cohomology (Section \ref{sec:FirstCohomology}).

For an admissible Poisson algebra $\PP$, we denote the center of the associated Lie algebra \ec{(P_{1},[\,,\,]_{1})} by
    \begin{equation*}
        Z_{\PP}(P_{1}) := \{ \xi \in P_{1} \mid [\xi, \cdot]_{1} = 0 \}.
    \end{equation*}
Note that the center is also a $P_{0}$--module.

\begin{theorem}\label{teo:CochainComplexes}
Every admissible Poisson algebra \ec{\PP=(P=P_{0} \ltimes P_{1}, \{\,,\,\})} induces two cochain complexes:
    \begin{equation}\label{eq:Complex}
        \big( \overline{\Gamma}^{\ast}_{\PP} := \oplus_{k}\ \overline{\Gamma}_{P_{0}}^{k} \equiv \overline{\Gamma}_{P_{0}}^{k}\big( \Omega_{P_{0}}^{1}; {Z}_{\PP}(P_{1}) \big),\, \overline{\partial}_{\D} \big)
    \end{equation}
and
\begin{equation}\label{eq:ComplexBar}
        \big( \Gamma^{\ast}_{\PP} := \oplus_{k}\ \Gamma_{R}^{k} \equiv \Gamma_{R}^{k}\big( P_{0}; {Z}_{\PP}(P_{1}) \big),\, \partial_{\D} \big),
    \end{equation}
where \ec{([\,,\,]_{1},\D,\K)} is the Poisson triple of $P$ corresponding to $\PP$, and the coboundary operators $\overline{\partial}_{\D}$ and $\partial_{\D}$ are defined in \eqref{eq:Partial} and \eqref{eq:PartialBar}, respectively, and associated to the contravariant derivative $\D$ and the Poisson algebra \ec{\PP_{0}=(P_{0},\cdot,\{\,,\,\}_{0})}.
\end{theorem}
\begin{proof}
By Theorem \ref{teo:correspondenceGPA-PT-LA}, the Poisson algebra $\PP$ is in correspondence with a Poisson triple  \ec{([\,,\,]_{1},\D,\K)} of $P$. Then, the claim follows from conditions \eqref{EcPT1}--\eqref{EcPT2} for Poisson triples of $P$ and Proposition \ref{prop:TwoCochainCmplx}.
\end{proof}

Note that we have two extreme cases: first, if \ec{Z_{\PP}(P_{1})=\{0\}}, then \ec{\overline{\Gamma}^{k}_{P_{0}} = \{0\} = \Gamma^{k}_{R}}, for all $k$. Second, if \ec{Z_{\PP}(P_{1})=P_{1}}, then we have the following:

\begin{corollary}\label{cor:Partialisd}
If the Lie algebra \ec{(P_{1},[\,,\,]_{1})} is abelian, \ec{[\,,\,]_{1}=0}, then the cochain complexes \eqref{eq:Complex} and \eqref{eq:ComplexBar} just coincides, respectively, with the cochain complexes
    \begin{equation}\label{eq:CochainAbelian}
        (\chi^{\ast} := \oplus_{k}\ \chi_{P_{0}}^{k} \equiv \chi_{P_{0}}^{k}(\Omega^{1}_{P_{0}}; P_{1}),\, \overline{\dd}_{\D}) \quad \text{and} \quad \big( \X{}^{\ast} := \oplus_{k}\ \X{R}^{k} \equiv \X{R}^{k}(P_{0}; P_{1}),\, \dd_{\D} \big),
    \end{equation}
induced by the coboundary operators \ec{\overline{\dd}_{\D}} in \eqref{eq:dD} and \ec{\dd_{\D}} in \eqref{eq:dDbar} associated to $\D$ and $\PP_{0}$. In particular, we have that $\D$ is flat and \ec{\K \in \chi^{2}_{P_{0}}} is a 2--cocycle of $\overline{\dd}_{\D}$.
\end{corollary}

Moreover, taking into account condition \eqref{EcPT2} and the property \eqref{eq:dK0} for $\K$, we get the following consequence of Theorem \ref{teo:CochainComplexes}.

\begin{corollary}\label{cor:Ktrivial}
Let $\PP$ be an admissible Poisson algebra and \ec{([\,,\,]_{1},\D,\K)} the corresponding Poisson triple of $P$. Then, $\K$ belongs to \ec{\overline{\Gamma}^{2}_{P_{0}}(\Omega^{1}_{P_{0}};Z_{\PP}(P_{1}))} and is a 2--cocycle of the cochain complex \eqref{eq:Complex} if and only if the contravariant derivative $\D$ is flat.
\end{corollary}

Now, consider the collection of $P_{0}$--linear mappings \ec{\Delta \equiv \Delta_{k}: \overline{\Gamma}^{k}_{P_{0}} \to \Gamma^{k}_{R}} defined by
    \begin{equation}\label{eq:Delta}
        (\Delta \overline{Q})(f_{1},\ldots,f_{k}) := \overline{Q}(\dd{f_{1}},\ldots,\dd{f_{k}}), \quad \text{for}\ \overline{Q} \in \overline{\Gamma}^{k}_{P_{0}}.
    \end{equation}
It is easy to see that $\Delta$ is well--defined and injective. Moreover, the following diagram is commutative:
    \begin{equation*}
        \xymatrix{
            \overline{\Gamma}^{k}_{P_{0}} \ar[r]^{\Delta} \ar[d]_{\overline{\partial}_{\D}} & \Gamma^{k}_{R} \ar[d]^{\partial_{\D}} \\
            \overline{\Gamma}^{k+1}_{P_{0}} \ar[r]^{\Delta} & \Gamma^{k+1}_{R}}
    \end{equation*}
Hence, we have a cochain complex morphism.

\begin{proposition}\label{prop:DeltaAst}
Let $\PP$ be an admissible Poisson algebra. Then, there exist a natural linear mapping in cohomology
    \begin{equation*}
        \Delta^{\ast}: \mathrm{H}^{k}(\overline{\Gamma}^{\ast}_{\PP}) \to \mathrm{H}^{k}(\Gamma^{\ast}_{\PP}),
    \end{equation*}
between the cohomology groups of the cochain complexes \eqref{eq:Complex} and \eqref{eq:ComplexBar}, defined by \ec{\Delta^{\ast}[\bar{Q}] := [\Delta \bar{Q}]}. Moreover, we have an exact sequence
    \begin{equation*}
        0 \to \mathrm{H}^{1}\big( \overline{\Gamma}^{\ast}_{\PP}) \to
              \mathrm{H}^{1}\big( \Gamma^{\ast}_{\PP}
              ).
    \end{equation*}
\end{proposition}

\begin{remark}
The family of morphisms in \eqref{eq:Delta} is well known in a geometric framework for the case of multivector fields and multi--derivations on a smooth manifold \cite{Duf05,Lau12}.
\end{remark}

    \section{Gauge Transformations and Deformations}\label{sec:equivalenceGPA}

Here, we introduce some transformations that preserve the class of admissible Poisson algebras, and then we formulate some Poisson equivalence criteria and present a way to construct new one-parametric admissible Poisson algebras from a given one.

Let \ec{\phi:P \to P} be an $R$--linear mapping. Since \ec{P=P_{0} \oplus P_{1}}, we have that $\phi$ admits a unique representation
    \begin{equation}\label{EcPhi}
        \phi(f \oplus \eta) = \big( \phi_{00}f + \phi_{01}\eta \big) \oplus \big( \phi_{10}f + \phi_{11}\eta \big), \quad f \oplus \eta \in P.
    \end{equation}
Here, \ec{\phi_{00}: P_{0} \rightarrow P_{0},\ \phi_{01}: P_{1} \rightarrow P_{0},\ \phi_{10}: P_{0} \rightarrow P_{1}} and \ec{\phi_{11}: P_{1}\rightarrow P_{1}} are the $R$--linear mappings defined by \ec{\phi_{ij} := \pr_i \circ \phi \circ \iota_j}, with \ec{i,j=0,1}. We will also write
    \begin{equation}\label{EcPhiMtx}
        \phi = \left(
                 \begin{array}{cc}
                    \phi_{00} & \phi_{01} \\
                    \phi_{10} & \phi_{11} \\
                 \end{array}
               \right).
    \end{equation}

\begin{lemma}\label{lemma:Endo}
An $R$--linear mapping \ec{\phi: P \to P} is an \emph{endomorphism} of the commutative algebra $P$, that is, \ec{\phi(\pi_{1}\,\pi_{2}) = \phi(\pi_{1})\phi(\pi_{2})} for all \ec{\pi_{1},\pi_{2} \in P}, if and only if
    \begin{itemize}
    \item the linear mapping \ec{\phi_{00}} is an endomorphism of the commutative algebra \ec{P_{0}}: \ec{\phi_{00}(fg) = (\phi_{00}f)(\phi_{00}g)};
    \item the linear mapping \ec{\phi_{01}} is \ec{\phi_{00}}--semilinear and its image is zero--squared: \ec{\phi_{01}(f\eta) = (\phi_{00}f)(\phi_{01}\eta)} and \ec{(\phi_{01}\eta)(\phi_{01}\xi) = 0}, respectively;
    \item the linear mapping \ec{\phi_{10}} is a \ec{\phi_{00}}--derivation of \ec{P_0}:
                \begin{equation}\label{eq:Endo3}
                    \phi_{10}(fg) = (\phi_{00}f)(\phi_{10}g) + (\phi_{00}g)(\phi_{10}f);
                \end{equation}
    \item the linear mapping \ec{\phi_{11}} is compatible with \ec{\phi_{00}, \phi_{01}} and \ec{\phi_{10}} in the following sense: $\phi_{11}(f\eta) = (\phi_{00}f)(\phi_{11}\eta) + (\phi_{01}\eta)(\phi_{10}f)$ and \ec{(\phi_{01}\eta)(\phi_{11}\xi) + (\phi_{01}\xi)(\phi_{11}\eta) = 0};
    \end{itemize}
for all \ec{f,g\in P_0} and \ec{\eta,\xi\in P_{1}}.
\end{lemma}

\begin{remark}
If \ec{\phi_{00}} is the identity mapping, then condition \eqref{eq:Endo3} is equivalent to \ec{\phi_{10}} belongs to \ec{\mathrm{Der}(P_{0};P_{1})} defined in \eqref{eq:DerP0P1}. If, in addition \ec{\phi_{11}} is the identity mapping, then \ec{\phi_{01}} belongs to \ec{\varepsilon(P_1; P_{0})} defined in \eqref{epsilon}.
\end{remark}

\paragraph{Gauge Equivalence.} Let \ec{\PP=(P,\{\,,\,\})} be an admissible Poisson algebra. Suppose that \ec{\PP'=(P,\{\,,\,\}')} is a Poisson algebra isomorphic to $\PP$ via a Poisson isomorphism \ec{\phi:\PP' \to \PP}, in the sense of \eqref{EcPhiAlg}. Then,
    \begin{equation}\label{EcBracketPrima}
        \big\{\pi_1,\pi_2\big\}' = \phi^{-1}\big\{\phi(\pi_1),\phi(\pi_2)\big\}, \quad \pi_1,\pi_2 \in P.
    \end{equation}
However, this relation \emph{does not} imply that $\PP'$ is an \emph{admissible} Poisson algebra, in general.

\begin{lemma}\label{lema:BracketGauge}
If \ec{\phi_{01}=0} and  \ec{\phi_{00}} is invertible, then \ec{(P,\{\,,\,\}')} is an admissible Poisson algebra.
\end{lemma}
\begin{proof}
From the hypotheses, it follows that \ec{\pr_{0}\{f \oplus \eta, g \oplus \xi\}' = \phi^{-1}_{00}\{\phi_{00}f, \phi_{00}g\}_{0}}, for all \ec{f \oplus \eta, g \oplus \xi \in P}. Since $\phi$ is a Poisson morphism, condition \ec{\phi_{01}=0} implies that \ec{\phi_{00}} is a Poisson morphism of the Poisson algebra \ec{\PP_{0}=(P_{0},\cdot,\{\,,\,\}_{0})}, and the lemma follows.
\end{proof}

Taking into account Lemma \ref{lemma:Endo}, this last fact leads to the following:

\begin{definition}\label{def:gauget}
By a \emph{gauge transformation} on the trivial extension algebra \ec{P=P_{0} \ltimes P_{1}} we mean an $R$--linear mapping \ec{\phi:P\to P} such that
    \begin{enumerate}[label=(\roman*)]
      \item \label{Gauge2} the linear mapping \ec{\phi_{01}} is trivial, \ec{\phi_{01}=0};
      \item \label{Gauge1} the linear mapping \ec{\phi_{00}} is a Poisson algebra isomorphism of \ec{\PP_{0}=(P_{0},\cdot,\{\,,\,\}_{0})};
      \item \label{Gauge3} the linear mapping \ec{\phi_{10}} is a \ec{\phi_{00}}--derivation of \ec{P_0}, in the sense of \eqref{eq:Endo3};
      \item \label{Gauge4} the pair \ec{(\phi_{11},\phi_{00})} is a
            $P_{0}$--module isomorphism of $P_{1}$.
    \end{enumerate}
\end{definition}

Clearly, the identity mapping on $P$ is a gauge transformation. Moreover, by Lemma \ref{lemma:Endo} and direct computations, one can show that:
    \begin{enumerate}[label=(\alph*)]
        \item \label{GaugeIso} Every gauge transformation is an $R$--algebra isomorphism.
        \item \label{GaugeInverse} The inverse of a gauge transformation is also a gauge transformation. Indeed, if \ec{\phi:P \to P} is a gauge transformation, then its inverse is given by
                \begin{equation}\label{eq:phi-1}
                    \phi^{-1} = \left(
                             \begin{array}{cc}
                                \phi_{00}^{-1} & 0 \\
                                -\phi^{-1}_{11} \circ \phi_{10} \circ \phi^{-1}_{00} & \phi_{11}^{-1} \\
                             \end{array}
                           \right).
                \end{equation}

        \item \label{GaugeClosed} Gauge transformations are closed under the composition.
    \end{enumerate}
Hence, it follows from the items \ref{GaugeIso}--\ref{GaugeClosed} that the set of all gauge transformations on $P$ is a group.

\begin{proposition}\label{prop:GaugeInduces}
If \ec{\PP'=(P,\{\,,\,\}')} is a Poisson algebra isomorphic to an admissible Poisson algebra \ec{\PP=(P,\{\,,\,\})} via a gauge transformation \ec{\phi:\PP' \to \PP}, then $\PP'$ is an admissible Poisson algebra. Moreover, the Poisson triples \ec{([\,,\,]_{1}',\D',\K')} and \ec{([\,,\,]_{1},\D,\K)} of $P$ corresponding to $\PP'$ and $\PP$, respectively, are related by the formulas
    \begin{align}
        [\eta,\xi]_{1}' &= \phi^{-1}_{11}[\phi_{11}\eta,\phi_{11}\xi]_{1}, \label{EcPT1prima} \\[0.15cm]
        \D'_{\dd{f}}\eta &= \phi^{-1}_{11}\big( \D_{\dd(\phi_{00}f)}\phi_{11}\eta + [\phi_{10}f,\phi_{11}\eta]_{1} \big), \label{EcPT2prima} \\[0.15cm]
        \K'(\dd{f},\dd{g}) &= \phi^{-1}_{11}\big( \K(\dd\phi_{00}f,\dd\phi_{00}g) + \D_{\dd(\phi_{00}f)}\phi_{10}g - \D_{\dd(\phi_{00}g)}\phi_{10}f
        + [\phi_{10}f,\phi_{10}g]_{1} - \phi_{10}\{f,g\}_0 \big), \label{EcPT3prima}
    \end{align}
for all \ec{f,g \in P_{0}} and \ec{\eta,\xi \in P_{1}}.
\end{proposition}
\begin{proof}
By definition of gauge transformation, the claims follow from the items \ref{GaugeIso}--\ref{GaugeClosed}, Lemma \ref{lema:BracketGauge} and long and tricky computations based on conditions \eqref{EcPT1}--\eqref{EcPT3} for Poisson triples of $P$.
\end{proof}

\begin{corollary}\label{cor:GaugeInduces}
Let \ec{(P, \{\,,\,\})} be an admissible Poisson algebra. Then, every gauge transformation $\phi$ on $P$ induces an admissible Poisson structure on $P$ by means of the formula \eqref{EcBracketPrima}, which is $\phi$--equivalent to the initial one \ec{\{\,,\,\}} and with corresponding Poisson triple of $P$ given by the formulas \eqref{EcPT1prima}--\eqref{EcPT3prima}.
\end{corollary}

In other words, a main property of gauge transformations is that they allow us to derive equivalent admissible Poisson structures from a given one.

Now, we arrive at one of the main results of this section.

\begin{theorem}\label{teo:GaugeEquiv}
Let $\PP'$ and $\PP$ be two admissible Poisson algebras. Then, the following assertions are equivalent:
    \begin{enumerate}[label=(\roman*)]
      \item \label{Equiv1} The admissible Poisson algebras $\PP'$ and $\PP$ are isomorphic via a gauge transformation on $P$.
      \item \label{Equiv2} The Poisson triples of $P$ corresponding to $\PP'$ and $\PP$ are related by formulas \eqref{EcPT1prima}--\eqref{EcPT3prima}, for some gauge transformation on $P$.
      \item \label{Equiv3} The Lie algebroids associated to $\PP'$ and $\PP$  are isomorphic via a Lie algebroid isomorphism induced by a gauge transformation on $P$.
    \end{enumerate}
\end{theorem}
\begin{proof}
Taking into account formulas \eqref{EcBracketPrima} and \eqref{eq:phi-1} and Proposition \ref{prop:GaugeInduces}, the equivalence of the items \ref{Equiv1} and \ref{Equiv2} follows from Lemmas \ref{LemaTripleToAlg} and \ref{lema:GAP-PT}. Now, we show that each gauge transformation \ec{\phi: P \to P} induces an isomorphism of Lie algebroids structures on $A$ in \eqref{A} over \ec{P_0} satisfying \eqref{EcA} and \eqref{eq:Rho0Proj0}: the (module) isomorphism is given by the tuple \ec{(\varphi,\phi_{00})}, where \ec{\varphi: A \to A} is defined by
    \begin{equation*}
        \varphi (g\dd f \oplus \eta) := (\phi_{00}g\,\dd\phi_{00}f) \oplus (\phi_{00}g \cdot \phi_{10}f + \phi_{11}\eta), \quad
        f, g \in P_0;\ \eta \in P_{1}.
    \end{equation*}
Then, the equivalence of the items \ref{Equiv2} and \ref{Equiv3} follows from Lemmas \ref{lemma:PT-LA} and \ref{lema:TripleLA}.
\end{proof}

From Proposition \ref{prop:GaugeInduces} and Lemma \ref{lema:IsoCohomo}, we deduce the following:

\begin{corollary}\label{cor:Lie1IsoGauge}
If two admissible Poisson algebras $\PP'$ and $\PP$ are isomorphic via a gauge transformation on $P$, then the corresponding Lie algebras \ec{(P_{1}, [\,,\,]_{1}')} and \ec{(P_{1}, [\,,\,]_{1})}, and first Poisson cohomologies are isomorphic.
\end{corollary}

Now, note that every $P_{0}$--valued derivation \ec{\phi_{10}:P_{0} \to P_{1}} in \eqref{eq:DerP0P1} induces a gauge transformation on $P$ by the formula
    \begin{equation}\label{eq:GaugePhi10}
        \phi = \left(
                 \begin{array}{cc}
                    \mathrm{id}_{P_{0}} & 0 \\
                    \phi_{10} & \mathrm{id}_{P_{1}} \\
                 \end{array}
               \right).
    \end{equation}
In particular, we introduce the following important subclass of gauge transformations.

\begin{definition}\label{def:MuGauge}
Let \ec{\mu:\Omega^1_{P_0} \to P_1} be a \ec{P_{0}}--linear mapping. By a \emph{$\mu$--gauge transformation} on $P$ we mean an $R$--linear mapping \ec{\phi=\phi_{\mu}} in \eqref{eq:GaugePhi10} such that
    \begin{equation*}
        \phi_{10}(f) = \mu(\dd{f}), \quad f \in P_{0}.
    \end{equation*}
\end{definition}

It is clear that a $\mu$--gauge transformation satisfies conditions \ref{Gauge2}--\ref{Gauge4} of Definition \ref{def:gauget} and hence, by Corollary \ref{cor:GaugeInduces}, such transformations preserve the family of admissible Poisson structures. Moreover, the set of all $\mu$--gauge transformation is an abelian group. This claim follows from the formula for composition \ec{\phi_{\mu_{1}} \circ \phi_{\mu_{2}} = \phi_{\mu_{1} + \mu_{2}}}. We remark that this class of gauge transformations is used in the study of the Hamiltonization problem for linearized dynamics along Poisson submanifolds \cite{RuGaVo20}.

\paragraph{Equivalence Criterion Via $\mathbf{\K}$.} Taking into account Corollaries \ref{cor:Lie1IsoGauge} and \ref{cor:Ktrivial}, consider two admissible Poisson algebras $\PP$ and $\PP'$ inducing the same $P_{0}$--linear Lie bracket and the same flat contravariant derivative on $P_{1}$. That is, such that the corresponding Poisson triples of $P$ are given, respectively, by
    \begin{equation}\label{eq:LieDKflat}
        \big( [\,,\,]_{1},\D,\K \big) \quad \text{and} \quad \big( [\,,\,]'_{1} = [\,,\,]_{1}, \D' = \D, \K' \big), \quad \text{with} \quad \mathrm{Curv}^{\D} = 0.
    \end{equation}
Note that this hypothesis implies that $\PP$ and $\PP'$ induce the same cochain complexes \ec{(\overline{\Gamma}^{\ast}_{\PP}, \overline{\partial}_{\D})} in \eqref{eq:Complex} and \ec{(\Gamma^{\ast}_{\PP}, \partial_{\D})} in \eqref{eq:ComplexBar}. Hence, we have the following:

\begin{theorem}\label{teo:Kequivalence}
Two admissible Poisson algebras $\PP$ and $\PP'$ satisfying \eqref{eq:LieDKflat} are isomorphic by means of a gauge transformation of the form \eqref{eq:GaugePhi10} if and only if the $\partial_{\D}$--cohomology class of the 2--cocycle \ec{\Delta(\K - \K')} is trivial. Here, the $P_{0}$--linear mapping $\Delta$ is defined in \eqref{eq:Delta}.
\end{theorem}
\begin{proof}
By the hypothesis \eqref{eq:LieDKflat} and Corollary \ref{cor:Ktrivial}, we have that $\K$ and $\K'$ define classes in the second cohomology group of \ec{(\overline{\Gamma}^{\ast}_{\PP},\overline{\partial}_{\D})} since $\D$ is flat. So, by Proposition \ref{prop:DeltaAst}, the 2--cocycle \ec{\Delta(\K - \K')} of \ec{(\Gamma^{1}_{R}, \partial_{\D})} defines a class in the second cohomology group. Hence, the theorem follows from the definition of $\Delta$, the transition rules \eqref{EcPT1prima}--\eqref{EcPT3prima} for Poisson triples of $P$ under a gauge transformation and Theorem \ref{teo:GaugeEquiv}. Indeed, if $\PP$ and $\PP'$ are isomorphic via a gauge transformation of the form \eqref{eq:GaugePhi10}, then \ec{\phi_{10} \in \Gamma^{1}_{R}} and \ec{\Delta(\K-\K') = - \partial_{\D}\phi_{10}}. Reciprocally, if \ec{\Delta(\K-\K') = \partial_{\D}Q} for some \ec{Q \in \Gamma^{1}_{R}}, then we take \ec{\phi_{10}=-Q}.
\end{proof}

In particular, taking into account Corollary \ref{cor:P0KD0}, we deduce the following fact.

\begin{corollary}\label{cor:FlatP0Sub}
Let $\PP$ be an admissible Poisson algebra with corresponding Poisson triple \ec{([\,,\,]_{1}, \D, \K)} such that $\D$ is flat. If the $\partial_{D}$--cohomology class of the 2--cocycle \ec{\Delta\K} is trivial, then $\PP$ is isomorphic to an admissible Poisson algebra for which \ec{P_{0} \oplus \{0\}} is a Poisson subalgebra.
\end{corollary}

Recall that for admissible Poisson algebras induced by Poisson modules it holds that \ec{P_{0} \oplus \{0\}} is a Poisson subalgebra (Section \ref{sec:PoissonModules}).

\begin{examplex}
Consider the admissible Poisson algebra \ec{\PP=(P=P_{0} \ltimes P_{1},\{\,,\,\})} defined by the bracket \eqref{eq:APmatrix}. Then, the following Poisson structures
     \begin{equation*}
        \{f\oplus M, g\oplus N\}_{X} = \{f,g\}_{0} \oplus \big( D_{\dd{f}}N- D_{\dd{g}}M + [|M,N|] + ( \{Xf,g\}_{0} + \{f,Xg\}_{0} - X\{f,g\}_{0} )\,I\,\big), \quad X \in \mathrm{Der}(P_{0}),
    \end{equation*}
define a \ec{\mathrm{Der}(P)}--parameterized family of admissible Poisson algebras \ec{\PP_{X}=(P,\{\,,\,\}_{X})} isomorphic to the initial one $\PP$. Clearly, \ec{\PP_{X}=\PP} for all \ec{X \in \mathrm{Poiss}(P_{0}, \cdot, \{\,,\,\}_{0})}. Here, $I$ is the identity matrix.
\end{examplex}

\paragraph{Deformations of APA's.} Let \ec{\PP=(P,\{\,,\,\})} be an admissible Poisson algebra, \ec{([\,,\,]_{1},\D,\K)} the corresponding Poisson triple of $P$ and \ec{(\overline{\Gamma}^{\ast}_{\PP} = \oplus_{k}\overline{\Gamma}^{k}_{P_{0}}, \overline{\partial}_{\D})} the associated cochain complex given in \eqref{eq:Complex}.

\begin{theorem}\label{teo:NewAdmissible}
Every 2--cocycle \ec{\mathscr{C} \in \overline{\Gamma}^{2}_{P_{0}}} of the cochain complex \ec{(\overline{\Gamma}^{\ast}_{\PP}, \overline{\partial}_{\D})} induces a $t$--parameterized family \ec{\PP^{t}=(P,\{\,,\,\}^{t})} of admissible Poisson algebras defined by
    \begin{equation}\label{eq:tFamily}
        \{f \oplus \eta, g \oplus \xi\}^{t} := \{f \oplus \eta, g \oplus \xi\} \,+\, 0 \oplus
        (t\mathscr{C})(\dd{f}, \dd{g}), \quad t \in R, \\
    \end{equation}
with \ec{f \oplus \eta, g \oplus \xi \in P}. In particular, we have \ec{\PP = \PP^{t=0}}. Moreover, there exists a $t$--parameterized family of $\mu$--gauge transformations inducing Poisson isomorphisms between the admissible Poisson algebras \ec{\PP^{t}} and $\PP$ if and only if the $\overline{\partial}_{\D}$--cohomology class of \ec{\mathscr{C}} is trivial.
\end{theorem}
\begin{proof}
Taking into account that the Poisson triple corresponding to \ec{\PP^{t}} is given by
    \begin{equation*}
        \big( [\,,\,]_{1}, \D, \K + t\mathscr{C} \big),
    \end{equation*}
the result follows by adapting the proof of Theorem \ref{teo:Kequivalence} to the case when $\D$ is not necessarily flat and \ec{\K'-\K \in \overline{\Gamma}^{2}_{P_{0}}}. Indeed, if for every $t$ we have that \ec{\PP^{t}} and $\PP$ are isomorphic via a $\mu_{t}$--gauge transformation \ec{\phi_{\mu_{t}}:\PP^{t} \to \PP}, then \ec{\mu_{t} \in \overline{\Gamma}^{1}_{P_{0}}} and \ec{t\mathscr{C} = \overline{\partial}_{\D}\mu_{t}}. In particular, \ec{\mathscr{C} = \overline{\partial}_{\D}\mu_{1}}. Conversely, if \ec{\mathscr{C} = \overline{\partial}_{\D}\vartheta} for some \ec{\vartheta \in \overline{\Gamma}^{1}_{P_{0}}}, then we take \ec{\mu_{t} = t\vartheta}.
\end{proof}

We can think of the family of brackets defined in \eqref{eq:tFamily} as a one-parametric deformation of the initial admissible Poisson structure \ec{\{\,,\,\}}. Hence, Theorem \ref{teo:NewAdmissible} says that this family defines an \emph{exact (infinitesimal) deformation} of \ec{\{\,,\,\}} by means of 2--cocycles of the cochain complex \ec{(\overline{\Gamma}^{\ast}_{\PP}, \overline{\partial}_{\D})} associated to $\PP$. Moreover, it gives sufficient cohomological conditions for the deformation to be \emph{trivial}.

\begin{corollary}\label{cor:Gamma1Admissible}
For every admissible Poisson algebra $\PP$ there exists a $t$--parameterized family of admissible Poisson algebras \ec{(P,\{\,,\,\}^{t})} isomorphic to $\PP$ given by
    \begin{equation*}
        \{f \oplus \eta, g \oplus \xi\}^{t}_{C} = \{f,g\}_{0} \oplus \big( \D_{\dd{f}}\big( \xi + tC(\dd{g}) \big) - \D_{\dd{g}}\big( \eta + tC(\dd{f}) \big) - tC\big( \dd\{f,g\}_{0} \big) + [\eta,\xi]_{1} + \K(\dd{f},\dd{g}) \big),
    \end{equation*}
with \ec{C \in \Gamma^{1}_{P_{0}}} arbitrary and \ec{f \oplus \eta, g \oplus \xi \in P}.
\end{corollary}

In other words, every $P_{0}$--linear mapping \ec{\Omega^{1}_{P_{0}} \to Z_{\PP}(P_{1})} defines \emph{exact trivial deformations} of admissible Poisson structures.

    \section{Poisson and Hamiltonian Derivations}\label{sec:hamder}

In this section, we present detailed descriptions of Poisson and Hamiltonian derivations of admissible Poisson algebras with some further applications to the computation of the first Poisson cohomology (Section \ref{sec:FirstCohomology}).

Let \ec{X:P \to P} be an $R$--linear mapping. Then, similar to \eqref{EcPhi} and \eqref{EcPhiMtx}, $X$ admits a unique representation \ec{X(f \oplus \eta) = (X_{00}f + X_{01}\eta) \oplus (X_{10}f + X_{11}\eta)}, for \ec{f \oplus \eta \in P}. Or, equivalently, one can write
    \begin{equation}\label{EcX}
        X = \left(
                 \begin{array}{cc}
                    X_{00} & X_{01} \\
                    X_{10} & X_{11} \\
                 \end{array}
               \right).
    \end{equation}

\begin{lemma}\label{lemma:Xpreserving}
An $R$--linear mapping \ec{X: P \to P} preserves
    \begin{enumerate}[label=(\alph*)]
      \item the commutative algebra \ec{P_{0}} if and only if \ec{X_{10}} vanishes:
                \begin{equation}\label{eq:X10Vanish}
                    X(P_{0} \oplus \{0\}) \subseteq P_{0} \oplus \{0\}
                        \quad \Longleftrightarrow \quad
                    X = \left(
                     \begin{array}{cc}
                        X_{00} & X_{01} \\
                        0 & X_{11} \\
                     \end{array}
                   \right);
                \end{equation}
      \item the \ec{P_{0}}--module \ec{P_{1}} if and only if \ec{X_{01}} vanishes:
                \begin{equation}\label{eq:X01Vanish}
                    X(\{0\} \oplus P_{1}) \subseteq \{0\} \oplus P_{1}
                        \quad \Longleftrightarrow \quad
                    X = \left(
                     \begin{array}{cc}
                        X_{00} & 0 \\
                        X_{10} & X_{11} \\
                     \end{array}
                   \right).
                \end{equation}
    \end{enumerate}
\end{lemma}

Let us introduce the set
    \begin{equation}\label{epsilon}
        \varepsilon(P_1; P_{0}) := \big\{ T: P_{1} \to P_{0} \mid T(f\eta) = f\,T(\eta) \,\ \text{and} \,\ T(\xi)\eta + T(\eta)\xi = 0,\ \text{for all}\ f \in P_{0};\ \eta,\xi \in P_{1} \big\},
    \end{equation}
and the \ec{P_{0}}--module of $R$--linear \ec{P_{1}}--valued derivations of \ec{P_{0}}
    \begin{equation}\label{eq:DerP0P1}
        \mathrm{Der}(P_{0}; P_{1}) := \{D: P_{0} \to P_{1} \mid D(fg) = fDg + gDf,\ \text{for all}\ f,g \in P_{0}\}.
    \end{equation}
We remark that the set \ec{\varepsilon(P_1; P_{0})} is non--empty, in general \cite{Zhu19}.

\begin{lemma}\label{lema:PD}
An $R$--linear mapping \ec{X:P \to P} is a derivation of the commutative algebra $P$ iff
    \begin{enumerate}[label=(\roman*)]
        \item \label{Der1} \ec{X_{00}} is a derivation of \ec{P_{0}}, \ec{X_{00} \in \mathrm{Der}(P_{0})};
        \item \label{Der2} \ec{X_{01}\in \varepsilon(P_1;P_{0})};
        \item \label{Der3} \ec{X_{10}} is a \ec{P_{1}}--valued derivation of \ec{P_{0}}, \ec{X_{10} \in \mathrm{Der}(P_{0};P_{1})}; and
        \item \label{Der4} \ec{X_{11}} is a \emph{generalized \ec{X_{00}}--derivation} of \ec{P_{1}},
                \begin{equation*}
                    X_{11}(f\eta) = fX_{11}\eta + (X_{00}f)\eta;
                \end{equation*}
    \end{enumerate}
for all \ec{f,g\in P_{0}} and \ec{\eta\in P_{1}}.
\end{lemma}
\begin{proof}
By definition, $X$ is a derivation of $P$ if and only if
    \begin{equation}\label{eq:XderProof}
        X(\pi_{1}\,\pi_{2}) = X(\pi_{1})\pi_{2} + \pi_{1}X(\pi_{2}),
    \end{equation}
for any \ec{\pi_{1},\pi_{2} \in P}. In particular, taking into account the product \eqref{EcPproduct}, we have that if \ec{\pi_{1} = f \oplus 0} and \ec{\pi_{2} = g \oplus 0}, then condition \eqref{eq:XderProof} implies the items \ref{Der1} and \ref{Der3}. If \ec{\pi_{1} = 0 \oplus \eta} and \ec{\pi_{2} = 0 \oplus \xi}, then \eqref{eq:XderProof} implies \ref{Der2}. If \ec{\pi_{1} = f \oplus 0} and \ec{\pi_{2} = 0 \oplus \eta}, then \eqref{eq:XderProof} implies \ref{Der4}. Reciprocally, by straightforward computations, one can show that conditions \ref{Der1}--\ref{Der4} for $X$ imply \eqref{eq:XderProof}.
\end{proof}

\begin{corollary}\label{cor:DerP1toP0}
Every derivation of the trivial extension algebra \ec{P_{0} \ltimes P_{1}} induces a derivation of the commutative algebra $P_{0}$.
\end{corollary}

In the remainder of this section, we assume that \ec{\PP=(P,\{\,,\,\})} is an admissible Poisson algebra and \ec{([\,,\,]_{1},\D,\K)} is the corresponding Poisson triple of $P$. Also, for each \ec{f \oplus \eta \in P}, we define the following element of \ec{\mathrm{Der}(P_{0};P_{1})} associated to $\PP$:
    \begin{equation}\label{EcTorsionDef}
    	\mathscr{T}_{f \oplus \eta} \equiv \mathscr{T}^{\PP}_{f \oplus \eta} \,:=\, \K\big(\dd{f}, \dd(\,\cdot\,)\big) - \D_{\dd (\cdot)}\eta.
    \end{equation}

\paragraph{Characterization of Casimir Elements.}
First, we denote the kernel of the $R$--linear mapping \ec{f \oplus \eta \mapsto \mathscr{T}_{f \oplus \eta}} by
    \begin{equation*}
        \ker{\mathscr{T}} = \big\{f \oplus \eta \in P \mid \mathscr{T}_{f \oplus \eta} = 0\big\}.
    \end{equation*}

\begin{proposition}\label{prop:Casim}
An element \ec{k \oplus \zeta \in P} is a Casimir element of \ec{\PP} if and only if
    \begin{equation}\label{eq:CasimDescrib}
        k \oplus \zeta \in \big( \mathrm{Casim}(\PP_{0}) \oplus P_{1} \big) \cap \ker{\mathscr{T}}
        \quad \text{and} \quad \D_{\dd{k}} + [\zeta,\cdot]_{1} = 0.
    \end{equation}
\end{proposition}

Note that the second condition in \eqref{eq:CasimDescrib} is well--defined since \ec{\D_{\dd{k}}:P_{1} \to P_{1}} is a \ec{P_{0}}--linear mapping for every \ec{k \in \mathrm{Casim}(\PP_{0})}.

\begin{corollary}
Every Casimir element of \ec{\PP=(P_{0} \ltimes P_{1}, \{\,,\,\})} induces a Casimir element of the Poisson algebra \ec{\PP_{0}=(P_{0},\{\,,\,\}_{0})}.
\end{corollary}

Proposition \ref{prop:Casim} follows from the definition of Casimir element and the following facts: the set \ec{\ker{\mathscr{T}}} is a Lie subalgebra of $\PP$, \ec{\{\ker{\mathscr{T}},\ker{\mathscr{T}}\} \subseteq \ker{\mathscr{T}}}. Moreover, the intersection \ec{(\mathrm{Casim}(\PP_{0}) \oplus P_{1}) \cap \ker{\mathscr{T}}} is a Poisson subalgebra of $\PP$. These claims are deduced from the formula
    \begin{equation}\label{eq:TorBracket}
    	\mathscr{T}_{\{f \oplus \eta, g \oplus \xi\}} =
    	\mathscr{T}_{f \oplus \eta} \circ \mathrm{ad}^{0}_{g}
    	- \mathscr{T}_{g \oplus \xi} \circ \mathrm{ad}^{0}_{f} \\
    	+ \big(\D_{\dd{f}} + \mathrm{ad}_{\eta}\big) \circ \mathscr{T}_{g \oplus \xi}
    	- \big(\D_{\dd{g}} + \mathrm{ad}_{\xi}\big) \circ \mathscr{T}_{f \oplus \eta},
    \end{equation}
for all \ec{f \oplus \eta, g \oplus \xi \in P}. Here, \ec{\mathrm{ad}^{0}_{f}:=\{f,\cdot\}_{0}} and \ec{\mathrm{ad}_{\eta}:=[\eta,\cdot]_{1}}.

\begin{examplex}\label{exm:SympType}
Let $\PP$ be an admissible Poisson algebra. If \ec{\PP_{0}} is of \emph{symplectic type}, \ec{\mathrm{Casim}(\PP_{0})=R}, then
    \begin{equation*}
        \mathrm{Casim}(\PP) = R \oplus \mathrm{H}^{0}_{\overline{\partial}_{\D}}(\overline{\Gamma}^{\ast}_{\PP} ),
    \end{equation*}
where \ec{\mathrm{H}_{\overline{\partial}_{\D}}^{0}(\overline{\Gamma}^{\ast}_{\PP})} is the zero cohomology group of the cochain complex \ec{(\overline{\Gamma}^{\ast}_{\PP}, \overline{\partial}_{\D})} in \eqref{eq:Complex}. In particular, \ec{\mathrm{Casim}(\PP) \simeq R} if and only if \ec{\mathrm{H}^{0}_{\overline{\partial}_{\D}}(\overline{\Gamma}^{\ast}_{\PP}) = \{0\}}. This is the case, for example, if the Lie algebra \ec{(P_{1},[\,,\,]_{1})} is centerless.
\end{examplex}

\paragraph{Characterization of Poisson Derivations.} We derive the following description of Poisson derivations of $\PP$ in terms of the corresponding Poisson triple.

\begin{proposition}\label{prop:PoissDer}
An $R$--linear mapping \ec{X:P \to P}, satisfying conditions \ref{Der1}--\ref{Der4} of Lemma \ref{lema:PD}, is a derivation of $\PP$ if and only if
    \begin{align}
        X_{00}\{f,g\}_{0} &= \{X_{00}f,g\}_{0} + \{f,X_{00}g\}_{0} - \big( X_{01}\circ \K \big)(\dd f, \dd g), \label{eq:PoissDerAPA1} \\
		X_{11}[\eta,\xi]_{1} &= [X_{11}\eta,\xi]_{1}+[\eta,X_{11}\xi]_{1}+\D_{\dd X_{01}\eta}\xi-\D_{\dd X_{01}\xi}\eta, \label{eq:PoissDerAPA2} \\
        X_{10}\{f,g\}_0 &= \mathscr{T}_{X_{00}f\oplus X_{10}f}(g)-\mathscr{T}_{X_{00}g\oplus X_{10}g}(f) - \big( X_{11} \circ \K \big)(\dd f, \dd g), \label{eq:PoissDerAPA5} \\
        \big( X_{11} \circ \D_{\dd f} \big)(\eta) &= \big( \D_{\dd X_{00}f} + \mathrm{ad}_{X_{10}f} \big)(\eta) - \mathscr{T}_{X_{01} \eta \oplus X_{11}\eta}(f), \label{eq:PoissDerAPA6} \\
		X_{01} \circ \mathrm{ad}_{\eta} &= 0, \label{eq:PoissDerAPA3} \\
        X_{01}\circ \D_{\dd f} &= \mathrm{ad}_{f}^{0}\circ X_{01}, \label{eq:PoissDerAPA4}
    \end{align}
for all \ec{f \oplus \eta, g \oplus \xi \in P}. Here, \ec{\mathrm{ad}^{0}_{f}:=\{f,\cdot\}_{0}} and \ec{\mathrm{ad}_{\eta}:=[\eta,\cdot]_{1}}.
\end{proposition}
\begin{proof}
By definition, $X$ is a derivation of $\PP$ if and only if it is a derivation of $P$ and satisfies
    \begin{equation}\label{eq:XpoissderProof}
        X\{\pi_{1},\pi_{2}\} = \{X(\pi_{1}),\pi_{2}\} + \{\pi_{1},X(\pi_{2})\},
    \end{equation}
for any \ec{\pi_{1},\pi_{2} \in P}. In particular, by using formula \eqref{EcBracketPT}, if \ec{\pi_{1} = f \oplus 0} and \ec{\pi_{2} = g \oplus 0}, then condition \eqref{eq:XpoissderProof} implies \eqref{eq:PoissDerAPA1} and \eqref{eq:PoissDerAPA5}. If \ec{\pi_{1} = 0 \oplus \eta} and \ec{\pi_{2} = 0 \oplus \xi}, then \eqref{eq:XpoissderProof} implies \eqref{eq:PoissDerAPA2} and \eqref{eq:PoissDerAPA3}. If \ec{\pi_{1} = f \oplus 0} and \ec{\pi_{2} = 0 \oplus \eta}, then \eqref{eq:XpoissderProof} implies \eqref{eq:PoissDerAPA4} and \eqref{eq:PoissDerAPA6}. Reciprocally, by direct computations, one can show that conditions \eqref{eq:PoissDerAPA1}--\eqref{eq:PoissDerAPA4} for $X$ imply \eqref{eq:XpoissderProof}.
\end{proof}

\begin{corollary}[The $P_{1}$--Preserving Case]\label{cor:PoissP1}
A derivation \ec{X:P \to P} which preserves the $P_{0}$--module $P_{1}$, in the sense of \eqref{eq:X01Vanish}, is a derivation of $\PP$ if and only if
    \begin{align}
        &\delta_{\PP_{0}}X_{00} = 0, \label{eq:X01Poiss1} \\
        &X_{11}[\eta,\xi]_{1} = [X_{11}\eta,\xi]_{1}+[\eta,X_{11}\xi]_{1}, \label{eq:X01Poiss2} \\
        &(\dd_{\D}X_{10})(f,g) = (X_{11} \circ \K)(\dd{f},\dd{g}) -\K(\dd{X_{00}f}, \dd{g}) - \K(\dd{f}, \dd{X_{00}g}), \label{eq:X01Poiss3} \\
        &\pmb{[}X_{11}, \D_{\dd{f}}\pmb{]} = \D_{\dd X_{00}f} + [X_{10}f,\cdot]_{1}, \label{eq:X01Poiss4}
    \end{align}
for all \ec{f,g \in P_{0}}. Here, $\delta_{\PP_{0}}$ is the coboundary operator induced by Poisson algebra $\PP_{0}$ and $\dd_{\D}$ is the contravariant differential \eqref{eq:dDbar} on \ec{\X{R}^{1}(P_{0};P_{1})} associated to the contravariant derivative $\D$ and $\PP_{0}$.
\end{corollary}

Note that condition \eqref{eq:X01Poiss1} says that \ec{X_{00}} is a 1--cocycle of the complex of the Poisson algebra \ec{\PP_{0}}. Taking into account condition \eqref{EcPT2} and formula \eqref{eq:dD2bar}, if in addition \ec{\K=0}, then \eqref{eq:X01Poiss3} implies that \ec{X_{10}} is a 1--cocycle (\ec{\dd_{\D}X_{10}=0}) of the cochain complex \ec{(\mathrm{Der}(P_{0};P_{1}), \dd_{\D})} since in this case $\D$ is flat.

\begin{corollary}[The case \ec{\K=0}]\label{cor:CaseK0}
Let \ec{\PP=(P_{0} \ltimes P_{1},\{\,,\,\})} be an admissible Poisson algebra such that \ec{P_{0} \oplus \{0\}} is a Poisson subalgebra. Then, every Poisson derivation of $\PP$ induces a derivation of the Poisson algebra \ec{\PP_{0}=(P_{0},\{\,,\,\}_{0})}.
\end{corollary}

This last corollary follows from Corollary \ref{cor:P0KD0} and condition \eqref{eq:PoissDerAPA1}.

\paragraph{Characterization of Hamiltonian Derivations.} Here we give a description of all Hamiltonian derivations of an admissible Poisson algebra in terms of the corresponding Poisson triple.

\begin{proposition}\label{prop:hamalgebra}
The Lie subalgebra of all Hamiltonian derivations of $\PP$ is given by
    \begin{equation}\label{EcHamDescription}
       \mathrm{Ham}(\PP)= \left\{
            X_{h \oplus \eta} =
            \left.
            \left(
              \begin{array}{cc}
                \mathrm{ad}^{0}_{h} & 0 \\
                \mathscr{T}_{h \oplus \eta} & \D_{\dd{h}} + \mathrm{ad}_{\eta} \\
              \end{array}
            \right)
                \,\right|\,
            h \oplus \eta \in P
            \right\}.
    \end{equation}
Here, \ec{\mathrm{ad}^{0}_{h} = \{h,\cdot\}_0}, \ec{\mathrm{ad}_{\eta} = [\eta,\cdot]_1} and the $R$--linear morphism \ec{h \oplus \eta \to \mathscr{T}_{h \oplus \eta}} is defined by \eqref{EcTorsionDef}. Moreover,
\begin{itemize}
  \item the mapping \ec{h \oplus \eta \mapsto X_{h \oplus \eta}} is a Lie algebra homomorphism, whose kernel consists of all Casimir elements of $\PP$,
    \begin{equation}\label{eq:TeoHamCasim}
        \mathrm{Casim}(\PP) = \big\{ k \oplus \zeta \in P \mid k \in \mathrm{Casim}(\PP_{0}), \ k \oplus \zeta \in \ker{\mathscr{T}} \,\ \text{and} \,\ \D_{\dd{k}} + \mathrm{ad}_{\zeta} = 0 \big\};
    \end{equation}

  \item a Hamiltonian derivation \ec{X_{h \oplus \eta}} of $\PP$ preserves the commutative algebra \ec{P_{0}}, in the sense of \eqref{eq:X10Vanish}, if and only if
        \begin{equation*}
            h \oplus \eta \in \ker{\mathscr{T}};
        \end{equation*}

  \item the set of all $P_{0}$--preserving Hamiltonian derivations of $\PP$,
    \begin{equation}\label{EcHamZero}
        \mathrm{Ham}_{0}(\PP) :=
         \left\{
            \left.
            \left(
              \begin{array}{cc}
                \mathrm{ad}^{0}_{h} & 0 \\
                0 & \D_{\dd{h}} + \mathrm{ad}_{\eta} \\
              \end{array}
            \right)
                \,\right|\,
            h \oplus \eta \in \ker{\mathscr{T}}
            \right\},
    \end{equation}
is a Lie subalgebra of \ec{\mathrm{Ham}(\PP)}.
\end{itemize}
\end{proposition}

\begin{corollary}\label{cor:HamToHamP0}
Every Hamiltonian derivation of \ec{\PP=(P_{0} \ltimes P_{1}, \{\,,\,\})} induces a Hamiltonian derivation of the Poisson algebra \ec{\PP_{0}=(P_{0},\{\,,\}_{0})}.
\end{corollary}

We divide the proof of Proposition \ref{prop:hamalgebra} into the following lemmas.

\begin{lemma}\label{lemma:HamP1preserv}
Every Hamiltonian derivation of $\PP$ preserves the \ec{P_{0}}--module \ec{P_{1}}.
\end{lemma}
\begin{proof}
Let \ec{X} be a Hamiltonian derivation of \ec{\PP} with Hamiltonian element \ec{h\oplus \eta \in P}. Then, by definition and representation \eqref{EcX}, we have \ec{X_{01}(\xi) = (\pr_{0} \circ X \circ \iota_{1}) (\xi) = \pr_{0}\{h\oplus\eta, 0\oplus \xi\} = \{h,0\}_{0}=0}, for all \ec{\xi \in P_{1}}. Hence, \ec{X_{01}=0} and the claim follows from Lemma \ref{lemma:Xpreserving}.
\end{proof}

\begin{lemma}\label{lemma:ham}
A $P_{1}$--preserving $R$--linear mapping \ec{X:P \to P} given in \eqref{eq:X01Vanish} is a Hamiltonian derivation of \ec{\PP=(P,\{\,,\,\})} if and only if there exists \ec{h \oplus \eta \in P} such that
	\begin{equation}\label{EcHamDescrip}
		X_{00} = \mathrm{ad}^{0}_h, \quad X_{10} = \mathscr{T}_{h \oplus \eta} \quad \text{and} \quad X_{11} = \D_{\dd{h}} + \mathrm{ad}_{\eta}.
	\end{equation}
Moreover, a Hamiltonian element for $X$ is just \ec{h \oplus \eta}.
\end{lemma}
\begin{proof}
Suppose that there exists \ec{h\oplus\eta \in P} such that \ec{X=\{h\oplus\eta, \cdot \}}. Then, by \eqref{EcBracketPT} and \eqref{EcX}, we have
    \begin{align*}
        X_{00}(f) &= (\pr_{0} \circ X \circ \iota_{0})(f) = \pr_{0}\{h \oplus \eta, f \oplus 0\} = \{h,f\}_{0}, \\
        X_{10}(f) &= (\pr_{1} \circ X \circ \iota_{0})(f) = \pr_{1}\{h \oplus \eta, f \oplus 0\} = \K(\dd{h},\dd{f}) - \D_{\dd{f}}\eta, \\
        X_{11}(\xi) &= (\pr_{1} \circ X \circ \iota_{1})(\xi) = \pr_{1}\{h \oplus \eta, 0 \oplus \xi\} = \D_{\dd{h}}\xi + [\eta,\xi]_{1},
    \end{align*}
for all \ec{f \in P_{0}} and \ec{\xi \in P_{1}}. Hence, conditions \eqref{EcHamDescrip} are necessary for $X$ to be a Hamiltonian derivation of $\PP$. By direct computations, one can show that they are also sufficient.
\end{proof}

\begin{proof}[Proof of Proposition \ref{prop:hamalgebra}]
From Lemma \ref{lemma:ham} we have the characterization \eqref{EcHamDescription} of all Hamiltonian derivations of \ec{\PP}. The claim that the mapping \ec{h \oplus \eta \mapsto X_{h \oplus \eta}} is a Lie algebra homomorphism follows from the formula \ec{\pmb{[}X_{h \oplus \eta},X_{h' \oplus \eta'}\pmb{]} = X_{\{h \oplus \eta, h' \oplus \eta'\}}}. So, it is clear that the kernel of this morphism consists of all Casimir elements in \eqref{eq:TeoHamCasim}. As a consequence of Lemma \ref{lemma:Xpreserving} and formula \eqref{eq:TorBracket}, the set \ec{\mathrm{Ham}_{0}(\PP)} in \eqref{EcHamZero} is the Lie subalgebra of \ec{\mathrm{Ham}(\PP)} consisting of all $P_{0}$--preserving Hamiltonian derivations of $\PP$.
\end{proof}

    \section{The First Poisson Cohomology}\label{sec:FirstCohomology}

Here, we apply results of the previous sections to give a partial description of the first Poisson cohomology of admissible Poisson algebras. In particular, we show that an obstruction to the triviality of these Poisson cohomologies can be formulated in terms of the first cohomology group of the cochain complex \eqref{eq:ComplexBar}.

Let \ec{\PP=(P_{0} \ltimes P_{1},\{\,,\,\})} be an admissible Poisson algebra. Consider the first Poisson cohomologies \ec{\mathscr{H}^{1}(\PP)} of $\PP$ and \ec{\mathscr{H}^{1}(\PP_{0})} of the given Poisson algebra $\PP_{0}$. First, we give sufficient conditions to have a morphism between these Poisson cohomologies.

\begin{proposition}\label{prop:MorphismH1}
There exists an $R$--linear morphism \ec{\mathscr{H}^{1}(\PP) \to \mathscr{H}^{1}(\PP_{0})} in the following cases:
    \begin{enumerate}[label=(\alph*)]
      \item \label{morphismK0}The commutative subalgebra \ec{P_{0} \oplus \{0\}} is a Poisson subalgebra of $\PP$.
      \item Every Poisson derivation of $\PP$ preserves the $P_{0}$--module $P_{1}$.
    \end{enumerate}
\end{proposition}
\begin{proof}
Taking into account Corollaries \ref{cor:DerP1toP0}, \ref{cor:CaseK0}, \ref{cor:PoissP1} and \ref{cor:HamToHamP0}, the $R$--linear morphism is given by
    \begin{equation}\label{eq:MorphismH1H10}
        \mathscr{H}^{1}(\PP) \ni
        \left[ X =
            \left(
              \begin{array}{cc}
                X_{00} & X_{01} \\
                X_{10} & X_{11}
              \end{array}
            \right)
        \right] \longmapsto
        \big[ X_{00} \big] \in \mathscr{H}^{1}(\PP_{0}),
    \end{equation}
with \ec{X \in \mathrm{Poiss}(\PP)}.
\end{proof}

We note that the morphism \eqref{eq:MorphismH1H10} is not surjective, in general, since Lemma \ref{lema:PD} and equations \eqref{eq:X01Poiss1}--\eqref{eq:X01Poiss4} imply that not every Poisson derivation of $\PP_{0}$ can be extended to a Poisson derivation of $\PP$.

\begin{remark}\label{rmrk:P1Deriv}
By condition \eqref{eq:PoissDerAPA3}, if the Lie algebra \ec{(P_{1},[\,,\,]_{1})} associated to $\PP$ is perfect, \ec{P_{1} = [P_{1},P_{1}]_{1}}, then every Poisson derivation of $\PP$ preserves the $P_{0}$--module $P_{1}$.
\end{remark}

Now, consider the first cohomology group \ec{\mathrm{H}_{\partial_{\D}}^{1}(\Gamma^{\ast}_{\PP})} of the cochain complex \ec{(\Gamma^{\ast}_{\PP}, \partial_{\D})} in \eqref{eq:ComplexBar}. Then, we have defined a natural $R$--linear mapping:
    \begin{equation}\label{eq:Jmorphis}
        \mathrm{J}:\mathrm{H}_{\partial_{\D}}^{1}( \Gamma^{\ast}_{\PP} ) \longrightarrow \mathscr{H}^{1}(\PP),
            \qquad
        \mathrm{J}[c] :=
        \left[
            \begin{array}{cc}
             0 & 0 \\
             c & 0
           \end{array}
        \right].
    \end{equation}
By Corollary \ref{cor:PoissP1}, Proposition \ref{prop:hamalgebra} and Theorem \ref{teo:CochainComplexes}, the morphism $\mathrm{J}$ is well--defined and its kernel is given by
    \begin{equation*}
        \ker{\mathrm{J}} = \left\{\, [\mathscr{T}_{k \oplus \eta}] \mid k \in \mathrm{Casim}(\PP_{0}),\, \D_{\dd{k}} + [\eta, \cdot]_{1} = 0 \right\},
    \end{equation*}
where the derivation \ec{\mathscr{T}_{k \oplus \eta}} is defined in \eqref{EcTorsionDef}. Consequently, we have the following:

\begin{proposition}\label{prop:H1KerJ}
The non--triviality of the quotient
    \begin{equation}\label{eq:H1KerT}
        {\mathrm{H}_{\partial_{\D}}^{1}(\Gamma^{\ast}_{\PP}
        )} \,/\, {\ker{\mathrm{J}}}
    \end{equation}
is an obstruction to the triviality of \ec{\mathscr{H}^{1}(\PP)}.
\end{proposition}

We note that \ec{\ker{\mathrm{J}}=\{0\}} in the following cases:
    \begin{enumerate}[label=(\alph*)]
      \item\label{KerJ0Symp} The Poisson algebra $\PP_{0}$ is of symplectic type, \ec{\mathrm{Casim}(\PP_{0}) = R}.
      \item The condition \ec{\mathrm{Casim}(\PP_{0}) = \{f \in P_{0} \mid \D_{\dd{f}} = 0\} \cap \{f \in P_{0} \mid \K(\dd{f}, \cdot) = 0\}} holds.
      \item The Lie algebra \ec{(P_{1},[\,,\,]_{1})} is centerless.
    \end{enumerate}
Indeed, in this last case we have \ec{\mathrm{H}_{\partial_{\D}}^{1}( \Gamma^{\ast}_{\PP}) = 0} by the definition of the cochain complex \eqref{eq:ComplexBar}.

\begin{lemma}\label{lema:KerJ0Abelian}
We have that \ec{\ker{\mathrm{J}}=\{0\}} if the Poisson triple of $P$ corresponding to $\PP$ is of the form \ec{([\,,\,]_{1}=0, \D, \K)} and the $\partial_{\D}$--cohomology class of the 2--cocycle \ec{\Delta\K}  is trivial. Here, the $R$--linear mapping $\Delta$ is defined in \eqref{eq:Delta}.
\end{lemma}

In particular, as we show below, for admissible Poisson algebras induced by Poisson modules (Section \ref{sec:PoissonModules}) and associated to a symplectic leaf of a Poisson manifold (Section \ref{sec:PoissonSub}) the kernel of $\mathrm{J}$ is trivial.

\paragraph{Restricted First Poisson Cohomology.}
Taking into account Lemma \ref{lemma:HamP1preserv}, we restrict the study of the first Poisson cohomology to Poisson derivations that preserve the $P_{0}$--module $P_{1}$. Based on Lemmas \ref{lemma:Xpreserving} and \ref{lema:PD}, the Lie algebra of all these derivations of $P$ is given by
    \begin{equation*}
    	\mathrm{Der}_{1}(P) := \left\{
            \left.
            \Big( \begin{smallmatrix}
                    X_{00} & 0 \\
                    X_{10} & X_{11} \\
                 \end{smallmatrix} \Big)
                \,\right|\,
            X_{00} \in \mathrm{Der}(P_{0}), X_{10} \in \mathrm{Der}(P_{0};P_{1}), X_{11}\ \text{is a generalized $X_{00}$--derivation of $P_{1}$} \right\}.
    \end{equation*}

We denote the Lie subalgebra of all Poisson derivations of \ec{\PP} in \ec{\mathrm{Der}_{1}(P)} by
    \begin{equation*}
        \mathrm{Poiss}_{1}(\PP) := \mathrm{Der}_{1}(P) \cap \mathrm{Poiss}(\PP).
    \end{equation*}
By Proposition \ref{prop:hamalgebra}, we have
\ec{\mathrm{Ham}(\PP) \subseteq \mathrm{Der}_{1}(P)}. Then, we define the \emph{restricted first Poisson cohomology} of $\PP$ as the quotient
    \begin{equation}\label{eq:H1Rest}
        \mathscr{H}^{1}_{\mathrm{rest}}(\PP) := {\mathrm{Poiss}_{1}(\PP)} \,/\, {\mathrm{Ham}(\PP)}.
    \end{equation}
Clearly, it holds that
    \begin{equation}\label{eq:RestrictedFirstCoho}
        \mathscr{H}^{1}_{\mathrm{rest}}(\PP) \subseteq \mathscr{H}^{1}(\PP).
    \end{equation}
Moreover, taking into account Remark \ref{rmrk:P1Deriv}, we have the following criterion.

\begin{lemma}\label{lema:H1Hrest}
If the Lie algebra \ec{(P_{1},[\,,\,]_{1})} associated to $\PP$ is perfect, then we have
    \begin{equation*}
        \mathscr{H}^{1}(\PP)=\mathscr{H}^{1}_{\mathrm{rest}}(\PP).
    \end{equation*}
\end{lemma}

Recall that an $R$--linear mapping \ec{\mathscr{L}:P_{1} \to P_{1}} is called a \emph{generalized derivation} if there exists a derivation \ec{\ell \in \mathrm{Der}(P_{0})} such that $\mathscr{L}(f\eta) = f \mathscr{L}\eta + \ell(f)\eta$, for all \ec{f \in P_{0}} and \ec{\eta \in P_{1}}. We call such a derivation $\mathscr{L}$ a generalized $\ell$--derivation. Note that a generalized $0$--derivation is just a $P_{0}$--linear morphism of $P_{1}$.

Let \ec{([\,,\,]_{1},\D,\K)} be the Poisson triple of $P$ corresponding to $\PP$. We introduce a set \ec{\mathfrak{M}(\PP)} consisting of all generalized $\ell$--derivations \ec{\mathscr{L}} of $P_{1}$ satisfying the following conditions:
    \begin{itemize}
      \item \ec{\ell \in \mathrm{Poiss}(\PP_{0})},
      \item $\mathscr{L}$ derives the $P_{0}$--linear Lie bracket \ec{[\,,\,]_{1}} in the sense of \eqref{eq:X01Poiss2},
    \end{itemize}
and there exists \ec{\theta \in \mathrm{Der}(P_{0};P_{1})} such that
    \begin{align}
        \pmb{[}\D_{\dd f},\mathscr{L}\pmb{]} + \D_{\dd\ell(f)} &= \mathrm{ad}_{\theta(f)}, \label{eq:ThetaDf}\\
        (\mathscr{L} \circ \K)(\dd{f},\dd{g}) - \K \big( \dd \ell(f),\dd{g} \big) - \K \big( \dd{f},\dd \ell(g) \big) &= - \big( \dd_{\D}\theta \big)(f,g), \label{eq:ThetaK}
    \end{align}
for all \ec{f,g \in P_{0}}. Here, \ec{\mathrm{ad}_{\eta}=[\eta,\cdot]_{1}}, for \ec{\eta \in P_{1}}, and \ec{\dd_{\D}} is the contravariant differential \eqref{eq:dDbar} on \ec{\mathrm{Der}(P_{0};P_{1})} associated to $\D$ and $\PP_{0}$. It is easy to see that \ec{\mathfrak{M}(\PP)} is an $R$--module.

\begin{lemma}\label{lemma:MP}
We have the following properties:
    \begin{itemize}
      \item \ec{\mathfrak{M}(\PP)} is a Lie algebra.
      \item \ec{\mathfrak{M}(\PP)} contains the $R$--submodule of generalized derivations of $P_{1}$ induced by $\D$,
            \begin{equation*}
                \mathscr{C}(\PP):= \{\D_{\dd{f}} \mid f \in P_{0}\} \subseteq \mathfrak{M}(\PP).
            \end{equation*}
      \item \ec{\mathfrak{M}(\PP)} contains the ideal \ec{\mathrm{Inn}(P_{1},[\,,\,]_{1})} of inner derivations of \ec{(P_{1},[\,,\,]_{1})},
            \begin{equation*}
                \mathrm{Inn}(P_{1},[\,,\,]_{1}) \subseteq \mathfrak{M}(\PP).
            \end{equation*}
      \item The intersection of these two last $R$--modules is given by
            \begin{equation*}
                \mathscr{C}(\PP) \cap \mathrm{Inn}(P_{1},[\,,\,]_{1}) = \{\D_{\dd{k}} = \mathrm{ad}_{\zeta} \mid k \in \mathrm{Casim}(\PP_{0}), \zeta \in P_{1}\}.
            \end{equation*}
    \end{itemize}
\end{lemma}
\begin{proof}
For simplicity, we identify a generalized $\ell$--derivation $\mathscr{L}$ in $\mathfrak{M}(\PP)$ with the triple \ec{(\mathscr{L}, \ell, \theta)}, for some \emph{fixed} $\theta$ in \eqref{eq:ThetaDf} and \eqref{eq:ThetaK}. Then, taking into account the conditions \eqref{EcPT1}--\eqref{EcPT3} for Poisson triples, the commutator of two given \ec{(\mathscr{L}, \ell, \theta)} and \ec{(\mathscr{L}', \ell', \theta')} of $\mathfrak{M}(\PP)$ also belongs to $\mathfrak{M}(\PP)$ and corresponds to the triple \ec{(\pmb{[}\mathscr{L},\mathscr{L}'\pmb{]}, \pmb{[}\ell,\ell'\pmb{]}, \mathscr{L} \circ \theta' - \mathscr{L}' \circ \theta + \theta \circ \ell' - \theta' \circ \ell)}. Hence, $\mathfrak{M}(\PP)$ is Lie algebra over $R$. Moreover, a generalized derivation \ec{\D_{\dd{h}} \in \mathscr{C}(\PP)}, with \ec{h \in P_{0}}, belongs to $\mathfrak{M}(\PP)$ since we can assign it the triple \ec{(\mathscr{L}=\D_{\dd{h}},\ell=\{h,\cdot\}_{0},\theta=-\K(\dd{h},\dd(\cdot)))}. Similarly, an inner derivation \ec{\mathrm{ad}_{\eta} \in \mathrm{Inn}(P_{1},[\,,\,]_{1})}, with \ec{\eta \in P_{1}}, belongs to $\mathfrak{M}(\PP)$ since we can assign it the triple $(\mathscr{L}=\mathrm{ad}_{\eta},\ell=0,\theta=\D_{\dd(\cdot)}\eta)$. The proof of the claim for the intersection of \ec{\mathscr{C}(\PP)} and \ec{\mathrm{Inn}(P_{1},[\,,\,]_{1})} is direct.
\end{proof}

\begin{theorem}\label{teo:RestrictedCoho}
Let $\PP$ be an admissible Poisson algebra. Then, there exists a short exact sequence
    \begin{equation}\label{eq:exactsqc}
        0 \,\longrightarrow\, \frac{\mathrm{H}_{\partial_{\D}}^{1}\big( \Gamma^{\ast}_{\PP})}{\ker{\mathrm{J}}}
        \,\longrightarrow\, \mathscr{H}^{1}_{\mathrm{rest}}(\PP)
        \,\longrightarrow\, \frac{\mathfrak{M}(\PP)}{\mathscr{C}(\PP) + \mathrm{Inn}(P_{1},[\,,\,]_{1})}
        \,\longrightarrow\, 0.
    \end{equation}
Here, \ec{\mathrm{H}_{\partial_{\D}}^{1}(\Gamma^{\ast}_{\PP})} is the first cohomology group of the cochain complex \ec{(\Gamma^{\ast}_{\PP}, \partial_{\D})} in \eqref{eq:ComplexBar} and $\mathrm{J}$ is the $R$--linear mapping defined in \eqref{eq:Jmorphis}.
\end{theorem}
\begin{proof}
By definition of $\mathrm{J}$, the natural induced $R$--linear mapping given by
    \begin{equation*}
        {\mathrm{H}_{\partial_{\D}}^{1}\big( \Gamma^{\ast}_{\PP} )} \,/\, {\ker{\mathrm{J}}}
        \,\ni \big\{ [c] + \ker{\mathrm{J}} \big\}
            \,\longmapsto\,
        \Big[ \begin{smallmatrix}
             0 & 0 \\
             c & 0
           \end{smallmatrix} \Big]
        \in \mathscr{H}^{1}_{\mathrm{rest}}(\PP),
    \end{equation*}
is injective. Moreover, by definition of \ec{\mathscr{H}^{1}_{\mathrm{rest}}(\PP)} and \ec{\mathfrak{M}(\PP)}, one can show that the $R$--linear mapping
    \begin{equation*}
        \mathscr{H}^{1}_{\mathrm{rest}}(\PP) \ni
            \Big[ \begin{smallmatrix}
                    X_{00} & 0 \\
                    X_{10} & X_{11}
                 \end{smallmatrix} \Big]
            \longmapsto
            [X_{11}] \in \frac{\mathfrak{M}(\PP)}{\mathscr{C}(\PP) + \mathrm{Inn}(P_{1},[\,,\,]_{1})}
    \end{equation*}
is well--defined and surjective.
\end{proof}

Clearly, if the short exact sequence \eqref{eq:exactsqc} splits, then
    \begin{equation*}
        \mathscr{H}^{1}_{\mathrm{rest}}(\PP) \simeq
        \frac{\mathrm{H}_{\partial_{\D}}^{1}\big( \Gamma^{\ast}_{\PP} \big)}{\ker{\mathrm{J}}} \oplus \frac{\mathfrak{M}(\PP)}{\mathscr{C}(\PP) + \mathrm{Inn}(P_{1},[\,,\,]_{1})}.
    \end{equation*}
In particular, if \ec{\mathrm{H}_{\partial_{\D}}^{1}(\Gamma^{\ast}_{\PP} ) = \ker{\mathrm{J}}}, then
    \begin{equation}\label{eq:Rest1Iso}
        \mathscr{H}^{1}_{\mathrm{rest}}(\PP) \simeq
        \frac{\mathfrak{M}(\PP)}{\mathscr{C}(\PP) + \mathrm{Inn}(P_{1},[\,,\,]_{1})}.
    \end{equation}
\begin{corollary}\label{cor:trivialcenter}

If the Lie algebra \ec{(P_{1},[\,,\,]_{1})} associated to $\PP$ is centerless, then \eqref{eq:Rest1Iso} holds.
\end{corollary}

Taking into account this corollary and Lemma \ref{lema:H1Hrest}, we derive the following description of the first Poisson cohomology of $\PP$.

\begin{theorem}\label{teo:H1MCI}
Let $\PP$ be an admissible Poisson algebra and \ec{(P_{1},[\,,\,]_{1})} the associated Lie algebra. If \ec{(P_{1},[\,,\,]_{1})} is perfect and centerless, then
    \begin{equation*}
        \mathscr{H}^{1}(\PP) \simeq \frac{\mathfrak{M}(\PP)}{\mathscr{C}(\PP) + \mathrm{Inn}(P_{1},[\,,\,]_{1})}.
    \end{equation*}
\end{theorem}

\begin{remark}\label{rmrk:Semisimple}
If \ec{\mathfrak{g}=(P_{1},[\,,\,]_{1})} is a finite--dimensional (real) semisimple Lie algebra, then $\mathfrak{g}$ is perfect and centerless.
\end{remark}

Now, we proceed to the computation of the restricted first Poisson cohomology of $\PP$ for two particular classes of $P_0$--preserving derivations of $P$. In particular, we derive another interpretation of the quotient \eqref{eq:H1KerT}.

\paragraph{Particular Cases.} First, we note the following simple but important fact.

\begin{lemma}\label{lemma:AdX11}
A morphism \ec{X_{11}:P_{1} \to P_{1}} is a generalized \ec{\mathrm{ad}_{h}^{0}}--derivation of $P_{1}$, with \ec{h \in P_{0}}, if and only if \ec{X_{11} = \D_{\dd{h}} + X_{11}'} for some $P_{0}$--linear morphism $X_{11}'$ of $P_{1}$. Here, \ec{\mathrm{ad}_{h}^{0}=\{h,\cdot\}_{0}}.
\end{lemma}

Now, consider the following Lie subalgebras of \ec{\mathrm{Der}_{1}(P)} associated to the Poisson algebras $\PP$ and $\PP_{0}$:
    \begin{align*}
    	& \mathrm{Der}_{1}'(P) := \left\{
            \left.
            \Big( \begin{smallmatrix}
                    \mathrm{ad}_{h}^{0} & 0 \\
                    X_{10} & \D_{\dd{h}} + X_{11}' \\
                 \end{smallmatrix} \Big)
                \,\right|\,
            h \in P_{0},\, X_{10} \in \mathrm{Der}(P_{0};P_{1}),\, X_{11}'\ \text{is a $P_{0}$--linear morphism of $P_{1}$}
            \right\}, \\
        & \mathrm{Der}_{1}''(P) := \left\{
            \left.
            \Big( \begin{smallmatrix}
                    \mathrm{ad}_{h}^{0} & 0 \\
                    X_{10} & \D_{\dd{h}} + \mathrm{ad}_{\eta} \\
                 \end{smallmatrix} \Big)
                \,\right|\,
            h \in P_{0},\, \eta \in P_{1},\, X_{10} \in \mathrm{Der}(P_{0};P_{1})
            \right\}.
    \end{align*}
Note that by definition, and Proposition \ref{prop:hamalgebra}, we have
    \begin{equation*}
        \mathrm{Ham}(\PP) \subseteq \mathrm{Der}_{1}''(P) \subseteq \mathrm{Der}_{1}'(P) \subseteq \mathrm{Der}_{1}(P).
    \end{equation*}
So, we denote the Lie subalgebra of all Poisson derivations of $\PP$ in \ec{\mathrm{Der}_{1}'(P)} and \ec{\mathrm{Der}_{1}''(P)} by
    \begin{equation*}
        \mathrm{Poiss}_{1}'(\PP) := \mathrm{Der}_{1}'(P) \cap \mathrm{Poiss}(\PP) \quad \text{and} \quad
        \mathrm{Poiss}_{1}''(\PP) := \mathrm{Der}_{1}''(P) \cap \mathrm{Poiss}(\PP).
    \end{equation*}
Then, we define the following quotients:
    \begin{equation}\label{eq:HrestPrime}
        (\mathscr{H}_{\mathrm{rest}}^{1})'(\PP) := {\mathrm{Poiss}_{1}'(\PP)} \,/\, {\mathrm{Ham}(\PP)} \quad \text{and} \quad
        (\mathscr{H}_{\mathrm{rest}}^{1})''(\PP) := {\mathrm{Poiss}_{1}''(\PP)} \,/\, {\mathrm{Ham}(\PP)}.
    \end{equation}
Clearly, it holds that
    \begin{equation*}
        (\mathscr{H}_{\mathrm{rest}}^{1})''(\PP) \subseteq
        (\mathscr{H}_{\mathrm{rest}}^{1})'(\PP) \subseteq
        \mathscr{H}^{1}_{\mathrm{rest}}(\PP).
    \end{equation*}

Let \ec{\mathfrak{M}_{0}(\PP)} be the Lie ideal of \ec{\mathfrak{M}(\PP)}, \ec{\pmb{[}\mathfrak{M}(\PP),\mathfrak{M}_{0}(\PP)\pmb{]} \subseteq \mathfrak{M}_{0}(\PP)}, consisting of all $P_{0}$--linear morphisms \ec{\mathscr{L}':P_{1} \to P_{1}} satisfying the following conditions:
    \begin{itemize}
      \item $\mathscr{L}'$ derives the $P_{0}$--linear Lie bracket \ec{[\,,\,]_{1}} in the sense of \eqref{eq:X01Poiss2},
      \item there exists \ec{\theta' \in \mathrm{Der}(P_{0};P_{1})} such that
                \begin{equation*}
                    \pmb{[}\D_{\dd f},\mathscr{L}'\pmb{]} = \mathrm{ad}_{\theta'(f)} \quad \text{and} \quad
                    (\mathscr{L}' \circ \K)(\dd{f},\dd{g}) = - \big( \dd_{\D}\theta' \big)(f,g),
                \end{equation*}
            for all \ec{f,g \in P_{0}}.
    \end{itemize}

\begin{proposition}
Let $\PP$ be an admissible Poisson algebra.
Then,
    \begin{equation*}
        (\mathscr{H}_{\mathrm{rest}}^{1})''(\PP) \simeq
        {\mathrm{H}_{\partial_{\D}}^{1}\big( \Gamma^{\ast}_{\PP} \big)} \big/ \,{\ker{\mathrm{J}}}
    \end{equation*}
and there exists a short exact sequence
    \begin{equation}\label{eq:Sqnc2}
        0 \,\longrightarrow\, (\mathscr{H}_{\mathrm{rest}}^{1})''(\PP)
        \,\longrightarrow\, (\mathscr{H}_{\mathrm{rest}}^{1})'(\PP)
        \,\longrightarrow\, \frac{\mathfrak{M}_{0}(\PP)}{\mathscr{C}_{0}(\PP) + \mathrm{Inn}(P_{1},[\,,\,]_{1})}
        \,\longrightarrow\, 0.
    \end{equation}
Here, \ec{\mathrm{H}_{\partial_{\D}}^{1}(\Gamma^{\ast}_{\PP})} is the first cohomology group of the cochain complex \ec{(\Gamma^{\ast}_{\PP}, \partial_{\D})} in \eqref{eq:ComplexBar}, $\mathrm{J}$ is the $R$--linear mapping defined in \eqref{eq:Jmorphis} and \ec{\mathscr{C}_{0}(\PP):= \{\D_{\dd{k}} \mid k\in \mathrm{Casim}(\PP_{0})\}} is a submodule of \ec{\mathfrak{M}_{0}}.
\end{proposition}
\begin{proof}
It follows from Theorem \ref{teo:RestrictedCoho} and Lemma \ref{lemma:AdX11}.
\end{proof}

In particular, if the short exact sequence \eqref{eq:Sqnc2} splits, then
    \begin{equation*}
        (\mathscr{H}_{\mathrm{rest}}^{1})'(\PP) \simeq
        (\mathscr{H}_{\mathrm{rest}}^{1})''(\PP) \oplus \frac{\mathfrak{M}_{0}(\PP)}{\mathscr{C}_{0}(\PP) +\mathrm{Inn}(P_{1},[\,,\,]_{1})}.
    \end{equation*}
Taking into account Theorem \ref{teo:H1MCI} and condition \eqref{eq:X01Poiss1}, we derive the following description of \ec{\mathscr{H}^{1}(\PP)}.

\begin{theorem}\label{teo:H1M0}
Let $\PP$ be an admissible Poisson algebra. If the Lie algebra \ec{(P_{1},[\,,\,]_{1})} associated to $\PP$ is perfect and centerless, and the first Poisson cohomology of $\PP_{0}$ is trivial, then
    \begin{equation*}
        \mathscr{H}^{1}(\PP) \simeq \frac{\mathfrak{M}_{0}(\PP)}{\mathscr{C}_{0}(\PP) +\mathrm{Inn}(P_{1},[\,,\,]_{1})}.
    \end{equation*}
\end{theorem}

The following example illustrates a case where a Poisson derivation of $\PP_0$ can be extended to a Poisson derivation of $\PP$. In particular, it shows that \ec{(\mathscr{H}^{1}_{\mathrm{rest}})''(\PP)} is non--trivial, in general.

\begin{examplex}\label{exm:PoissNoHam}
Let \ec{\PP_{0}=(P_{0}, \cdot, \{\,,\,\}_{0})} be a Poisson algebra admitting a Poisson derivation $W$ which is not Hamiltonian. For the \ec{P_{0}}--module \ec{P_{1}} of (\ec{k \times k})--matrices with entries in \ec{P_{0}}, consider the admissible Poisson algebra \ec{\PP = (P, \{\,,\,\})} defined by the bracket \eqref{eq:APmatrix}. In this case, the center of the Lie algebra \ec{(P_{1}, [|\,,\,|])} is non--trivial. Moreover, for every \ec{h \oplus N\in P}, the derivation \ec{X \in \mathrm{Poiss}_{1}''(\PP)} given by
    \begin{equation*}
        X_{01} = 0, \quad X_{00} = \{h,\cdot\}_{0}, \quad X_{10} = W (\cdot)\,I - D_{\dd{(\cdot)}}N, \quad X_{11} = D_{\dd{h}} + [| N, \cdot\,|],
    \end{equation*}
is not Hamiltonian, with respect to $\PP$, since \ec{\mathscr{T}_{h\oplus N} = - D_{\dd{(\cdot)}}N} and hence cannot be \ec{X_{10} = \mathscr{T}_{h\oplus N}}. Here,  $I$ is the identity matrix.
\end{examplex}

\paragraph{The Abelian Case, \ec{\mathbf{[\,,\,]_{1}=0}}.} Recall that a short exact sequence does not necessarily splits, in general. In a geometric framework, for the so--called infinitesimal Poisson algebra of a Poisson submanifold, the corresponding short exact sequence \eqref{eq:exactsqc} split. In the general algebraic setting, one can give sufficient conditions to realize this property.

Let $\PP$ be an admissible Poisson algebra such that the induced Lie bracket \eqref{EcLieBraP1} is abelian, that is, the corresponding Poisson triple of $P$ is of the form \ec{([\,,\,]_{1}=0,\D,\K)}. First, we define the following $R$--submodule of \ec{\varepsilon(P_{1};P_{0})} in \eqref{epsilon}:
    \begin{equation*}
        \mathscr{E}(\PP) := \left\{ T \in  \varepsilon(P_{1};P_{0}) \mid T \circ \D_{\dd{f}} = \mathrm{ad}^{0}_{f} \circ T,\ \D_{\dd{(T\eta)}}\xi = \D_{\dd{(T\xi)}}\eta;\ f \in P_{0}, \eta,\xi \in P_{1} \right\},
    \end{equation*}
with \ec{\mathrm{ad}^{0}_{f} := \{f,\cdot\}_{0}}.

\begin{theorem}\label{teo:secAbelian}
There exists an exact sequence
    \begin{equation}\label{ExctSqncAbelian}
        0 \,\longrightarrow\, \mathscr{H}^{1}_{\mathrm{rest}}(\PP)
        \,\longrightarrow\, \mathscr{H}^{1}(\PP)
        \,\longrightarrow\, \mathscr{E}(\PP)
    \end{equation}
\end{theorem}
\begin{proof}
The theorem follows from the fact that we have the following exact sequence
    \begin{eqnarray*}
        0 \,\longrightarrow\, \mathrm{Poiss}_{1}(\PP)
        \,\longhookrightarrow\, \mathrm{Poiss}(\PP)
        \,\longrightarrow\, \mathscr{E}(\PP),
    \end{eqnarray*}
where the $R$--linear mapping \ec{\mathrm{Poiss}(\PP) \ni \Big(
                \begin{smallmatrix}
                  X_{00} & X_{01} \\
                X_{10} & X_{11} \\
                \end{smallmatrix}
            \Big) \mapsto X_{01} \in \mathscr{E}(\PP)}
is well--defined by conditions \eqref{eq:PoissDerAPA2} and \eqref{eq:PoissDerAPA4} since \ec{[\,,\,]_{1}=0}.
\end{proof}

Now, let \ec{\widetilde{\mathfrak{M}}(\PP)} be the Lie algebra of all generalized $\ell$--derivations $\mathscr{L}$ of $P_{1}$ such that
    \begin{equation*}
        \ell \in \mathrm{Poiss}(\PP_{0}) \quad \text{and} \quad \pmb{[}\D_{\dd f},\mathscr{L}\pmb{]} = - \D_{\dd\ell(f)}.
    \end{equation*}

\begin{theorem}\label{teo:SplitDeltaK}
If the $\dd_{\D}$--cohomology class of the 2--cocycle $\Delta\K$ is trivial, then
    \begin{equation*}
        \mathscr{H}^{1}(\PP) \simeq
        \mathrm{H}_{\dd_{\D}}^{1}\big( \X{}^{\ast} \big) \oplus \frac{\widetilde{\mathfrak{M}}(\PP)}{\mathscr{C}(\PP)} \oplus
        \mathscr{E}(\PP)
    \end{equation*}
Here, \ec{\mathrm{H}_{\dd_{\D}}^{1}(\X{}^{\ast})} is the first cohomology group of the cochain complex \ec{(\X{}^{\ast}, \dd_{\D})} induced by the coboundary operator \ec{\dd_{\D}} in \eqref{eq:dDbar} associated to $\D$ and the Poisson algebra $\PP_{0}$.
\end{theorem}

This theorem is deduced from the following two lemmas.

\begin{lemma}
If the $\dd_{\D}$--cohomology class of the 2--cocycle $\Delta\K$ is trivial, then
    \begin{equation*}
        \mathscr{H}^{1}_{\mathrm{rest}}(\PP) \simeq
        \mathrm{H}_{\dd_{\D}}^{1}\big( \X{}^{\ast} \big) \oplus \frac{\widetilde{\mathfrak{M}}(\PP)}{\mathscr{C}(\PP)}.
    \end{equation*}
\end{lemma}
\begin{proof}
Since \ec{[\,,\,]_{1}=0} the following facts hold: by definition, we have \ec{\mathrm{Inn}(P_{1},[\,,\,]_{1})=0}. By Corollary \ref{cor:Partialisd}, the cochain complex \ec{(\Gamma_{\PP},\partial_{\D})} associated to $\PP$ coincides with \ec{(\X{}^{\ast}, \dd_{\D})}, and $\K$ is a 2--cocycle. So, by Proposition \ref{prop:DeltaAst}, we have \ec{\dd_{\D}\Delta\K = 0}. From this, by Lemma \ref{lema:KerJ0Abelian}, we have \ec{\ker{\mathrm{J}}=\{0\}} since \ec{[\Delta\K]=0} by hypothesis. Then, by Theorem \ref{teo:RestrictedCoho}, the short exact sequence \eqref{eq:exactsqc} associated to $\PP$ is given by
    \begin{equation}\label{eq:ExactSeqTilde}
        0 \,\longrightarrow\, \mathrm{H}_{\dd_{\D}}^{1}\big( \X{}^{\ast} \big)
        \,\longrightarrow\, \mathscr{H}^{1}_{\mathrm{rest}}(\PP)
        \,\longrightarrow\, {\mathfrak{M}(\PP)} \,/\, {\mathscr{C}(\PP)}
        \,\longrightarrow\, 0.
    \end{equation}
Now, let us fix a primitive \ec{c \in \X{R}^{1}} of \ec{\Delta\K}, that is, \ec{\Delta{\K} = \dd_{\D}c}. First, we note that for every generalized $\ell$--derivation \ec{\mathscr{L} \in \widetilde{\mathfrak{M}}(\PP)} there exists \ec{\theta^{\mathscr{L}}_{c} := \mathscr{L} \circ c - c \circ \ell} such that \eqref{eq:ThetaK} holds. So, by definition, we have \ec{\mathfrak{M}(\PP)=\widetilde{\mathfrak{M}}(\PP)} since \ec{[\,,\,]_{1}=0}. Moreover, taking into account that the assignment \ec{(\mathscr{L},c) \to \theta^{\mathscr{L}}_{c}} is $R$--bilinear, the following ($c$--depending) $R$--linear mapping
    \begin{equation*}
        {\widetilde{\mathfrak{M}}(\PP)} \,/\, {\mathscr{C}(\PP)} \ni [\mathscr{L}]
            \,\longmapsto\,
            \left[
              \begin{array}{cc}
                \ell & 0 \\
                \theta^{\mathscr{L}}_{c} & \mathscr{L}
              \end{array}
            \right]
        \in \mathscr{H}^{1}_{\mathrm{rest}}(\PP).
    \end{equation*}
is a (well--defined) \emph{section} of the exact sequence \eqref{eq:ExactSeqTilde}. Hence, the lemma follows.
\end{proof}

\begin{lemma}
If the $\dd_{\D}$--cohomology class of the 2--cocycle $\Delta\K$ is trivial, then the exact sequence \eqref{ExctSqncAbelian} is short and splits.
\end{lemma}
\begin{proof}
Let us fix a primitive \ec{c \in \X{R}^{1}} of \ec{\Delta\K}, that is, \ec{\Delta{\K} = \dd_{\D}c}. Then, the following ($c$--depending) $R$--linear mapping
    \begin{equation*}
        \mathscr{E}(\PP) \ni T \,\longmapsto\,
        \left[
          \begin{array}{cc}
            \phantom{-}T \circ c & \phantom{-}T \\
            -c \circ T \circ c & -c \circ T
          \end{array}
        \right]
        \in \mathscr{H}^{1}(\PP)
    \end{equation*}
is a (well--defined) \emph{section} of the morphism \ec{\mathscr{H}^{1}(\PP) \rightarrow \mathscr{E}(\PP)} in \eqref{ExctSqncAbelian}.
\end{proof}

In particular, for admissible Poisson algebras induced by Poisson modules, the corresponding exact sequence \eqref{ExctSqncAbelian} is short and splits (Corollary \ref{cor:H1Lambda}) since in this case we have \ec{\K=0} (Corollary \ref{cor:APAbyPM}).

    \section{The Poisson Module Case} \label{sec:PoissonModules}

Here, we apply the results of the previous sections to the special class of admissible Poisson algebras associated to Poisson modules.

Recall that a \emph{Poisson module} over a Poisson algebra \ec{\PP_{0}=(P_0, \cdot, \{\,,\,\}_0)} is a
tuple \ec{(P_{1},\lambda)} consisting of a \ec{P_0}--module \ec{P_1} and an $R$--linear mapping \ec{\lambda:P_0 \times P_1\to P_1} satisfying the following conditions \cite{ReVoWe96, Bur01, Car03}:
    \begin{align}
        \lambda(f,g\eta) &= g \lambda(f,\eta) + \{f,g\}_0\,\eta, \label{EcLambdaD} \\
        \lambda(fg,\eta) &= f \lambda(g,\eta) + g\lambda(f,\eta), \label{EcLambdaMultip} \\
        \lambda\big(\{f,g\}_0, \eta\big) &= \lambda\big(f,\lambda(g,\eta)\big) - \lambda\big(g,\lambda(f,\eta)\big), \label{EcLambdaCurv}
    \end{align}
for all \ec{f,g \in P_0} and \ec{\eta \in P_1}.

\begin{examplex}
Every Poisson algebra $\PP_{0}$ is a Poisson module over itself, where \ec{P_{1}=P_{0}} and $\lambda:P_{0}\times P_{0}\to P_{0}$ is just defined by \ec{\lambda(f,g):=\{f,g\}_{0}}, for all \ec{f,g\in P_{0}}.
\end{examplex}

Some alternative versions of the following result can be found in \cite{ReVoWe96, Bur01, Car03}.

\begin{proposition}\label{prop:AdmissiblePoissonModule}
Every Poisson module \ec{(P_{1}, \lambda)} over a Poisson algebra $\PP_{0}$ induces an admissible Poisson algebra \ec{\PP_{\lambda} = (P_0 \ltimes P_1, \{\,,\,\}_{\lambda})} such that
    \begin{enumerate}[label=(\roman*)]
      \item  the subalgebra \ec{P_{0} \oplus \{0\}} of the commutative algebra \ec{P_0 \ltimes P_1} is a Poisson subalgebra of $\PP_{\lambda}$; \label{PM1}
      \item  the \ec{P_{0}}--module \ec{P_{1}} is endowed with the trivial Lie algebroid structure over \ec{P_{0}}. \label{PM2}
    \end{enumerate}
Moreover, the admissible Poisson structure is defined by
    \begin{equation}\label{EcPMGPABracket}
        \{f \oplus \eta, g \oplus \xi \}_{\lambda} := \{f,g\}_0 \oplus \big(\lambda(f, \xi) - \lambda(g, \eta) \big),
	\end{equation}
for all \ec{f,g \in P_{0}} and \ec{\eta,\xi \in P_{1}}. The converse is also true.
\end{proposition}

By direct computations, one can show that conditions \eqref{EcLambdaD}--\eqref{EcLambdaCurv} are equivalent to the Jacobi identity for the bracket in \eqref{EcPMGPABracket}. Items \ref{PM1} and \ref{PM2} of Proposition \ref{prop:AdmissiblePoissonModule} are consequences of the following observation: by Lemma \ref{lema:GAP-PT}, the Poisson triple of $P$ corresponding to $\PP_{\lambda}$ is given by
    \begin{equation}\label{EcModuleTriple}
        \big( [\,,\,]^{\lambda}_1 = 0,\, \D^{\lambda},\, \K^{\lambda} = 0 \big),
    \end{equation}
where the contravariant derivative \ec{\D^{\lambda}} is defined by
    \begin{equation}\label{eq:DLambda}
        \D^{\lambda}_{\dd f}\eta := \lambda(f, \eta), \quad f \in P_{0},\, \eta \in P_{1}.
    \end{equation}
So, it is clear that conditions \eqref{EcPT1} and \eqref{EcPT3} are satisfied and \eqref{EcPT2} is equivalent to \eqref{EcLambdaCurv},  which implies that \ec{\D^{\lambda}} is flat. Taking into account formula \eqref{EcModuleTriple}, from Corollaries \ref{cor:P0KD0} and \ref{cor:P1algebroid} we deduce that \ec{P_{0} \oplus \{0\}} is a Poisson subalgebra of \ec{\PP_{\lambda}} and \ec{P_{1}} has the trivial Lie algebroid structure over \ec{P_{0}}.

\begin{corollary}\label{cor:APAbyPM}
An admissible Poisson algebra $\PP$ is induced by a Poisson module if and only if the corresponding Poisson triple of $P$ is of the form \eqref{EcModuleTriple}. Moreover, the Poisson module structure over \ec{\PP_{0}} is defined by the formula \eqref{eq:DLambda}.
\end{corollary}

In other words, there exists a one--to--one correspondence between Poisson modules and flat contravariant derivatives which is defined by the formula \eqref{eq:DLambda} \cite{ReVoWe96, Bur01, Car03}.

\begin{examplex}\label{Ex:PM}
Let \ec{(P_0,\cdot,\{\,,\,\}_{0})} be a Poisson algebra and consider the \ec{P_{0}}--module \ec{P_1=\mathrm{End}_{R}(P_{0})}. Then, \ec{P_{1}} is a Poisson module over \ec{P_{0}} with \ec{\lambda:P_{0} \times P_{1}\to P_{1}} defined  by
    \begin{equation}\label{exm:LambdaEnd}
        \lambda(f,T) := \mathrm{ad}_{f}^{0}\circ T.
    \end{equation}
Here, \ec{\mathrm{ad}_{f}^{0} := \{f,\cdot\}_{0}}. The induced admissible Poisson structure on \ec{P_{0} \ltimes P_1} is given by
    \begin{equation}\label{eq:APSLambda}
        \{f\oplus T, g\oplus S\}_{\lambda} = \{f,g\}_{0}\oplus \big( \mathrm{ad}_{f}^{0}\circ S-\mathrm{ad}_{g}^{0}\circ T \big),
    \end{equation}
for all \ec{f\oplus T,\, g\oplus S \in P_{0}\ltimes P_1}.
\end{examplex}

\begin{remark}
By omitting condition \eqref{EcLambdaMultip}, we get a more general definition of a Poisson module which can be found in \cite{Car03}. But in this case, the bracket in \eqref{EcPMGPABracket} does not satisfy the Jacobi identity.
\end{remark}

By Theorem \ref{teo:correspondenceGPA-PT-LA}, Poisson modules also induce Lie algebroids satisfying condition \eqref{EcA}.

\begin{proposition}
Every Poisson module \ec{(P_{1},\lambda)} over a Poisson algebra \ec{(P_{0},\cdot,\{\,,\,\}_{0})} induces a Lie algebroid \ec{(\Omega^{1}_{P_{0}} \oplus P_{1}, \cSch{\,,\,}_{\lambda}, \varrho_{\lambda})} over \ec{P_{0}}, where the Lie bracket is defined by
    \begin{equation*}
        \cSch{\dd f \oplus \eta,\dd g \oplus \xi}_{\lambda} := \dd \{f,g\}_0\oplus \big( \lambda(f,\xi)-\lambda(g,\eta) \big),
    \end{equation*}
and the anchor map is given by
    \begin{equation*}
         \varrho_{\lambda}(\dd f \oplus \eta) := \{f, \cdot \}_{0},
    \end{equation*}
for all \ec{f,g \in P_0} and \ec{\eta, \xi \in P_1}.

\end{proposition}

\begin{examplex}
The Lie algebroid \ec{(\Omega^{1}_{P_{0}} \oplus P_{1}, \cSch{\,,\,}_{\lambda}, \varrho_{\lambda})} induced by the Poisson module of Example \ref{Ex:PM} is given by \ec{\cSch{\dd f \oplus T,\dd g \oplus S}_{\lambda} = \dd \{f,g\}_0 \oplus (\mathrm{ad}_{f}^{0} \circ S-\mathrm{ad}_{g}^{0}\circ T)} and \ec{\varrho_{\lambda}(\dd{f}\oplus T) = \mathrm{ad}_{f}^{0}}, for all \ec{f,g \in P_{0}} and  \ec{S,T\in \mathrm{End}_{R}(P_{0})}.
\end{examplex}

\paragraph{Gauge Equivalence.} Let \ec{\PP=(P,\{\,,\,\})} be an (arbitrary) admissible Poisson algebra and \ec{\PP_{\lambda}=(P,\{\,,\,\}_{\lambda})} an admissible Poisson algebra induced by a Poisson module \ec{(P_{1},\lambda)} over $\PP_{0}$.
Here we apply the results of Section \ref{sec:equivalenceGPA} to derive some equivalence criteria.

First, if $\PP$ is isomorphic to $\PP_{\lambda}$ by means of a gauge transformation \ec{\phi: \PP \to \PP_{\lambda}}, then the Poisson triples of $P$ corresponding to $\PP$ and $\PP_{\lambda}$ are related by
    \begin{align}
        [\,,\,]_{1} &= 0, \label{eq:TripMod1} \\[0.15cm]
        \D_{\dd{f}} &= \phi^{-1}_{11} \circ \D^{\lambda}_{\dd(\phi_{00}f)} \circ \phi_{11}, \label{eq:TripMod2} \\[0.15cm]
        \K(\dd{f},\dd{g}) &= \phi^{-1}_{11}\big( \D^{\lambda}_{\dd(\phi_{00}f)}\phi_{10}g - \D^{\lambda}_{\dd(\phi_{00}g)}\phi_{10}f - \phi_{10}\{f,g\}_0 \big), \label{eq:TripMod3}
    \end{align}
due to Theorem \ref{teo:GaugeEquiv} and transition rules \eqref{EcPT1prima}--\eqref{EcPT3prima}, for all \ec{f,g \in P_{0}}. From here, we deduce that
    \begin{enumerate}[label=(\alph*)]
      \item the admissible Poisson algebra $\PP$ is not necessarily associated to a Poisson module, by Corollary \ref{cor:APAbyPM} and formula \eqref{eq:TripMod3};

      \item \label{item:PoissModAbelian} a necessary condition for $\PP$ to be isomorphic to $\PP_{\lambda}$ by means of a gauge transformation is that the Lie algebra \ec{(P_{1},[\,,\,]_{1})} is abelian, by \eqref{eq:TripMod1}. In particular, by \eqref{EcPT2}, the contravariant derivative $\D$ must be flat.
    \end{enumerate}

Note that the transition rules \eqref{eq:TripMod1}--\eqref{eq:TripMod3}, by Corollaries \ref{cor:GaugeInduces} and \ref{cor:APAbyPM} and the formula \eqref{eq:DLambda}, provides a way to derive Poisson modules from a given one.

\begin{proposition}
Let \ec{(P_{1},\lambda)} be a Poisson module over a Poisson algebra \ec{\PP_{0} = (P_{0}, \cdot, \{,\}_{0})}. Then, every $P_{0}$--module isomorphism \ec{(\phi_{11}:P_{1} \to P_{1},\phi_{00}:P_{0} \to P_{0})} of $P_{1}$ such that $\phi_{00}$ is a Poisson algebra isomorphism of $\PP_{0}$ induces a Poisson module \ec{(P_{1},\lambda')} over \ec{\PP_{0}}, where \ec{\lambda'} is given by
    \begin{equation*}
        \lambda'(f,\eta) = \phi_{11}^{-1}\big( \lambda(\phi_{00}f, \phi_{11}\eta) \big),
    \end{equation*}
for all \ec{f \in P_{0}} and \ec{\eta \in P_{1}}.
\end{proposition}

Now, taking into account the item \ref{item:PoissModAbelian}, consider an admissible Poisson algebra $\PP$ such that the corresponding Poisson triple of $P$ is of the form
    \begin{equation}\label{eq:TripleAbelian}
        \big( [\,,\,]_{1}=0,\D,\K \big).
    \end{equation}
Then, by using Corollary \ref{cor:Partialisd}, we deduce the following:

\begin{theorem}\label{teo:PisoPlambda}
An admissible Poisson algebra satisfying \eqref{eq:TripleAbelian} is isomorphic to an admissible Poisson algebra associated to a Poisson module by means of a gauge transformation of the form \eqref{eq:GaugePhi10} if and only if the $\dd_{\D}$--cohomology class of the 2--cocycle \ec{\Delta \K \in \X{R}^{2}} is trivial. Here, the $R$--linear mapping $\Delta$ is defined in \eqref{eq:Delta}.
\end{theorem}
The proof follows from Corollaries \ref{cor:FlatP0Sub} and \ref{cor:APAbyPM} and transitions rules \eqref{EcPT1prima}--\eqref{EcPT3prima} and \eqref{eq:TripMod1}--\eqref{eq:TripMod3}.

\begin{examplex}
Let \ec{\PP_{0} = (P_0,\cdot,\{\,,\,\}_{0})} be a Poisson $R$--algebra. Consider the admissible Poisson algebra \ec{\PP=(P=P_{0} \ltimes P_{1}, \{\,,\,\}_{k_{0}})}, where \ec{P_1=\mathrm{End}_{R}(P_{0})} is the \ec{P_{0}}--module of $R$--endomorphisms of \ec{P_{0}} and, for a fixed \ec{k_{0} \notin \mathrm{Casim}(\PP_{0})}, the admissible Poisson structure is given by
    \begin{multline*}
        \{f\oplus T,g\oplus S\}_{k_{0}} = \{f,g\}_{0} \oplus \left( \mathrm{ad}^{0}_f \circ S - \mathrm{ad}^{0}_g \circ T - \left( f\,\mathrm{ad}^{0}_{g} - g\,\mathrm{ad}^{0}_{f}\right)\circ \mathrm{ad}^{0}_{k_{0}} \right.\\
         \left. + \{f,g\}_{0} \cdot \mathrm{ad}^{0}_{k_{0}} + \{f,k_{0}\}_{0} \cdot \mathrm{ad}^{0}_{g} - \{g,k_{0}\}_{0} \cdot \mathrm{ad}^{0}_{f} \right), \quad f \oplus T, g \oplus S \in P,
    \end{multline*}
with \ec{\mathrm{ad}^{0}_{h} = \{h,\cdot\}_0}. In this case, the \ec{P_{0}}--linear Lie bracket \eqref{EcLieBraP1} on \ec{P_{1}} induced by $\PP$ is abelian. But, by Corollary \ref{cor:APAbyPM}, the Poisson algebra $\PP$ is not associated to a Poisson module over $\PP_{0}$ since
    \begin{equation*}
        \K^{\PP}(\dd{f},\dd{g}) = \{f,g\}_{0}\,\mathrm{ad}^{0}_{k_{0}} + \{f,k_{0}\}_{0}\,\mathrm{ad}^{0}_{g} - \{g,k_{0}\}_{0}\,\mathrm{ad}^{0}_{f} - \left( f\,\mathrm{ad}^{0}_{g} -  g\,\mathrm{ad}^{0}_{f}\right)\circ \mathrm{ad}^{0}_{k_{0}}, \quad f,g \in P_{0}.
    \end{equation*}
However, $\PP$ is gauge--isomorphic to the admissible Poisson algebra \ec{(P,\{\,,\,\}_{\lambda})} defined by \eqref{eq:APSLambda}, which is induced by the Poisson module \ec{(P_{1},\lambda)} over \ec{P_{0}} defined by \eqref{exm:LambdaEnd}, by means of the $\mu$--gauge transformation \ec{\phi_{\mu}:P \to P} given by \ec{\phi_{\mu}(f \oplus T) = f \oplus ( T - \mathrm{ad}^{0}_{k_{0}f})}, with \ec{\mu(\dd{f}) = -\mathrm{ad}^{0}_{k_{0}f}}, for all \ec{f \oplus T, g \oplus S \in P}.
\end{examplex}

\paragraph{Cochain Complex.} Let \ec{(P_{1},\lambda)} be a Poisson module  over the given Poisson algebra $\PP_{0}=(P_{0},\cdot,\{\,,\,\}_{0})$. Consider the admissible Poisson algebra \ec{\PP_{\lambda}=(P,\{\,,\,\}_{\lambda})} associated to \ec{(P_{1},\lambda)}.

Since the Poisson triple of $P$ corresponding  to $\PP_{\lambda}$ is of the form \eqref{EcModuleTriple}, by Corollary \ref{cor:Partialisd}, the cochain complex \eqref{eq:ComplexBar} induced by $\PP_{\lambda}$ just coincide with the cochain complex
    \begin{equation}\label{eq:XbarCochain}
        \big( \X{}^{\ast} := \oplus_{k}\ \X{R}^{k} \equiv \X{R}^{k}(P_{0}; P_{1}),\, \dd_{\lambda} \big)
    \end{equation}
associated to the Poisson module \ec{(P_{1},\lambda)} with coboundary operator given by \cite{Car03,Zhu19}
    \begin{equation*}
        \fiteq{(\dd_{\lambda}Q)(f_{0},\ldots,f_{k}) := \sum_{i=0}^{k}(-1)^{i}\,\lambda\big( f_{i}, Q(f_{0}, \ldots, \widehat{f}_{i},\ldots,f_{k}) \big) + \sum_{0 \leq i < j \leq k} (-1)^{i+j}\, Q\big( \{f_{i},f_{j}\}_{0},f_{0} \ldots, \widehat{f}_{i},\ldots,\widehat{f}_{j}, \ldots f_{k} \big)},
    \end{equation*}
for all \ec{Q \in \X{R}^{k}} and \ec{f_{0},\ldots,f_{k} \in P_{0}}.

Denote by \ec{\mathrm{H}_{\lambda}^{1}} the first cohomology group of the canonical cochain complex \eqref{eq:XbarCochain}. By Theorem \ref{teo:SplitDeltaK} and Corollary \ref{cor:APAbyPM}, we deduce a description of the first Poisson cohomology of admissible Poisson algebras induced by Poisson modules.

\begin{corollary}\label{cor:H1Lambda}
For the admissible Poisson algebra $\PP_{\lambda}$, we have
    \begin{equation*}
        \mathscr{H}^{1}\big( \PP_{\lambda} \big) \,\simeq\,
        \mathrm{H}_{\lambda}^{1} \oplus
        \frac{\mathfrak{M}_{\lambda}}{\mathscr{C}_{\lambda}} \oplus
        \mathscr{E}_{\lambda}.
    \end{equation*}
Here, $\mathfrak{M}_{\lambda}$ is the Lie algebra of all generalized $\ell$--derivations $\mathscr{L}$ of $P_{1}$ satisfying
    \begin{equation*}
        \ell \in \mathrm{Poiss}(\PP_{0}) \quad \text{and} \quad \pmb{[}\lambda(f, \cdot), \mathscr{L}\pmb{]} = -\lambda\big( \ell(f),\cdot \big), \quad f \in P_{0};
    \end{equation*}
$\mathscr{C}_{\lambda}$ is the $R$--submodule of generalized derivations of $P_{1}$ induced by $\lambda$,
        \begin{equation*}
            \mathscr{C}_{\lambda} := \{\lambda(f,\cdot) \mid f \in P_{0}\} \subseteq \mathfrak{M}_{\lambda};
        \end{equation*}
and $\mathscr{E}_{\lambda}$ is the $R$--module of all $R$--linear mapping \ec{T \in \varepsilon(P_{1};P_{0})}, defined in \eqref{epsilon}, such that
    \begin{equation*}
        T\big( \lambda( f, \eta) \big) = \{f, T\eta\}_{0} \quad \text{and} \quad \lambda( T\eta, \xi) = \lambda( T\xi,\eta), \quad f \in P_{0}, \eta,\xi \in P_{1}.
    \end{equation*}
\end{corollary}

\begin{remark}
This corollary recovers the result of \cite[Theorem 4.1]{Zhu19}, in which the quotient \ec{{\mathfrak{M}_{\lambda}}/{\mathscr{C}_{\lambda}}} is denoted by \ec{\mathrm{hp}^{1}(P)} and is called the restricted first Poisson cohomology group of the trivial extension $P$.
\end{remark}

\paragraph{Deformations.} Now, starting from a Poisson module, we construct admissible Poisson algebras by using a special class of deformations.

From Theorems \ref{teo:NewAdmissible} and \ref{teo:PisoPlambda} we deduce that every 2--cocycle of the cochain complex \eqref{eq:XbarCochain} induces an \emph{exact (infinitesimal) deformation} of the admissible Poisson algebra $\PP_{\lambda}$. Moreover, the deformation is trivial if the class of the corresponding 2--cocycle is zero.

\begin{theorem}
Every 2--cocycle \ec{\mathscr{C} \in \X{R}^{2}} of the canonical cochain complex of \ec{(P_{1},\lambda)} induces a $t$--parameterized family \ec{\PP^{t}_{\lambda}=(P,\{\,,\,\}^{t}_{\lambda})} of admissible Poisson algebras defined by
    \begin{equation*}
        \{f \oplus \eta, g \oplus \xi\}^{t} := \{f,g\}_0 \oplus \big(\lambda(f, \xi) - \lambda(g, \eta) + (t\mathscr{C})(f,g) \big), \quad t \in R, \\
    \end{equation*}
with \ec{f \oplus \eta, g \oplus \xi \in P}. In particular, we have \ec{\PP_{\lambda} = \PP^{t=0}_{\lambda}}. Moreover, if the cohomology class of \ec{\mathscr{C}} is trivial, then $\PP^{t}_{\lambda}$ and $\PP_{\lambda}$ are Poisson isomorphic for all $t$.
\end{theorem}

In particular, we have that every $R$--linear mapping \ec{c: P_{0} \to P_{1}} defines an exact trivial deformation of $\PP_{\lambda}$ given by
    \begin{equation*}
        \{f \oplus \eta, g \oplus \xi\}^{t}_{\lambda,c} = \{f,g\}_{0} \oplus \big(\, \lambda\big( f, \xi + t c(g) \big) - \lambda\big(g, \eta + t c(f)\big) - t c(\{f,g\}_{0}) \big).
    \end{equation*}

    \section{Poisson Algebras of Poisson Submanifolds} \label{sec:PoissonSub}

In this section, following the results of \cite{RuGaVo20}, we show that the commutative algebra of fiberwise affine functions on the normal bundle of an embedded Poisson submanifold carries an admissible Poisson structure determining the so--called \emph{infinitesimal Poisson algebra} of the submanifold. As a consequence of the general results of Section \ref{sec:FirstCohomology}, we derive some splitting properties of the first cohomology of these Poisson algebras.

To formulate a geometric version of our algebraic approach, we start with an arbitrary vector bundle.

\paragraph{The Trivial Extension Algebra \ec{\boldsymbol{\Cinf{\mathrm{aff}}(E)}}.} Let \ec{E \overset{\pi}{\rightarrow} S} be a vector bundle over a smooth manifold $S$. Consider the $\Cinf{S}$--module of \emph{fiberwise affine} $\Cinf{}$--\emph{functions} on $E$,
    \begin{equation*}
        \Cinf{\mathrm{aff}}(E) := \pi^{\ast}\Cinf{S} \oplus \Cinf{\mathrm{lin}}(E),
    \end{equation*}
where \ec{\Cinf{\mathrm{lin}}(E)} is the $\Cinf{S}$--module of fiberwise linear functions on $E$.

It is clear that the module \ec{\Cinf{\mathrm{lin}}(E)} is isomorphic to the $\Cinf{S}$--module \ec{\Gamma E^{\ast}} of smooth sections of the dual bundle $E^{\ast}$ over $S$. So, we have
    \begin{equation*}
        \Cinf{\mathrm{aff}}(E) \simeq \Cinf{S} \oplus \Gamma{E^{\ast}}.
    \end{equation*}
Consequently, \ec{\Cinf{\mathrm{aff}}(E)} is a commutative algebra isomorphic to the \emph{trivial extension algebra} of the commutative algebra \ec{P_{0}=\Cinf{S}} by the $P_{0}$--module \ec{P_{1} = \Gamma{E^{\ast}}} with multiplication given by
    \begin{equation}\label{eq:TEAaff}
        (f \oplus \eta) \cdot (g \oplus \xi) = fg \oplus (f\xi + g\eta),
    \end{equation}
for all \ec{f,g \in \Cinf{S}} and \ec{\eta,\xi \in \Gamma{E^{\ast}}}.

Observe that the multiplication \eqref{eq:TEAaff} can be also introduced in terms of the linearization procedure on the total space $E$ at the zero section \ec{S \hookrightarrow E} as follows: let
    \begin{equation}\label{eq:Aff}
        \mathrm{Aff}: \Cinf{E} \to \Cinf{\mathrm{aff}}(E)
    \end{equation}
be the \emph{linearization mapping} \cite{RuGaVo20} which assigns to each smooth function $F$ on $E$ its affine part \ec{\mathrm{Aff}(F)} in the Taylor expansion of $F$ at the points of $S$. Then, for any \ec{\varphi_{1},\varphi_{2} \in \Cinf{\mathrm{aff}}(E)} the formula \eqref{eq:TEAaff} reads \ec{\varphi_{1} \cdot \varphi_{2} = \mathrm{Aff}(\varphi_{1} \varphi_{2})}.

\paragraph{Derivations of the Trivial Extension Algebra \ec{\boldsymbol{\Cinf{\mathrm{aff}}(E)}}.} Let \ec{Z \in \Gamma\,\T{E}} be a vector field on $E$ which is tangent to the zero section \ec{S \hookrightarrow E}. Then, one can associate to $Z$ a natural derivation \ec{Z^{(2)} \in \mathrm{Der}_{1}(\Cinf{\mathrm{Aff}})} of the trivial extension algebra \ec{\Cinf{\mathrm{aff}}(E)} given by
    \begin{equation}\label{Aff2}
        Z^{(2)}\varphi := \mathrm{Aff}\big( \mathrm{L}_{Z}\varphi \big), \quad \varphi \in \Cinf{\mathrm{aff}}(E).
    \end{equation}
It is clear that \ec{Z_{01}^{(2)}=0} and hence, by Lemma \ref{lemma:Xpreserving}, the derivation $Z^{(2)}$ preserves the $\Cinf{S}$--module $\Cinf{\mathrm{lin}}(E)$. Moreover, the components \ec{Z_{00}^{(2)}} and \ec{Z_{11}^{(2)}}
are uniquely determined by the Lie derivative \ec{\mathrm{L}_{\mathrm{var}_{S}Z}} along the fiberwise linear vector field \ec{\mathrm{var}_{S}Z}, which is called the \emph{first variation of $Z$ at $S$} \cite{RuGaVo20}. In general, the component \ec{Z_{10}^{(2)}: \Cinf{S} \to \Cinf{\mathrm{lin}}(E)} is non--trivial.

Taking into account Lemma \ref{lemma:Xpreserving}, we have the following simple but important criteria \cite{RuGaVo20}.

\begin{lemma}\label{lem:Z2}
For the derivation \ec{Z^{(2)}} in \eqref{Aff2}, the following assertions are equivalent:
    \begin{enumerate}[label=(\alph*)]
        \item The component $Z_{10}^{(2)}$ vanishes, \ec{Z_{10}^{(2)} = 0}.
        \item The derivation $Z^{(2)}$ is the restriction of the Lie derivative along a linear vector field on $E$.
        \item The canonical splitting \ec{\T_{S}E = \T{S} \oplus E} is invariant with respect to the differential of the flow of $Z$.
    \end{enumerate}
\end{lemma}

So, the component \ec{Z_{10}^{(2)}} measures the deviation of \ec{Z^{(2)}} from the property to be the restriction to \ec{\Cinf{\mathrm{aff}}(E)} of a derivation of the commutative algebra \ec{\Cinf{E}}.

\paragraph{Admissible Poisson Structures on \ec{\boldsymbol{\Cinf{\mathrm{aff}}(E)}}.} Now, suppose that the base $S$ is equipped with a Poisson bivector field $\psi$. Recall that a contravariant connection \cite{Vais91} on \ec{E^{\ast}} consists of $\R{}$--linear operators \ec{\D_{\alpha}:\Gamma{E^{\ast}} \to \Gamma{E^{\ast}}} which are \ec{\Cinf{S}}--linear in \ec{\alpha \in \Gamma\,\T^{\ast}S} and satisfy the Leibniz type rule
    \begin{equation*}
        \D_{\alpha}(f\eta) = f\D_{\alpha}\eta + \big(\dlie{\psi^{\natural}\alpha}f\big)\eta,
    \end{equation*}
for all \ec{f \in \Cinf{S}} and \ec{\eta \in \Gamma{E^{\ast}}}. The curvature \ec{\mathrm{Curv}^{\D}} of $\D$ is defined as
    \begin{equation*}
        \mathrm{Curv}^{\D}(\alpha,\beta) := \D_{\alpha}\D_{\beta} - \D_{\beta}\D_{\alpha} - \D_{[\alpha,\beta]_{\T^{\ast}S}},
    \end{equation*}
for all \ec{\alpha, \beta \in \Gamma\,\T^{\ast}S}. Here, \ec{[\,,\,]_{\T^{\ast}S}} denotes the Lie bracket for 1--forms on the Poisson manifold \ec{(S,\psi)} \cite{Duf05}.

Let \ec{\PP_{0} = (\Cinf{S}, \{\,,\,\}_{\psi})} be the Poisson algebra of \ec{(S,\psi)}. Suppose we are given a Poisson triple \ec{([\,,\,]_{\mathrm{fib}},\D,\K)} of the trivial extension algebra \ec{\Cinf{\mathrm{aff}}(E)}, that is,
    \begin{itemize}
        \item a fiberwise Lie algebra structure \ec{[\,,\,]_{\mathrm{fib}}} on \ec{E^{\ast}};
        \item a contravariant connection $\D$ on \ec{E^{\ast}};
        \item a \ec{E^{\ast}}--valuated bivector field \ec{\K: \Gamma\,\T^{\ast}S \times \Gamma\,\T^{\ast}S \to \Gamma{E^{\ast}}} on $S$;
    \end{itemize}
satisfying the following conditions \cite{RuGaVo20}
    \begin{equation*}
        \pmb{[}\D_{\alpha},\mathrm{ad}_{\eta}\pmb{]} = \mathrm{ad}_{\D_{\alpha}\eta},
            \qquad
        \mathrm{Curv}^{\D}(\alpha,\beta) = \mathrm{ad}_{\K(\alpha,\beta)},
            \qquad
        \underset{(\alpha,\beta,\delta)}{\mathfrak{S}} \D_{\alpha}\,\K\big(\beta,\delta\big) + \K\big(\alpha,[\beta,\delta]_{\T^{\ast}S}\big) = 0,
    \end{equation*}
for all \ec{\alpha,\beta,\delta \in \Gamma\,\T^{\ast}S} and \ec{\eta \in \Gamma{E^{\ast}}}. Here, \ec{\mathrm{ad}_{\eta} := [\eta,\,]_{\mathrm{fib}}}.

Applying Theorem \ref{teo:correspondenceGPA-PT-LA} to the commutative algebra \ec{P = \Cinf{\mathrm{aff}}(E)}, we get an admissible Poisson structure on \ec{\Cinf{\mathrm{aff}}(E)} defined by (Lemma \ref{LemaTripleToAlg})
    \begin{equation}\label{EcBracketAff}
        \{f \oplus \eta, g \oplus \xi\}^{\mathrm{aff}} := \psi(\dd{f},\dd{g}) \oplus \big( \D_{\dd{f}}\xi - \D_{\dd{g}}\eta + [\eta,\xi]_
        {\mathrm{fib}}+ \K(\dd{f},\dd{g}) \big),
    \end{equation}
for all \ec{f,g \in \Cinf{S}} and \ec{\eta,\xi \in \Gamma{E^{\ast}}}. This yields the following:

\begin{proposition}\label{prop:APAGeom}
Suppose that there exists a Poisson triple associated to \ec{E \overset{\pi}{\rightarrow} (S,\psi)}. Then, the commutative algebra of fiberwise affine functions on $E$ admits an admissible Poisson structure defined by the formula \eqref{EcBracketAff}.
\end{proposition}

\paragraph{Derivations of the APA \ec{\boldsymbol{\Cinf{\mathrm{aff}}(E)}}.} Suppose we start again with a vector bundle \ec{E \overset{\pi}{\rightarrow} S} over $S$ equipped with a Poisson bivector field \ec{\Pi \in \Gamma\wedge^{2}\T E}, and the corresponding Poisson bracket $\{f,g\}_{E}=\Pi(\dd f, \dd g)$. Now, assume that the zero section $S\hookrightarrow E$ is a Poisson submanifold of
$(E,\Pi)$, that is,  $\Pi$ is tangent to $S$ and restricted to a Poisson tensor $\psi$ on $S$.

Consider the Lie algebra \ec{\mathrm{Poiss}_{S}(E,\Pi)} of all Poisson vector fields $Z$ on $E$ which are tangent to $S$, \ec{\dlie{Z}\Pi=0} and \ec{Z_{p} \in \T_{p}S}, for every \ec{p \in S}. If we suppose that there exists a Poisson triple associated to \ec{E \overset{\pi}{\rightarrow} (S,\psi)}, then we have the following:

\begin{lemma}\label{lema:PoissZ2}
For every \ec{Z \in \mathrm{Poiss}_{S}(E,\Pi)}, its linearization \ec{Z^{(2)}} in \eqref{Aff2} is a Poisson derivation of the admissible Poisson algebra \ec{\PP_{\mathrm{aff}}=(\Cinf{\mathrm{aff}}(E), \{\,,\,\}^{\mathrm{aff}})} with \ec{Z_{01}^{(2)}=0},
    \begin{equation*}
        Z^{(2)}\{\phi_{1},\phi_{2}\}^{\mathrm{aff}} = \big\{ Z^{(2)}\phi_{1},\phi_{2} \big\}^{\mathrm{aff}} + \big\{ \phi_{1},Z^{(2)}\phi_{2} \big\}^{\mathrm{aff}}.
    \end{equation*}
Moreover, the linearization mapping $Z\rightarrow Z^{(2)}$ is a Poisson algebra homomorphism.
\end{lemma}

Now, consider the Lie algebra \ec{\mathrm{Ham}(E,\Pi)} of Hamiltonian vector fields \ec{X_{H}=\mathbf{i}_{dH}\Pi}, with \ec{H \in \Cinf{E}}.

\begin{lemma}\label{lema:HamZ2}
For every Hamiltonian vector field  $X_{H}$, its linearization \ec{X_{H}^{(2)}} in \eqref{Aff2} is a Hamiltonian derivation of the admissible Poisson algebra \ec{\PP_{\mathrm{aff}}=(\Cinf{\mathrm{aff}}(E), \{\,,\,\}^{\mathrm{aff}})} of the form \eqref{EcHamDescription}, associated to an element \ec{h \oplus \eta \in \Cinf{S} \oplus \Gamma E^{\ast}}, where
    \begin{equation*}
        h = H|_{S}, \quad \text{and} \quad \eta=(\dd_{S}H)|_{E}
    \end{equation*}
is the second term in the decomposition of \ec{\dd_{S}H} relative to the canonical splitting \ec{\T_{S}^{\ast}{E}=E^{\circ}\oplus (\T{S})^{\circ}}.
\end{lemma}

In particular, we have
    \begin{equation*}
        \big( X_{H}^{(2)} \big)_{00} = \dlie{\mathbf{i}_{\dd h}\psi}
    \end{equation*}
and the \emph{first variation vector field} \ec{\mathrm{var}_{S}X_{H}} induces a Hamiltonian derivation of $\PP$ if and only if
    \begin{equation*}
        \big( X_{H}^{(2)} \big)_{10}(f\oplus0) = \mathrm{Aff}\big( \dlie{X_{H}}(\pi^{\ast}f) \big) = 0, \quad \text{for all} \quad f \in \Cinf{S}.
    \end{equation*}
Here, \ec{\mathrm{Aff}} is the linearization mapping \eqref{eq:Aff}.

\paragraph{Tangent First Poisson Cohomology.} Since clearly \ec{\mathrm{Ham}(E,\Pi) \subseteq \mathrm{Poiss}_{S}(E,\Pi)}, one can define the ``tangent'' first Poisson cohomology by
    \begin{equation*}
        \mathscr{H}_{S}^{1}(E,\Pi) := \frac{\mathrm{Poiss}_{S}(E,\Pi)}{\mathrm{Ham}(E,\Pi)}.
    \end{equation*}
Then, it follows from Lemmas \ref{lema:PoissZ2} and \ref{lema:HamZ2} that we have a linear mapping
    \begin{equation*}
        \varrho: \mathscr{H}_{S}^{1}(E,\Pi) \longrightarrow \mathscr{H}^{1}_{\mathrm{rest}}(\PP_{\mathrm{aff}}),
    \end{equation*}
induced by the linearizing mapping \ec{[Z] \mapsto [Z^{(2)}]}. Here, \ec{\mathscr{H}^{1}_{\mathrm{rest}}(\PP_{\mathrm{aff}})} is the restricted first Poisson cohomology of the admissible Poisson algebra \ec{\PP_{\mathrm{aff}}}, defined in \eqref{eq:H1Rest}.

Moreover, if we also consider the subalgebra of all Poisson vector fields in \ec{\mathrm{Poiss}_{S}(E,\Pi)} whose restrictions to $S$ are Hamiltonian relative to $\psi$,
    \begin{equation*}
        \mathrm{Poiss}_{S}'(E,\Pi):= \left\{ Z \in \mathrm{Poiss}_{S}(E,\Pi) \mid Z|_{S}=\mathbf{i}_{\dd f}\psi,\ f \in \Cinf{S} \right\},
    \end{equation*}
we get a linear mapping
        \begin{equation*}
            \varrho': (\mathscr{H}_{S}^{1})'(E,\Pi) := \frac{\mathrm{Poiss}_{S}'(E,\Pi)}{\mathrm{Ham}(E,\Pi)} \longrightarrow (\mathscr{H}^{1}_{\mathrm{rest}})'(\PP_{\mathrm{aff}});
        \end{equation*}
since clearly \ec{\mathrm{Ham}(E,\Pi) \subseteq \mathrm{Poiss}_{S}'(E,\Pi)}. Here, \ec{(\mathscr{H}^{1}_{\mathrm{rest}})'(\PP_{\mathrm{aff}})} is a ``restricted'' first Poisson cohomology of \ec{\PP_{\mathrm{aff}}} defined in \eqref{eq:HrestPrime}.

We remark that in the general case when we start with a Poisson manifold \ec{(M,\Psi)} and a regular Poisson submanifold \ec{S \subset M} with the normal bundle $E$, the above arguments can be applied as follows: fixing transversal $\Sigma$ in \eqref{EcLL} and choosing an exponential mapping \ec{\mathbf{e}:E \rightarrow M} such that \ec{\dd_{S} \mathbf{e}(E)= \Sigma}, we define an admissible Poisson structure on \ec{\Cinf{\mathrm{aff}}(E)} by \eqref{EcBracketAff}, where \ec{\Pi=\mathbf{e}^{\ast}\Psi}.

\paragraph{The Normal Bundle of a Poisson Submanifold.} As we will see now, the normal bundle of every (embedded) Poisson submanifold admits a Poisson triple \cite{RuGaVo20}.

Suppose that \ec{E=\T_{S}M/\T{S}} is the \emph{normal bundle} of an embedded Poisson submanifold \ec{(S,\psi)} of a Poisson manifold \ec{(M,\Psi)}. Fix a splitting
    \begin{equation}\label{EcLL}
        \T_{S}M = \T{S} \oplus \Sigma,
    \end{equation}
where \ec{\Sigma \subset \T_{S}M} is a subbundle complementary to $\T{S}$. Here, \ec{\T_{S}M = \T{M}|_{S}}.

\begin{theorem}\label{teo:PoissonSubm}
Every subbundle $\Sigma$ in \eqref{EcLL} induces a Poisson triple \ec{([\,,\,]_{\mathrm{fib}}, \D^{\Sigma}, \K^{\Sigma})} associated to \ec{E \to (S,\psi)} and the corresponding admissible Poisson algebra \ec{\PP^{\Sigma} = (\Cinf{\mathrm{aff}}(E),\cdot, \newline \{,\}^{\Sigma})} is defined by the formula \eqref{EcBracketAff}. Moreover, $\PP^{\Sigma}$ is independent of the choice of $\Sigma$ up to isomorphism.
\end{theorem}

The proof of this theorem can be found in \cite{RuGaVo20}, where $\PP^{\Sigma}$ is called an \emph{infinitesimal Poisson algebra} of the Poisson submanifold \ec{(S,\psi)}. Alternatively, the infinitesimal Poisson algebra can be defined as follows \cite{Mar12}: since $S$ is embedded, we have that $S$ is a Poisson submanifold of $M$ if and only if the vanishing ideal \ec{I(S) = \left\{f \in \Cinf{M} \mid f|_{S} = 0\right\}} is also an ideal in the Lie algebra \ec{(\Cinf{M},\{,\}_{M})}. So, the infinitesimal Poisson algebra $\PP^{\Sigma}$ is naturally identified with the quotient Poisson algebra \ec{\Cinf{M} / I^{2}(S)}.

\begin{examplex}\label{exm:GoodGeom}
Consider the Poisson manifold \ec{M = \mathbb{R}^{3}_{w} \times \mathbb{R}^3_{z}} equipped with the Lie--Poisson bracket
    \begin{equation*}
        \{w_{i},w_{j}\} = \epsilon^{ijk}w_{k},\quad
        \{w_{i},z_{a}\} = \epsilon^{iab}z_{b},\quad
        \{z_{a},z_{b}\} = \varepsilon^{2}\, \epsilon^{abi}w_{i}, \quad \varepsilon \in \mathbb{R};
    \end{equation*}
for \ec{i,j,k,a,b = 1,2,3}. Then for every fixed \ec{\varepsilon}, there exists the following $3$--dimensional Poisson submanifold of \ec{M}:
    \begin{equation*}
        S = \big\{ (w,z)\in M \mid \varepsilon w+z=0 \big\}.
    \end{equation*}
Fixing the transversal $\Sigma$ generated by \ec{\{\partial/\partial{z_{a}}\}}, we choose a tubular neighborhood $U$ of $S$ as
\ec{U = \mathbb{R}^{3}_{w} \times \mathbb{R}^3_{z}} which is equipped with (adapted) coordinates \ec{x = w} and $y = \varepsilon w+z$. Then, the Poisson bivector field $\psi$ on \ec{S}  is determined by \ec{\psi^{ij}(x) = \epsilon^{ijk}x_{k}} and the corresponding Poisson triple \ec{([\,,\,]_{U^{\ast}}, \D^{\Sigma}, \K^{\Sigma})} is given by \ec{\lambda^{ab}_c=2 \varepsilon \epsilon^{abc}}, \ec{\D^{ia}_{b}=\epsilon^{iab}} and \ec{\K^{ij}_{a}=0}. In this case, the contravariant connection \ec{\D^{\Sigma}} is flat. From \eqref{EcBracketAff} it follows that the corresponding admissible Poisson structure on \ec{\Cinf{\mathrm{aff}}(U)} is given by
    \begin{equation*}
        \{f \oplus \eta, g\oplus\xi\}^{\Sigma} = \{f,g\}_{\mathfrak{so}(3)} \oplus
        \Big[ \epsilon^{iab} \big( \xi_{a}\tfrac{\partial f}{\partial x_{i}}-\eta_{a}\tfrac{\partial g}{\partial x_{i}} \big)
        + \epsilon^{ijk} x_{k} \big( \tfrac{\partial f}{\partial x_{i}}\tfrac{\partial \xi_{b}}{\partial x_{j}} - \tfrac{\partial g}{\partial x_{i}}\tfrac{\partial \eta_{b}}{\partial x_{j}} \big) + 2\varepsilon \epsilon^{acb}\eta_{a}\xi_{c} \Big]\dd{y}_{b},
    \end{equation*}
Here, \ec{\eta= \eta_{a}\dd{y_{a}}} and \ec{\xi= \xi_{a}\dd{y_{a}}}. Hence, \ec{(\Cinf{\mathrm{aff}}(U), \{\,,\,\}^{\mathrm{aff}})} is an admissible Poisson algebra.
\end{examplex}

\begin{remark}
Another approach to infinitesimal Poisson algebras of a Poisson submanifold \ec{(S,\psi)} of \ec{(M,\Psi)} was developed in the paper \cite{FeMa22}. The authors introduce the notion of a \textit{first order jet of a Poisson structure} at $S$, defined as a class of bivector fields \ec{\Pi \in \Gamma\,\wedge^{2}\T{M}}
tangent to $S$, modulo bivector fields that vanish to second order along $S$, and which satisfy the Jacobi identity up to second order. Moreover, it was shown that the first order jet of a Poisson structure at $S$ is naturally related with the restricted cotangent Lie algebroid \ec{A=(\T_{S}^{\ast}M, [\,,\,]_{S}, \rho=\Pi^{\sharp}|_{S})}. On the other hand, in the present paper, fixing a transversal subbundle $\Sigma$, we start with the admissible Poisson algebra \ec{\PP^{\Sigma}} which by Theorem \ref{teo:correspondenceGPA-PT-LA} induces the Lie algebroid \ec{\tilde{A}=(\T^{\ast}S \oplus E^{\ast},\cSch{\,,\,}^{\Sigma},\mathrm{pr}_{1}\circ\psi^{\sharp})} associated with the Poisson triple \ec{([\,,\,]_{\mathrm{fib}},\D^{\Sigma},\K^{\Sigma})}. The natural identification \ec{\T_{S}^{\ast}M \simeq \T^{\ast}S\oplus E^{\ast}} gives a Lie algebroid isomorphism between $A$ and $\tilde{A}$. One can show that the Jacobi identity for $\Pi$ up to second order is just equivalent to the structure equations \eqref{EcPT1}--\eqref{EcPT3} for the corresponding Poisson triple.
\end{remark}

Finally, we apply general results of Section \ref{sec:FirstCohomology} to describe the first cohomology of the infinitesimal Poisson algebra $\PP_{\mathrm{aff}}$ of a symplectic leaf, studied also in \cite{KaVo98}.

\paragraph{Infinitesimal Poisson Cohomology of a Symplectic Leaf.} Suppose that $S$ is a symplectic leaf of the Poisson manifold \ec{(M,\Psi)} and $E$ its normal bundle. Then, the dual $E^{\ast}$ is a locally trivial bundle of Lie algebras with typical fiber $\mathfrak{g}$, called the isotropy of $S$ \cite{Duf05,Mac95}. Fixing a transversal $\Sigma$ in \eqref{EcLL}, by Theorem \ref{teo:PoissonSubm}, we have an admissible Poisson algebra \ec{\PP=\PP^{\Sigma}} associated to the Poisson triple \ec{([\,,\,]_{\mathrm{fib}}, \D=\D^{\Sigma}, \K=\K^{\Sigma})}. By the infinitesimal Poisson cohomology \cite{KaVo98} of $S$ we just mean the cohomology of $\PP^{\Sigma}$. This is an infinitesimal ingredient of the germ Poisson cohomology at the symplectic leaf $S$ \cite{VeVo18}. Here, we formulate some results on the computing of \ec{\mathscr{H}^{1}(\PP^{\Sigma})}.

Recall that the Lie algebra \ec{\mathrm{Der}(E^{\ast},[\,,\,]_{\mathrm{fib}})} of \emph{Lie derivations} of \ec{(E^{\ast},[\,,\,]_{\mathrm{fib}})} consists of all first order differential operators \ec{\LL:\Gamma{E^{\ast}} \to \Gamma{E^{\ast}}} satisfying: there exists a vector field \ec{\ell \in \X{S}} such that
    \begin{equation*}
        \LL(f\eta) = f\LL\eta + \ell(f)\eta,
    \end{equation*}
and
    \begin{equation}\label{eq:PreserEast}
        \LL[\eta,\xi]_{\mathrm{fib}}=[\LL\eta,\xi]_{\mathrm{fib}}+ [\eta,\LL\xi]_{\mathrm{fib}},
    \end{equation}
for all \ec{f \in \Cinf{S}} and \ec{\eta,\xi \in \Gamma{E^{\ast}}}.

We define the following Lie subalgebra of \ec{\mathrm{Der}(E^{\ast},[\,,\,]_{\mathrm{fib}})}:
    \begin{equation*}
        \mathrm{PDer}(E^{\ast},[\,,\,]_{\mathrm{fib}}) := \big\{ \LL \in \mathrm{Der}(E^{\ast},[\,,\,]_{\mathrm{fib}}) \mid \ell \in \mathrm{Poiss}(S,\psi) \big\}.
    \end{equation*}

\begin{theorem}\label{teo:H1Geom}
Suppose that the isotropy $\mathfrak{g}$ is a perfect Lie algebra,
    \begin{equation}\label{eq:gPerfect}
        \mathfrak{g}=[\mathfrak{g},\mathfrak{g}].
    \end{equation}
Then, the first cohomology of the infinitesimal Poisson algebra $\PP$ is of the form
    \begin{equation}\label{eq:GeomSplit}
        \mathscr{H}^{1}\big( \PP \big) \simeq
        {\mathrm{H}_{\partial_{\D}}^{1}\big( \Gamma^{\ast}_{\PP} \big)} \oplus \frac{\mathfrak{M}(\PP)}{\mathscr{C}(\PP) + \mathrm{Inn}(E^{\ast},[\,,\,]_{\mathrm{fib}})}.
    \end{equation}
Here, \ec{\mathrm{H}_{\partial_{\D}}^{1}(\Gamma^{\ast}_{\PP})} is the first cohomology group of the coboundary operator \ec{\partial_{\D}} defined in \eqref{eq:PartialBar} and
    \begin{itemize}
      \item \ec{\mathfrak{M}(\PP)} is the Lie subalgebra consisting of all Lie derivations \ec{\LL \in \mathrm{PDer}(E^{\ast},[\,,\,]_{\mathrm{fib}})} for which there exist \ec{\theta \in \X{S} \otimes \Gamma{E^{\ast}}} such that
            \begin{align*}
                \pmb{\big[}\D_{\dd{f}},\LL\pmb{\big]} + \D_{\dd{\ell(f)}} &= [\theta(f),\cdot]_{\mathrm{fib}}, \\
                \big( \mathscr{L} \circ \K \big)(\dd{f},\dd{g}) - \K \big( \dd \ell(f),\dd{g} \big) - \K \big( \dd{f},\dd \ell(g) \big) &= - \big( \dd_{\D}\theta \big)(f,g),
            \end{align*}
          for \ec{f,g \in \Cinf{S}}; where \ec{\dd_{\D}} is the contravariant differential \eqref{eq:dDbar} on \ec{\X{S} \otimes \Gamma{E^{\ast}}} induced by $\D$ and $\psi$;

      \item \ec{\mathscr{C}(\PP):= \{\D_{\dd{f}} \mid f \in \Cinf{S}\} \subseteq \mathrm{PDer}(E^{\ast},[\,,\,]_{\mathrm{fib}})} is the $\mathbb{R}$--submodule of Lie derivations induced by $\D$;

      \item \ec{\mathrm{Inn}(E^{\ast},[\,,\,]_{\mathrm{fib}}) = \{\,[\eta, \cdot]_{\mathrm{fib}}\mid \eta \in \Gamma{E^{\ast}}\}} is the Lie ideal of inner derivations of \ec{(E^{\ast},[\,,\,]_{\mathrm{fib}})}.
    \end{itemize}
\end{theorem}
\begin{proof}
Consider the Lie algebra \ec{(P_1=\Gamma {E^{\ast}},[\,,\,]_{1}=[\,,\,]_{\mathrm{fib}})} of smooth sections of \ec{E^{\ast}} equipped with the pointwise Lie bracket. Since \ec{E^{\ast}} is a locally trivial Lie bundle with typical fiber $\mathfrak{g}$, we have the subbundle \ec{[E^{\ast},E^{\ast}]_{\mathrm{fib}}} of \ec{E^{\ast}} and the following relation (see, for example, \cite{Gun11}):
    \begin{equation}\label{eq:EPerfect}
        \Gamma{([E^{\ast},E^{\ast}]_{\mathrm{fib}})} \simeq [\Gamma {E^{\ast}},\Gamma {E^{\ast}}]_{\mathrm{fib}}.
    \end{equation}
By the assumption \eqref{eq:gPerfect}, we have \ec{[E^{\ast},E^{\ast}]_{\mathrm{fib}}=E^{\ast}} and together with \eqref{eq:EPerfect} this implies that \ec{\Gamma {E^{\ast}}} is perfect. Hence, taking into account that in the geometric framework the short exact sequence \eqref{eq:exactsqc} splits, the theorem follows from Lemma \ref{lema:H1Hrest} and the observation \ref{KerJ0Symp} after Proposition \ref{prop:H1KerJ}.
\end{proof}

Let \ec{\mathscr{Z}_{\mathfrak{g}}:=Z(E^{\ast})} be the subbundle of \ec{E^{\ast}} whose fiber over \ec{x\in S} is just the center of the Lie algebra \ec{E^{\ast}_{x}}. It is clear that \ec{\mathscr{Z}_{\mathfrak{g}}} is a locally trivial vector bundle with the center of \ec{\mathfrak{g}} as typical fiber. We note that \ec{Z(\Gamma{E^{\ast}},[\,,\,]_{\mathrm{fib}}) \simeq \Gamma(\mathscr{Z}_{\mathfrak{g}})} \cite{Gun11}.  In particular, by Example \ref{exm:SympType}, if $\mathfrak{g}$ is centerless then
    \begin{equation*}
        \mathrm{Casim}(\PP) \simeq \R{}.
    \end{equation*}

We also observe that the first term in the splitting \eqref{eq:GeomSplit} can be represented as
    \begin{equation*}
        \mathrm{H}_{\partial_{\D}}^{1}(\Gamma^{\ast}_{\PP} )=\frac{\mathfrak{Z}\big( \PP \big)}{\big\{ \D\eta \mid \eta \in \Gamma(\mathscr{Z}_{\mathfrak{g}}) \big\}},
    \end{equation*}
where \ec{\mathfrak{Z}(\PP) := \{W \in \X{S} \otimes \Gamma (\mathscr{Z}_{\mathfrak{g}}) \mid W\{f,g\}_{\psi}=\D_{\dd f}W(g) - \D_{\dd g}W(f),\ \text{for all}\ f,g\in \Cinf{S} \}}. So, in other words, the first cohomology of the infinitesimal Poisson algebra $\PP$ is given by special classes of \ec{\Gamma(\mathscr{Z}_{\mathfrak{g}})}--valued vector fields on $S$ and Lie derivations of \ec{(E^{\ast},[\,,\,]_{\mathrm{fib}})}.

\begin{corollary}
If the isotropy $\mathfrak{g}$ is perfect and centerless, in particular, a semisimple Lie algebra, then
    \begin{equation}\label{eq:H1Geometric}
        \mathscr{H}^{1}\big( \PP \big) \simeq
        \frac{\mathfrak{M}(\PP)}{\mathscr{C}(\PP) + \mathrm{Inn}(E^{\ast},[\,,\,]_{\mathrm{fib}})}.
    \end{equation}
Additionally, if the first de Rham cohomology of $S$ is trivial, we have
    \begin{equation}\label{eq:H1GeometricDeRhamTrivial}
        \mathscr{H}^{1}\big( \PP \big) \simeq
        \frac{\mathfrak{M}_{0}(\PP)}{\mathscr{C}_{0}(\PP) +\mathrm{Inn}(E^{\ast},[\,,\,]_{\mathrm{fib}})},
    \end{equation}
where \ec{\mathfrak{M}_{0}(\PP)} is the Lie ideal of \ec{\mathfrak{M}(\PP)} consisting of all $\Cinf{S}$--linear morphisms \ec{\LL:\Gamma E^{\ast} \to \Gamma E^{\ast}} satisfying the condition \eqref{eq:PreserEast} and for which there exists \ec{\theta \in \X{S} \otimes \Gamma{E^{\ast}}} such that
    \begin{equation*}
        \pmb{\big[}\D_{\dd{f}},\LL\pmb{\big]} = [\theta(f),\cdot]_{\mathrm{fib}}\quad \text{and} \quad
        \big( \mathscr{L} \circ \K \big)(\dd{f},\dd{g}) = - \big( \dd_{\D}\theta \big)(f,g),
    \end{equation*}
for all \ec{f,g \in \Cinf{S}}, and \ec{\mathscr{C}_{0}(\PP)=\{\D_{\dd_{k}}\mid k\in\mathrm{Casim}(S,\psi)\}}.
\end{corollary}
\begin{proof}
Taking into account Theorem \ref{teo:H1Geom}, relation \eqref{eq:H1Geometric} follows from Theorem \ref{teo:H1MCI} and relation \eqref{eq:H1GeometricDeRhamTrivial} follows from triviality of the first Poisson cohomology of the symplectic leaf $S$ and Theorem \ref{teo:H1M0}.
\end{proof}

\end{document}